\documentclass[ijoc,nonblindrev]{informs3} 
\def\theARTICLETOP{{\vskip10pt}}

\usepackage{tikz,pgfplots}
\pgfplotsset{compat=1.17}
\usetikzlibrary{plotmarks,shapes,decorations,arrows,patterns,petri}

\OneAndAHalfSpacedXI 

\usepackage{natbib}
 \bibpunct[, ]{(}{)}{,}{a}{}{,}%
 \def\bibfont{\small}%
 \def\BIBand{and}%

\pgfmathdeclarefunction{gauss}{2}{%
	\pgfmathparse{1/(#2*sqrt(2*pi))*exp(-((x-#1)^2)/(2*#2^2))}%
}

\usepackage{url}
\usepackage{xcolor}
\usepackage{float}
\usepackage{subfigure}
\usepackage{algorithmic}
\usepackage{algorithm}
\usepackage{dsfont}
\usepackage{verbatim}

\usepackage{multirow}

\TheoremsNumberedThrough     
\ECRepeatTheorems

\EquationsNumberedThrough    

\renewenvironment{proof}{{\sc Proof.}}{\hspace*{\fill}$\Box$}
\usepackage{amsmath}

\newcommand{\bY}{{\mathbf Y}}
\newcommand{\bU}{{\mathbf U}}

\newcommand{\bZ}{{\mathbf Z}}

\newcommand{\by}{{\mathbf y}}
\newcommand{\bu}{{\mathbf u}}
\newcommand{\bzero}{{\mathbf 0}}
\newcommand{\bone}{{\mathbf 1}}

\newcommand{\cA}{{\mathcal A}}

\newcommand{\cG}{{\mathcal G}}
\newcommand{\cL}{{\mathcal L}}
\newcommand{\cN}{{\mathcal N}}
\newcommand{\cO}{{\mathcal O}}
\newcommand{\cS}{{\mathcal S}}
\newcommand{\cU}{{\mathcal U}}

\newcommand{\frakv}{{\mathfrak{v}}}

\newcommand{\EE}{{\mathbb E}}
\newcommand{\II}{{\mathbb I}}
\newcommand{\PP}{{\mathbb P}}
\newcommand{\RR}{{\mathbb R}}

\def\eqdef {\buildrel \rm def \over =}    
\def\q{$\kern1.4em$}     

\DeclareMathOperator{\var}{Var}
\DeclareMathOperator{\Cov}{Cov}
\DeclareMathOperator{\Var}{Var}

\DeclareMathOperator{\tr}{tr}

\DeclareMathOperator{\cde}{cde}
\DeclareMathOperator{\cmc}{cmc}

\DeclareMathOperator{\cderqmc}{cde-rqmc}

\def\vol{{\rm vol}}

\def\Var{{\rm Var}}
\def\Cov{{\rm Cov}}

\def\MISE{{\rm MISE}}

\def\ISB{{\rm ISB}}

\def\IV{{\rm IV}}

\def\IC{{\rm IC}}
\def\tr{{\sf t}}
\def\d{{\rm d}}

\pgfplotscreateplotcyclelist{defaultcolorlist}{%
	{blue!85!black,line width=0.9pt,mark=*,solid},
	{red!85!black,line width=0.9pt,mark=square*,solid},
	{lime!80!black,line width=0.9pt,mark=triangle*,solid},
	{orange!90!black,line width=0.9pt,mark=diamond*,solid},
	{violet!85!black,mark=o,solid},
	{brown!85!black,mark=square,dashed},
	{purple!85!black,mark=triangle,dotted}}
\pgfplotscreateplotcyclelist{comparison}{%
	{blue!85!black,line width=0.9pt,mark=*,solid},
	{blue!85!black,line width=0.9pt,mark=diamond*,dashed},
	{red!85!black,line width=0.9pt,mark=*,solid},
	{red!85!black,line width=0.9pt,mark=diamond*,dashed},
	{lime!85!black,line width=0.9pt,mark=*,solid},
	{lime!85!black,line width=0.9pt,mark=diamond*,dashed},
}
	\pgfplotscreateplotcyclelist{dens}{%
	{blue!85!black,line width=1.5pt,mark=none,solid},
	{blue!85!black,line width=0.9pt,mark=o,dashed},
  {blue!85!black,line width=0.9pt,mark=+,dotted},
  {red!85!black,line width=0.9pt,mark=*,solid},
  {red!85!black,line width=0.9pt,mark=diamond*,dashed},
  {lime!85!black,line width=0.9pt,mark=*,solid},
  {lime!85!black,line width=0.9pt,mark=diamond*,dashed},
	}
	
	\pgfplotscreateplotcyclelist{canti}{%
		{blue},
		{blue},
		{blue},
		{blue},
		{blue},
		{orange!60,line width = 1.0},          
		{black,line width = 1.0}}

\newif\ifnotes\notestrue

\def\mpierre#1{{\color{red} #1}}

\def\mpierre#1{#1}

\def\hpierre#1{}

\def\hflorian#1{}

\def\hamal#1{}
\def\perhaps#1{{\color{gray} #1}}

\def\new#1{{#1}}

\newif\iflong
\longfalse

\newif\ifplots\plotstrue
\plotsfalse

\begin{document}

\RUNAUTHOR{L'Ecuyer, Puchhammer, Ben Abdellah}

\RUNTITLE{MC and QMC Density Estimation via Conditioning}

\TITLE{Monte Carlo and Quasi-Monte Carlo Density Estimation via Conditioning}

\ARTICLEAUTHORS{%



 \AUTHOR{Pierre L'Ecuyer}
 \AFF{D\'{e}partement d'Informatique et de Recherche Op\'{e}rationnelle, 
 Pavillon Aisenstadt, Universit\'{e} de Montr\'{e}al, C.P. 6128,
 Succ. Centre-Ville, Montr\'{e}al, Qu\'{e}bec, Canada H3C 3J7, \EMAIL{lecuyer@iro.umontreal.ca}}
 
  \AUTHOR{Florian Puchhammer}
 \AFF{ Basque Center for Applied Mathematics, Alameda de Mazarredo 14, 48009 Bilbao, Basque Country, Spain; and 
  D\'{e}partement d'Informatique et de Recherche Op\'{e}rationnelle, 
  Universit\'{e} de Montr\'{e}al,
 \EMAIL{fpuchhammer@bcamath.org}}
 
  \AUTHOR{Amal Ben Abdellah}
 \AFF{D\'{e}partement d'Informatique et de Recherche Op\'{e}rationnelle, 
 Pavillon Aisenstadt, Universit\'{e} de Montr\'{e}al, C.P. 6128,
 Succ. Centre-Ville, Montr\'{e}al, Qu\'{e}bec, Canada H3C 3J7, \EMAIL{amal.ben.abdellah@umontreal.ca}}

}

\ABSTRACT{%
Estimating the unknown density from which a given independent sample originates is more difficult than estimating the mean, in the sense that for the best popular non-parametric density estimators, the mean integrated square error converges more slowly than at the canonical rate of $\cO(1/n)$. When the sample is generated from a simulation model and we have control over how this is done, we can do better. We examine an approach in which conditional Monte Carlo yields, under certain conditions, a random conditional density which is an unbiased estimator of the true density at any point. By averaging independent replications, we obtain a density estimator that converges at a faster rate than the usual ones. Moreover, combining this new type of estimator with randomized quasi-Monte Carlo to generate the samples typically brings a larger improvement on the error and convergence rate than for the usual estimators, because the new estimator is smoother as a function of the underlying uniform random numbers.}


\maketitle


%

\section{Introduction}

Simulation is commonly used to generate $n$ realizations of a random variable $X$
that may represent a payoff, a cost, or a performance of some kind, and then to estimate 
from this sample the unknown expectation of $X$ together with a confidence interval
on this expectation \citep{sASM07a,sLAW14a}. 
Simulation books focus primarily on how to improve the quality of
the estimator of $\EE[X]$ and of the confidence interval.
Estimating a given quantile of the distribution of $X$, or the sensitivity of
$\EE[X]$ with respect to some parameter in the model, also with a confidence interval,
are other well-studied topics in the literature. 

However, large simulation experiments can provide a lot more information than just point 
estimates with confidence intervals.
Running simulations of a complex system for hours, with thousands of runs, 
only to report confidence intervals on a few single numbers is poor data valorization.
A simulation experiment can give much more useful information than this.
In particular, it can provide an estimate of the entire distribution of $X$, 
and not only its expectation or a specific quantile.  
Moreover, and perhaps more importantly, users are typically more interested in the whole distribution 
than on a confidence interval on the mean. The following examples of typical 
simulation models show why.

In many real-life stochastic simulation models, the prime focus of interested 
is the distribution of certain random \emph{delays}.
These delays can be for example the waiting times of calls in a telephone call center,
the waiting times of patients at a walk-in medical clinic or at the emergency, 
the waiting time of passengers at an airport checking counter,
the delivery time of an order, the travel time in some transportation network, etc.
In all these situations, a user will be mostly interested in 
\emph{what his own waiting time is likely to be}.  
That is, she/he is much more interested in the \emph{probability distribution}
of the waiting time 
than in having a good estimate of the expectation (or global mean) \citep{sNEL08a,sSMI15a}.

When one makes a call to a call center and all agents are busy, a good forecast of the 
waiting time is certainly appreciated. Based on this forecast, the caller may decide 
that there is enough time to engage in another activity before getting an answer.
The expected waiting time alone is not sufficient to make such as decision,
because for example it does not tell the probability of missing the call when 
going out for $x$ minutes.
A distributional forecast, which provides a density of the waiting time distribution
(perhaps conditional on the current time and system state) is much more informative and helpful
\citep{ccTHI21a}. This applies to waiting times in many other types of service systems. 
Users are interested in the density (or distribution) of their waiting time, not just the expectation. 
When a manufacturer orders parts from a supplier, or a retail store orders items from 
the manufacture or distributor, an estimate of the density of the time until delivery 
(and not just its expectation) gives them good idea of what can happen,
including the probability that the parts or items arrive on time, 
and the distribution of the delay if there is one.
In a large construction project that involves many activities of random durations and precedence 
constraints, the total time to complete the project is a random variable $X$ usually modeled
by a stochastic activity network (see Section~\ref{sec:san}).
Knowing the density of $X$ permits one to assess the risks in signing contracts that 
impose various types of penalties when $X$ is too large. 
%
In many other situations, $X$ is a cost or a profit and estimating the  
density of $X$ is again more interesting and useful than just the expectation.
In a finance application, for example, $X$ may represent an investment loss 
over a given month, and the density of $X$ provides much more information on the 
possibilities of large losses (and perhaps bankruptcy) than just the mean.
We will report numerical experiments for a finance-type example in 
Section~\ref{sec:function-multinormal}.

So, simulation users are interested in the whole distribution of the output 
and not only the mean.
One way to visualize the \emph{entire distribution of $X$} is to look at the empirical 
\emph{cumulative distribution function} (cdf) of the observations. 
But density estimators (including histograms) are preferred because they give a better 
visual insight on the distribution than the cdf.
For this reason, leading simulation software routinely provides histograms and enhanced boxplots 
that give a rough idea of the distribution of the output random variables of interest.  
For example, the standard output display in Simio$^{\mbox{\footnotesize\sc tm}}$ is a histogram enhanced 
with a boxplot, named SMORE (Simio Measure Of Risk and Error) \citep{sSTU10a,sSMI15a}.
When $X$ has a continuous distribution, these histograms and boxplots are 
in fact just primitive forms of density estimators.  
So why are better density estimators not routinely offered?
Mainly because non-parametric density estimation is difficult.  


If the density of $X$ is \emph{assumed} to have a known \emph{parametric form}, e.g., a normal or
gamma distribution, then one can estimate the parameters from data in the usual way 
(e.g., by maximum likelihood) and things are simple.
But in typical complex models, $X$ does not have a known and simple form of distribution.
There are \emph{semi-parametric} procedures in which the density is assumed to belong to a Hilbert space
of functions which are linear combinations a finite number of fixed basis functions,
and the coefficients are estimated by penalized regression.
These are known as smoothing spline models \citep{tGU93a,tYU20a}.   
But it is often difficult to select basis functions that capture the unknown density
and a good choice depends on the problem.
In this paper, we focus on \emph{non-parametric} methods, in the sense that we assume 
no particular form for the density of $X$.  On the other hand, the input distributions 
in the simulation model may be parametric (they often are).  
The most widely used non-parametric density estimation methods are the \emph{histogram} 
and the \emph{kernel density estimator} (KDE) \citep{tPAR62a,tSIL86a,tWAN95a,tSCO15a}.
Given $n$ independent realizations of $X$, the mean integrated square error (MISE) between the
true density and a histogram with optimally selected divisions converges only as $\cO(n^{-2/3})$.  
With the KDE, the MISE converges as $\cO(n^{-4/5})$ in the best case.
These rates are slower than the canonical $\cO(n^{-1})$ rate for the variance 
of the sample average as an unbiased estimator of the mean.
The slower rates 
stem from the presence of bias.  
See for example \cite{tSCO15a} for the details. 
For a histogram, taking wider rectangles reduces the variance but increases
the bias by flattening out the short-range density variations. A compromise must be made to 
minimize the MISE. The same happens with the KDE, with the rectangle width replaced 
by the bandwidth of the kernel.
Selecting a good bandwidth for the KDE is particularly difficult.
The bandwidth should ideally vary over the interval in which we estimate the density;
it should be smaller where the density is larger and/or smoother, and vice-versa.
This is complicated to implement.
Handling discontinuities in the density is also problematic.
These difficulties have discouraged the use of KDEs (instead of histograms)
to report simulation results in general-purpose software.

The KDE and other related density estimation methods were developed mainly for the situation 
where $n$ independent realizations of $X$ are given and nothing else is known,
as traditionally assumed in classical non-parametric statistics,
and one wishes to estimate the density from them \citep{tSCO15a}.
But in a Monte Carlo setting in which the $n$ observations are generated by simulation,
there are opportunities to do better by controlling the way we generate the realizations and by
exploiting the fact that we know the underlying stochastic model.
This is the subject of the present paper.  

Our approach combines two general ideas.
The first one is to build a smooth estimator of the cdf via 
conditional Monte Carlo (CMC), and take the corresponding conditional density 
to estimate the unknown density.  We call it a \emph{conditional density estimator} (CDE).
Under appropriate conditions, the CDE is unbiased and has uniformly-bounded variance,
so its MISE is $\cO(n^{-1})$ for $n$ samples.
This idea of using CMC was mentioned by \cite{sASM07a}, page 146, Example 4.3, 
and further studied in \cite{vASM18a}, but only for the special case of estimating 
the density of a sum of i.i.d.{} continuous random variables having a known density.
\cite{vASM18a} simply ``hides'' the last term of the sum, meaning that the last random variable 
is not generated, and he takes a shifted version of the known density of this last variable 
to estimate the density, the value at risk, and the conditional value at risk of the sum.
His setting is equivalent to a sum of two independent random variables: 
the first one is the partial sum which is generated and on which we condition, 
and the second one is the last variable which is not generated.
\new{\cite{oFU06b} mentioned this same idea in one of his examples.}

Smoothing by CMC before taking a stochastic derivative has been studied earlier
for estimating the derivative of an expectation 
\citep{oGON87a,oLEC94a,vFU97a} and the derivative of a quantile \citep{vFU09a} with respect 
to a \emph{model parameter}.  This is known as \emph{smoothed perturbation analysis} (SPA).
In retrospect, one can say that the CDE at a given point $x$ is an SPA estimator
obtained by viewing the cdf $F(x)$ as the expectation and $x$ as the model parameter.
\new{However, nobody studied this idea for density estimation until \cite{vASM18a} 
did it for his special case.}  

The main contribution of this paper is to show how this CDE approach can be used to 
estimate the density in a much more general setting than \cite{vASM18a},
to give conditions under which it provides an unbiased density estimator,
and to examine how effective it is via experiments on several types of examples.
In most of these examples, $X$ is \emph{not} defined as a sum of random variables,
and we often have to hide more than just one random variable to do the conditioning.
A key unbiasedness condition is that the conditional cdf must be a continuous function 
of the point $x$ at which we estimate the density. 
In other words, the conditional distribution of $X$ under the selected conditioning
must have a density with respect to the Lebesgue measure.
The variance of the density estimator may depend strongly on which variables we hide,
i.e., on what we are conditioning.  We illustrate this with several examples
and we provide guidelines for the choice of conditioning.
Interestingly, while the KDE is defined as an average of $n$ randomly-shifted copies of
the (fixed) kernel density, the CDE is an average of $n$ conditional densities which are
generally different and random.  

In addition to being unbiased, the CDE often has less variation than the KDE as a function 
of the underlying uniform random numbers.  As a result, its combination with 
\emph{randomized quasi-Monte Carlo} (RQMC) tends to bring much more improvement than for the KDE.
We have observed this in all our experiments.
Under appropriate conditions, it can be proved that combining the CDE with RQMC provides
a density estimator whose MISE converges at a faster rate than $\cO(n^{-1})$,
for instance  $\cO(n^{-2+\epsilon})$ for any $\epsilon > 0$ in some situations.
We observe this fast rate empirically on numerical examples.
This happens essentially when the CDE is a smooth function of the underlying uniforms.
To our knowledge, this type of convergence rate has never been proved or observed for 
non-parametric density estimation.

\hpierre{The question of how to use RQMC for density estimation was raised in a discussion 
  between A.~B.~Owen, F.~Hickernell, and P.~L'Ecuyer
  at a workshop on High-Dimensional Numerical Problems in Banff, in September 2015.}

The combination of RQMC with an ordinary KDE was studied by \cite{vBEN21a},
who were able to prove a faster rate than $\cO(n^{-4/5})$ for the MISE when the RQMC points
have a small number of dimensions.  They observed this faster rate empirically on examples.
They also showed that the MISE reduction from RQMC degrades rapidly when the bandwidth 
is reduced (to reduce the bias) or when the dimension increases.
The CDE+RQMC approach studied in the present paper avoids this problem (there is no bias 
and no bandwidth) and is generally much more effective than the KDE+RQMC combination.
We provide numerical comparisons in our examples.

Other Monte Carlo density estimators were proposed very recently, also based on the idea of
estimating the derivative of the cdf, but using a likelihood ratio (LR) method instead. 
The LR method was originally designed to estimate the derivative of the expectation 
with respect to parameters of the distribution of the underlying input random variables 
\citep{oGLY87a,oLEC90a}.
\cite{vLAU19a} proposed an estimator that combines a clever change of variable with the LR
method, to estimate the density of a sum of random variables as in \cite{vASM18a},
but in a setting where the random variables can be dependent.
\cite{oPEN18a} proposed a generalized version of the LR gradient estimator method, named GLR,
to estimate the derivative of an expectation with respect to a more general model parameter.
\cite{vLEI18a} sketched out how GLR could be used to estimate a density.
Formulas for these GLR density estimators are given in Theorem 1 of \cite{oPEN20a}. 
We compare them with the CDE estimators in our numerical illustrations.

Density estimation has other applications than just visualizing the distribution
of an output random variable \citep{tVAN00a,tSCO15a}.
For instance when computing a confidence interval for a quantile using the 
central-limit theorem (CLT), 
one needs a density estimator at the quantile to estimate the variance 
\citep{tSER80a,sASM07a,tNAK14a,tNAK14b}. 
See Section~\ref{sec:quantile} in the Supplement.
Another application is for maximum likelihood estimation when the likelihood does not have
a closed-form expression, so to maximize it with respect to some parameter $\theta$, 
the likelihood function (which in the continuous case is a density at any value of $\theta$)
must be estimated \citep{tVAN00a,oPEN20a}.
A related application is the estimation of the posterior density of $\theta$ given some data,
in a Bayesian model \citep{tEFR16a}.

The remainder is organized as follows.
In Section~\ref{sec:model}, we define our general setting, 
recall key facts about density estimators,
introduce the general CDEs considered in this paper, 
prove some of their properties, and give small examples to provide insight on the key ideas.
We also briefly recall GLR density estimators.
%
In Section~\ref{sec:rqmc}, we explain how to combine the CDE with RQMC
and discuss the convergence properties for this combination.
Section~\ref{sec:experiments} reports experimental results with various examples.
Some of the examples feature creative ways of conditioning to improve
the effectiveness of the method.  Additional examples are examined in the Online Supplement.
\new{Section~\ref{sec:guidelines} summarizes the key issues and guidelines on 
the construction and applications of the CDE.}
A conclusion is given in Section~\ref{sec:conclusion}.
The main ideas of this paper were presented at a SAMSI workshop on QMC methods
in North Carolina, and at a RICAM workshop 
in Linz, Austria, both in 2018.

\section{Model and conditional density estimator}
\label{sec:model}

\subsection{Density estimation setting}

We have a real-valued random variable $X$ that can be simulated from its 
exact distribution, but we do not know the cdf $F$ and density $f$ of $X$.
Typically, $X$ will be an easily computable function of several other 
random variables with known densities.
Our goal is to estimate $f$ over a finite interval $[a,b]$.
Let $\hat f_n$ denote an estimator of $f$ based on a sample of size $n$. 
We measure the quality of $\hat f_n$
by the \emph{mean integrated square error} (MISE), defined as
\begin{equation}
\label{eq:mise}
  \MISE = \MISE(\hat f_n) = \int_a^b \EE [(\hat f_n(x) - f(x))^2] \d x.
\end{equation}
The MISE is the sum of the 
\emph{integrated variance} (IV) and the \emph{integrated square bias} (ISB):
\[
  \MISE = \IV + \ISB 
	      = \int_a^b \EE (\hat f_n(x) - \EE [\hat f_n(x)])^2 \d x
	      + \int_a^b (\EE [\hat f_n(x)] - f(x))^2 \d x.
\]
\hpierre{If we also want to take into account the computing cost of the estimator, we can
use the \emph{work-normalized MISE}, defined as the MISE multiplied by the expected computing cost.}%
A standard way of constructing $\hat f_n$ when $X_1,\dots,X_n$ are $n$ independent realizations 
of $X$ is via a KDE, defined as follows \citep{tPAR62a,tSCO15a}:
\[
 \hat f_n(x) = \frac{1}{nh} \sum_{i=1}^n  k\left(\frac{x-X_i}{h}\right),
   \label{eq:kde}
\]
where the \emph{kernel} $k$ is a probability density over $\RR$,
usually symmetric about 0 and non-increasing over $[0,\infty)$, 
and the constant $h > 0$ is the \emph{bandwidth}, whose role is to stretch [or compress]
the kernel horizontally to smooth out [or unsmooth] the estimator $\hat f_n$.
The KDE was developed for the setting in which $X_1,\dots,X_n$ are given a priori, 
and it is the most popular estimator for this situation.
It can be used as well when $X_1,\dots,X_n$ 
are independent observations produced by simulation from a generative model,
but then there is an opportunity to do better, as we now explain.

\subsection{Conditioning and the stochastic derivative as an unbiased density estimator}

Since the density of $X$ is the derivative of its cdf, $f(x) = F'(x)$, a natural idea
would be to take the derivative of an estimator of the cdf as a density estimator.
The simplest candidate for a cdf estimator is the \emph{empirical cdf} 
\[
  {\hat F_n(x)} = \frac{1}{n} \sum_{i=1}^n \II[X_i\le x],
\]
but $\d {\hat F_n(x)} /\d x = 0$ almost everywhere, 
so this one \emph{cannot} be a useful density estimator.
Here, $\hat F_n(x)$ is an unbiased estimator of $F(x)$ at each $x$,
but its derivative is a biased estimator of $F'(x)$.  That is,
because of the discontinuity of $\hat F_n$, we cannot exchange the derivative and expectation:
\[
  0 \;=\; \EE\left[\frac{\d {\hat F_n(x)}}{\d x}\right]  \;\not=\;  \frac{\d \EE[\hat F_n(x)]}{\d x} \;=\; F'(x). 
\]

A general framework to construct a continuous estimator of $F$ via CMC is the following.
Replace the indicator $\II[X\le x]$ by its \emph{conditional cdf} given 
filtered (reduced) information ${\cG}$:
$
  {F(x \mid \cG)} \;\eqdef\; \PP[X \leq x \mid \mathcal{G}],
$
where ${\cG}$ is a sigma-field that contains not enough information to reveal $X$ 
but enough to compute $F(x \mid \cG)$.
Here, knowing the realization of $\cG$ means knowing the realizations of all $\cG$-measurable 
random variables.
Our CDE to estimate $f(x)$ will be the \emph{conditional density} 
$f(x\mid \cG) \eqdef F'(x \mid \cG) = \d F(x \mid \cG) /\d x$, when it exists.
\new{We assume that this estimator can be computed (or approximated) for (almost) all realizations of $\cG$.}  
Under the following assumption, we prove that $f(x \mid \cG)$ exists almost surely and
is an unbiased estimator of $f(x)$ whose variance is bounded uniformly in $x$.
Since $F(\cdot\mid \cG)$ cannot decrease,  $f(\cdot\mid \cG)$ is never negative.

\begin{assumption}
\label{ass:cmc-dct}
For all realizations of $\cG$, $F(x \mid \cG)$ is a continuous function of $x$ over the interval $[a,b]$,
and is differentiable except perhaps at a countable set of points $D(\cG) \subset [a,b]$.
For all $x\in [a,b]$, $F(x \mid \cG)$ is differentiable at $x$ w.p.1.
There is also a random variable ${\Gamma}$ defined over the same probability space as $F(x \mid \cG)$,
such that $\EE[\Gamma^2] \le K_\gamma$ for some constant $K_\gamma < \infty$, and for which
 $\;\sup_{x\in [a,b]\setminus D(\cG)} F'(x \mid \cG) \le \Gamma.$
\end{assumption}

\begin{proposition}
\label{th:cmc-dct}
Under Assumption~\ref{ass:cmc-dct}, 
$\EE[f(x \mid \cG)] = f(x)$ 
and $\Var[f(x \mid \cG)] \le K_\gamma$ for all $x\in [a,b]$.
\end{proposition}

\begin{proof}
We adapt the proof of Theorem 1 of \cite{oLEC90a}.
By Theorem 8.5.3 of \cite{mDIE69a}, 
which is a form of mean value inequality theorem for non-differentiable functions,
for every $x\in [a,b]$ and $\delta > 0$, with probability 1, we have 
\[
  0 \;\le\; \frac{\Delta(x,\delta,\cG)}{\delta} 
	\;\eqdef\; \frac{ F(x+\delta \mid \cG) - F(x \mid \cG)}{\delta}
	\;\le\;  \sup_{y\in [x,x+\delta] \setminus D(\cG)}  F'(y \mid \cG)
	\;\le\;  \Gamma.
\]
Then, by the dominated convergence theorem,
\[
  \EE\left[\lim_{\delta\to 0} \frac{\Delta(x,\delta,\cG)}{\delta}\right] 
     = \lim_{\delta\to 0} \EE\left[\frac{\Delta(x,\delta,\cG)}{\delta}\right], 
\]
which shows the unbiasedness.
Moreover, $\Var[f(x \mid \cG)] = \Var[F'(x \mid \cG)] \le \EE[\Gamma^2] \le K_\gamma$.
\end{proof}

Suppose now that $\cG^{(1)},\dots,\cG^{(n)}$ are $n$ independent realizations of $\cG$,
so $F(x \mid \cG^{(1)}),\dots, F(x \mid \cG^{(n)})$ are independent realizations of $F(x \mid \cG)$,
and consider the CDE
\begin{equation}
\label{eq:hatf-cde}
  \hat f_{\cde,n}(x) = \frac{1}{n} \sum_{i=1}^n f(x \mid \cG^{(i)}).
\end{equation}
Under Assumption~\ref{ass:cmc-dct}, it follows from Proposition 1 that $\ISB(\hat f_{{\cde},n}) = 0$ 
and $\MISE(\hat f_{{\cde},n}) = \IV(\hat f_{{\cde},n}) \le (b-a) K_\gamma / n$.
An unbiased estimator of this IV is given by 
\begin{equation}
\label{eq:hat-iv}
  \widehat{\IV} = \widehat{\IV}(\hat f_{{\cde},n}) 
	  = {\frac{1}{n-1}} \int_a^b \sum_{i=1}^n 
		      \left[f(x \mid \cG^{(i)}) - \hat f_{\cde,n}(x) \right]^2 \d x.
\end{equation}
In practice, this integral can be approximated by evaluating the integrand at a finite
number of points over $[a,b]$ and taking the average, multiplied by $(b-a)$. 

The variance of the CDE estimator at $x$ is $\Var[f(x\mid \cG)]$,
where $x$ is fixed and $\cG$ is random.
This differs from the variance associated with the conditional density
$f(\cdot\mid \cG)$, which is $\Var[X\mid \cG]$.
It is well known that in general, when estimating $\EE[X]$, a CMC estimator never has a larger variance
than $X$ itself, and the more information we hide, the smaller the variance.
That is, if $\cG \subset \tilde\cG$ are two sigma-fields such that $\cG$ contains only a subset
of the information of $\tilde\cG$, then 
\begin{equation}
\label{eq:var-cmc}
  \Var[\EE[X\mid \cG]] \le \Var[\EE[X\mid \tilde\cG]] \le \Var[X].
\end{equation}
Noting that $F(x \mid \cG) = \EE[\II[X\le x] \mid \cG]$, we also have
\[
  \Var[F(x\mid \cG)] \le \Var[F(x\mid \tilde\cG)] \le \Var[\II[X\le x]] = F(x)(1-F(x)).
\]
Thus, (\ref{eq:var-cmc}) applies as well to the (conditional) cdf estimator.
However, applying it to the CDE is less straightforward.
It is obviously not true that $\Var[F'(x\mid \cG)] \le \Var[\d\II[X\le x]/\d x]$ 
because the latter is zero almost everywhere.
Nevertheless, we can prove the following.

\begin{lemma}
\label{th:decreasing-variance}
If $\cG \subset \tilde\cG$ both satisfy Assumption~\ref{ass:cmc-dct}, 
then for all $x\in [a,b]$,  we have $\Var[f(x \mid \cG)] \le \Var[f(x \mid \tilde\cG)]$.
\end{lemma}

\begin{proof}
The result does not follow directly from (\ref{eq:var-cmc}) because $F'$ is not an expectation;
this is why our proof does a little detour.
For an arbitrary $x\in [a,b]$ and a small $\delta > 0$, define the random variable 
$I = I(x,\delta) = \II[x < X \le x+\delta]$.  We have
$\EE[I\mid \cG] = F(x+\delta \mid \cG) - F(x \mid \cG)$, 
as in the proof of Proposition~\ref{th:cmc-dct}, and similarly for $\tilde \cG$.
Using (\ref{eq:var-cmc}) with $I$ in place of $X$ gives 
\begin{equation}
\label{eq:decomp-cdf}
    \Var[\EE[I\mid \cG]] \le \Var[\EE[I\mid \tilde\cG]].
\end{equation}
We have
\[
  f(x \mid \cG) =  \lim_{\delta\to 0} \frac{F(x+\delta \mid \cG) - F(x \mid \cG)}{\delta}
	    = \lim_{\delta\to 0} \EE[I(x,\delta)/\delta \mid \cG]  
\]
and similarly for $\tilde\cG$.  Combining this with (\ref{eq:decomp-cdf}), we obtain 
\begin{eqnarray*}
  \Var[f(x \mid \cG)]   
	  &=&  \Var[\lim_{\delta\to 0} \EE[I(x,\delta)/\delta \mid \cG]] 
	  ~=~  \lim_{\delta\to 0} \Var[\EE[I(x,\delta)/\delta \mid \cG]] \\
	 &\le& \lim_{\delta\to 0} \Var[\EE[I(x,\delta)/\delta \mid \tilde\cG]] 
	  ~=~  \Var[\lim_{\delta\to 0} \EE[I(x,\delta)/\delta \mid \tilde\cG]] 
    ~=~  \Var[f(x \mid \tilde\cG)], 
\end{eqnarray*}
in which the exchange of ``\Var'' with the limit 
(at two places) can be justified by a similar argument as in Proposition~\ref{th:cmc-dct}.
More specifically, we need to apply the dominated convergence theorem to $\EE[I(x,\delta)/\delta \mid \cG]$,
which is just the same as in Proposition~\ref{th:cmc-dct}, and also to its square, which is also valid
because the square is bounded uniformly by $\Gamma^2$.
This completes the proof.
\end{proof}

This lemma tells us that conditioning on less information (hiding more) always reduces
the variance of the CDE (or keep it the same).  
But if we hide more, the CDE may be harder or more costly to compute,
so a compromise must be made to minimize the work-normalized MISE
(which is the MISE multiplied by the expected time to compute the estimator),
and the best compromise is generally problem-dependent.
When none of $\cG$ or $\tilde\cG$ is a subset of the other, the variances of the 
corresponding conditional density estimators may differ significantly, and
Lemma~\ref{th:decreasing-variance} does not apply, so other \new{strategies} must be used
to select $\cG$ when there are multiple possibilities.

In our setting, the most important condition is that $\cG$ must satisfy Assumption~1.
Any such $\cG$ provides an unbiased density estimator with finite variance.
When there are multiple choices, in general we want to choose $\cG$ so that the conditional density 
tends to be \emph{spread out} as opposed to being concentrated in a narrow peak.
We give concrete examples of this in Section~\ref{sec:experiments}.
This criterion is heuristic. 
If $f$ is very spiky itself, then the CDE must be spiky as well, 
because $\Var[X\mid \cG] \le \Var[X]$, and yet $\Var[f(x\mid \cG)]$ can be very small,
even zero in degenerate cases.
Also, a large $\Var[X\mid \cG]$ for all $\cG$ is not sufficient, because the large variance
may come from two or more separate spikes, and this is why we write ``spread out'' instead of
``large variance''.
Roughly, we want the CDE $f(\cdot\mid \cG)$ to be spread out relative to $f$, 
for all realizations of $\cG$.

A more elaborate selection criterion should take into account the IV of the CDE, 
its computing cost, and also the variation of
of the resulting CDE as a function of the underlying uniform random numbers, 
in case we want to use RQMC to generate those random numbers (see Section~\ref{sec:rqmc}).
For real-life models, it is usually much too hard to precompute such measures, so the best
practice would be to identify a few promising candidates and either:
(1) perform pilot runs to compare their effectiveness and select one or
(2) take a convex combination of the corresponding CDEs, 
as explained in Section~\ref{sec:combination}.
We believe that finding a good $\cG$ will always remain largely problem-dependent
and it sometimes requires creativity.  
\new{No simple selection method works universally.
On the other hand, to make good selections, it is useful to understand certain basic principles.}
We illustrate this with a variety of examples \new{in the next subsection and} in Section~\ref{sec:experiments}.
\hpierre{Here, I tried to explain a bit more how to select $\cG$.}

\hpierre{The variance of the CDE estimator at $x$ is $\Var[f(x\mid \cG)]$.
In this expression, $x$ is fixed and $\cG$ is random, and the variance depends on how much
the conditional density at $x$ changes when $\cG$ varies.  
We emphasize that this is \emph{not} the variance associated with the conditional density
$f(\cdot\mid \cG)$, and these two variances may not be strongly linked.} 
\hpierre{here is where we need more contribution!}

\subsection{Small examples to provide insight}
\label{sec:small-examples}

To illustrate some key ideas, this subsection provides simple examples formulated in the 
special setting in which $X = h(Y_1,\dots,Y_d)$ where $Y_1,\dots,Y_d$ are 
independent continuous random variables, each $Y_j$ has cdf $F_j$ and density $f_j$, 
and we condition on $\cG = \cG_{-k}$ defined as the information that remains after erasing 
the value taken by the single input variable $Y_k$.  
We can write $\cG_{-k} = (Y_1,\dots, Y_{k-1},Y_{k+1},\dots,Y_d)$.
The CDE $f(x\mid \cG_{-k})$ will be related to the density $f_k$ and will depend on the form of $h$. 
Checking for the continuity of the conditional cdf is usually easy in this case.
Note that this setting is only a particular case of our framework.
In many applications, $X$ is not defined like this in a way that $\cG_{-k}$ would satisfy 
Assumption 1 for some $k$. In Section~\ref{sec:experiments}, we examine examples 
that do not fit this setting and we provide more elaborate forms of conditioning.

Our first example is a sum of random variables, similar to \cite{vASM18a}.
It conveys the CDE idea in a simple setting. 
It also shows that selecting 
which variable to hide is not straightforward \new{even in this very simple setting}, 
and that the optimal choice \new{may depend} on the value of $x$ at which we estimate the density.
The second example shows how the choice of $\cG$ can make a significant difference in performance,
and that it is usually better to hide variables having a larger variance contribution.
The third example illustrates what we have to do to verify Assumption~1 for a given application.
The fourth example shows that we cannot always obtain an unbiased CDE by hiding a single variable.
The fifth example shows that it is not always easy to know what is the optimal information 
to hide.  On the other hand, the CDE can still work well even if we do not use the optimal $\cG$.

\begin{example} \rm
\label{ex:sum-asm18}
A very simple situation is when $X = h(Y_1,\dots,Y_d) = Y_1 + \cdots + Y_d$,
a sum of $d$ independent continuous random variables.
By hiding $Y_k$ for an arbitrary $k$, we get
\begin{eqnarray*}
  F(x\mid \cG_{-k}) &=& \PP[X\le x \mid S_{-k}] = \PP[Y_k \le x-S_{-k}] = F_k(x-S_{-k}),
\end{eqnarray*}
where $S_{-k} \eqdef \sum_{j=1,\, j\not=k}^d Y_j$,
and the density estimator becomes $f(x\mid \cG_{-k}) = f_k(x-S_{-k})$.
This form also works when the $Y_j$'s are not independent if we are able to compute
the density of $Y_k$ conditional on $\cG_{-k}$.
It then suffices to replace $f_k$ by this conditional density.
\cite{vASM18a} studied exactly this model, with independent variables and $k=d$.

When the $Y_j$'s have different distributions and we want to hide one, 
which one should we hide?  Intuition may suggest to hide the one having
the largest variance. This simple rule works well in a majority of cases, 
although it is not always optimal.  In particular, the optimal choice of 
variable $Y_k$ may depend on the value of $x$ at which we estimate
the density.  To illustrate this, let $d=2$, $X = Y_1 + Y_2$,
$f_1(y) = 2y$, and $f_2(y) = 2(1-y)$, for $y\in (0,1)$.
Then, $f(x) > 0$ for $0 < x < 2$.
If we hide $Y_2$, the density estimator at $x$ is $f_2(x-Y_1)$ and its second moment is 
$\EE[f_2^2(x-Y_1)] = \int_0^1 f_2^2(x-y_1) f_1(y_1)\d y_1$ whereas if we hide $Y_1$,
the density estimator at $x$ is $f_1(x-Y_2)$ and its second moment is 
$\EE[f_1^2(x-Y_2)] = \int_0^1 f_1^2(x-y_2) f_2(y_2)\d y_2$. 
One can easily verify that when $x$ is close to 0, these integrands are nonzero only when 
both $y_1$ and $y_2$ are also close to 0, and then the second integral is smallest, so it is 
better to hide $Y_1$.  When $x$ is close to 2, the opposite is true and it is better to hide $Y_2$.
In applications, changing the conditioning as a function of $x$ adds complications and is 
normally not necessary.  Using the same conditioning for all $x$, even when not optimal, 
is usually preferable because of its simplicity.
\hflorian{Another way to approach this situation is explained in Section \ref{sec:combination}.}
\end{example}

\begin{example}
\label{ex:sum2unif}
The following small example provides further insight into the choice of $\cG$.
Suppose $X$ is the sum of two independent uniform random variables:
$X = Y_1 + Y_2$ where $Y_1 \sim \cU(0,1)$ and $Y_2 \sim \cU(0,\epsilon)$ 
where $0 < \epsilon < 1$.
The exact density of $X$ here is 
$f(x) = x/\epsilon$ for $0\le x\le \epsilon$,
$f(x) = 1$ for $\epsilon\le x\le 1$, and 
$f(x) = (1+\epsilon-x)/\epsilon$ for $1\le x\le 1+\epsilon$.
Figure~\ref{fig:sum2unif} illustrates this density.

\begin{figure}[htbp] 
\begin{center}
 \begin{tikzpicture} 
    \begin{axis}[ 
      xlabel=$x$,
      ylabel=$f(x)$,
      width=0.4\columnwidth,
	    height=0.2\columnwidth,
	    legend style={at={(1.3,0.3)},anchor=south east},	
			grid,
			no marks,
      ] 
    \addplot table[x=logN,y=logIV] { 
      logN  logIV 
      0		0
      0.75	1
      1		1	
      1.75	0
      }; 
    \end{axis}
  \end{tikzpicture}
\hskip 15pt
 \begin{tikzpicture} 
    \begin{axis}[ 
      xlabel=$x$,
      width=0.4\columnwidth,
	    height=0.2\columnwidth,
	    legend style={at={(1.3,0.3)},anchor=south east},	
      grid,
			no marks,
      ] 
      \addplot table[x=logN,y=logIV] { 
      logN  logIV 
      0		0
      0.0625	1.
      1		1
      1.0625	0
 }; 
%
    \end{axis}
  \end{tikzpicture}
\end{center}
\caption{Exact density of $X$ for the model in Example~\ref{ex:sum2unif}
	      with $\epsilon=3/4$ (left) and $\epsilon=1/16$ (right).}
\label{fig:sum2unif}
\end{figure}
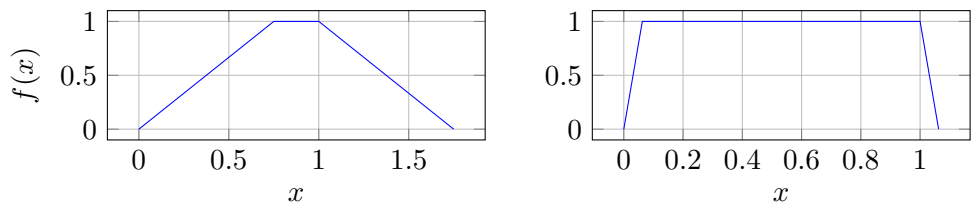

With $\cG = \cG_{-1}$, we have
$F(x\mid \cG_{-1}) = \PP[X\le x \mid Y_2] = \PP[Y_1 \le x-Y_2\mid Y_2] = x - Y_2$ and 
the density estimator is $f(x\mid \cG_{-1}) = 1$ for $Y_2 \le x \le 1 + Y_2$, and 0 elsewhere.
If $\cG = \cG_{-2}$ instead, then 
$F(x\mid \cG_{-2}) =  \PP[Y_2 \le x-Y_1\mid Y_1] = (x - Y_1)/\epsilon$ and 
the density estimator is $f(x\mid \cG_{-2}) = 1/\epsilon$ for $Y_1 \le x \le \epsilon + Y_1$, and 0 elsewhere.
In both cases, Assumption 1 holds and the density estimator with one sample is a uniform density, 
but the second one is over a narrow interval if $\epsilon$ is small.
When $\epsilon$ is small, $\cG = \cG_{-2}$ gives a density estimator $\hat f_{\cde,n}$ 
which is a sum of high narrow peaks and has much larger variance.
%
For this simple example, we can also derive exact formulas for the IV of the CDE under MC.  
For $\cG = \cG_{-1}$, $f(x\mid \cG_{-1}) = \II[Y_2 \le x \le 1 + Y_2]$ is a Bernoulli random variable 
with mean $\PP[x-1\le Y_2 \le x] = f(x)$, so its variance is $f(x)(1-f(x))$.
Integrating this over $[0,\, 1+\epsilon]$ gives $\IV = \epsilon/3$ for one sample.
For a sample of size $n$, this gives  $\IV = \epsilon/(3n)$.
For $\cG = \cG_{-2}$, $f(x\mid \cG_{-2}) = \II[Y_1 \le x \le \epsilon + Y_1]/\epsilon$ has also 
mean $f(x)$, but its variance is $\epsilon^{-1} f(x)(1-\epsilon f(x))$,
which is much larger than $f(x)(1-f(x))$ when $\epsilon$ is small.
Integrating over $[0,\, 1+\epsilon]$ gives $\IV = 1/\epsilon -1 + \epsilon/3$ for one sample,
which is also much larger than $\epsilon/3$ when $\epsilon$ is small.
The take-away: It is usually better to condition on lower-variance information
and hide variables having a large variance contribution.
\end{example}

\begin{example}
\label{ex:sum2normal}
In this example, we illustrate how Assumption~\ref{ass:cmc-dct} can be verified.
Let $X$ be the sum of two independent normal random variables, 
$X = Y_1 + Y_2$, where $Y_1 \sim \cN(0,\sigma_1^2)$, $Y_2 \sim \cN(0,\sigma_2^2)$,
and $\sigma_1^2 + \sigma_2^2 = 1$, so $X\sim \cN(0,1)$.
Let $\Phi$ and $\phi$ denote the cdf and density of the standard normal distribution.
With $\cG = \cG_{-2}$, we have $F(x\mid \cG_{-2}) =  \PP[Y_2 \le x-Y_1] = \Phi((x - Y_1)/\sigma_2)$ 
and the CDE is $f(x\mid \cG_{-2}) = \phi((x - Y_1)/\sigma_2) / \sigma_2$.  
Assumption~\ref{ass:cmc-dct} holds with $\Gamma = \phi(0)/\sigma_2$ and $K_{\gamma}=\Gamma^2$,
so this estimator is unbiased for $f(x) = \phi(x)$.   Its variance is 
\begin{align}
  \var[\phi((x-Y_1)/\sigma_2) /\sigma_2] 
  &= \EE[\exp[-(x-Y_1)^2/\sigma_2^2]/(2\pi \sigma_2^2)] - \phi^2(x) \nonumber\\
  &= \frac{1} {\sigma_2^2 \sqrt{2\pi}} \EE[\phi(\sqrt{2} ( x-Y_1)/\sigma_2)]-\phi^2(x) \nonumber\\
  &= \frac{1} {\sigma_2 \sqrt{2\pi(1+\sigma_1^2)}} 
	     \phi\left( \sqrt{2} x / \sqrt{1 + \sigma_1^2} \right) - \phi^2(x).
\label{eq:sum2normal-exact-var}
\end{align}
\end{example}

\begin{example}
\label{ex:minmax}
If $X$ is the min or max of two or more continuous random variables, then in general 
$F(\cdot \mid \cG_{-k})$ is not continuous, so if we hide only one variable,
Assumption~\ref{ass:cmc-dct} does not hold.
Indeed, if $X = \max(Y_1, Y_2)$ where $Y_1$ and $Y_2$ are independent,
with $\cG = \cG_{-2}$ (we hide $Y_2$), we have 
\[
  \PP[X\le x\mid Y_1=y] = \begin{cases} 
	         \PP[Y_2\le x\mid Y_1=y] = F_2(x) & \mbox{ if } x\geq y;\\ 
					                               0  & \mbox{ if } x < y.\\ 
	       \end{cases}
\]
If $F_2(y) > 0$, this function is discontinuous at $x=y$.
The same holds for the maximum of more than two variables.
One way to handle this is to generate all the variables, 
then hide the maximum and compute 
its conditional density given the other ones.
Without loss of generality, suppose $Y_1$ is the maximum and $Y_2 = y_2$ the second largest.
Then the CDE of the max is $f(x\mid \cG) = f_1(x\mid Y_1 > y_2)$.
Note that for independent random variables whose cdf's and densities have an analytical form,
the cdf and density of the max can often be computed analytically.
See Section~\ref{sec:san} for more on this.
A very similar story holds if we replace the max by the min.
\hflorian{The idea of using conditioning for gradient estimation has been used for this
example already in \cite[p. 731]{oFU08a}. In that paper, however, the authors take the 
derivative w.r.t. a distribution parameter of $Y_1$ and not w.r.t. $x$. 
Consequently, the interchange of the derivative with the expectation was justified in 
their case but also, the potential for density estimation remained unnoticed.}
\hpierre{Oh, I thought they were taking the derivative w.r.t. $x$ when you pointed out this
 example to me. If they only consider the derivative w.r.t. a parameter $\theta$ of $X_1$,
 then this was done before them.  For example in L'Ecuyer and Perron (1994).  
 And this is not really what we do here.   
 The derivative w.r.t. $x$ is quite a different thing!  
 In this case, I think we do not have to cite Mike Fu. }
\hflorian{I hid the passage in question now. Well, then we probably need to find a place to 
 reference vFU09a, i.e., Fu, Hong, and Hu's paper on quantile sensitivities.}
\hpierre{Of course, there is no need to estimate $f$ for this simple example, because it
can be computed exactly, which amount to hiding \emph{both} variables.  
For $X = \max(Y_1, Y_2)$, we have $F(x) = \PP[\max(Y_1, Y_2)\le x] = F_1(x) F_2(x)$ and then 
$f(x) = F_1(x) f_2(x) + F_2(x) f_1(x)$.} 
\end{example}

\begin{example}
Suppose $X = Z\cdot C$ where $Z\sim N(0,1)$ and $C$ is continuous with support over $(0,\infty)$.
We can hide $Z$ and generate $X\sim N(0,C^2)$ conditional on $C$, or do the opposite. 
Which one is best depends on the distribution of $C$.
Here we have $\Var[X] = \EE[\Var[X\mid C]] = \EE[C^2]$ while $\Var[\EE[X\mid C]] = 0$.
So the usual variance decomposition tells us nothing about what to hide.
This illustrates the fact that there is rarely a simple rule to \new{find} the optimal $\cG$.
\end{example}

\subsection{Convex combination of conditional density estimators}
\label{sec:combination}

When there are many possible choices of $\cG$ for a given problem, 
one can select more than one and take a convex linear combination of the corresponding CDEs
as the final density estimator.  
This idea is well known for general mean estimators \citep{sBRA87a}.
More specifically, suppose $\hat f_{0,n}, \dots, \hat f_{q,n}$ are $q+1$ distinct 
unbiased density estimators.  Typically, these estimators are dependent
and based on the same simulations. 
They could be all CDEs based on different choices of $\cG$ 
(so they will not hide the same information), 
but there could be non-CDEs as well.  
A convex combination can take the form
\begin{equation}
  \hat f_n(x) = \beta_0\hat f_{0,n}(x) + \cdots + \beta_q \hat f_{q,n}(x)   
	            = \hat f_{0,n}(x) - \sum_{\ell=1}^q \beta_{\ell} (\hat f_{0,n}(x) - \hat f_{\ell,n}(x))
\label{eq:convex-comb1}
\end{equation}
for all $x\in\RR$, where 
$\beta_0 + \cdots + \beta_q = 1$.
This is equivalent to choosing $\hat f_{0,n}(x)$ as the main estimator, 
and taking the $q$ differences $\hat f_{0,n}(x) - \hat f_{\ell,n}(x)$ as control variables
\citep{sBRA87a}, {Problem 2.3.9.}
With this interpretation, the optimal coefficients $\beta_{\ell}$ can be estimated via 
standard control variate theory \citep{sASM07a} by trying to minimize the IV of 
$\hat f_n(x)$ w.r.t. the $\beta_\ell$'s.
More precisely, if we denote $\IV_{\ell} = \IV(\hat f_{\ell,n}(x))$ and 
$ \IC_{\ell,k} = \int_{a}^{b}\Cov[\hat f_{\ell,n}(x),\hat f_{k,n}(x)]\d x$,
we obtain 
 \begin{equation*}
 \IV = \IV\left(\hat{f}_{n}(x)\right) 
     = \sum_{\ell=0}^{q}\beta_\ell^2 \IV_\ell + 2 \sum_{0\leq \ell < k\leq q} \beta_\ell\beta_k \IC_{\ell,k}.
 \end{equation*}
Given the $\IV_\ell$'s and $\IC_{\ell,k}$'s (or good estimates of them), 
this IV is a quadratic function of the $\beta_\ell$'s,
which can be minimized exactly as in standard least-squares linear regression.
That is, the optimal coefficients $\beta_j$ obey the standard linear regression formula.
Estimating the density and coefficients from the same data yields biased but 
consistent density estimators, and the bias is rarely a problem.
We followed this approach for some of the examples in Section~\ref{sec:experiments}.
\cite{oCUI20a} obtained an equivalent formula from a slightly different 
but equivalent reasoning.

\hpierre{In fact, the optimal coefficients should depend on $x$, so we might take $\beta_\ell(x)$
 that depend on $x$.  We can estimate the optimal coeficients at a few values of $x$, and then fit a 
 spline model, for example.  I did that earlier in \cite{vLEC08b}; see that paper for the details.
 I think we can write a section (before the examples) that examines this issue for the general case,
 and explains how to select coefficients $\beta_\ell(x)$ that depend on $x$.  I think that this can improve 
 the value of the paper.}
\hamal{In all our experiments we consider only the coefficients that don't depend on x.}
Given that the best choice of $\cG$ generally depends on $x$, one may also adopt 
a more refined approach which allows the coefficients $\beta_j$ to depend on $x$:
\begin{equation}
  \hat f_n(x) = \beta_0(x) \hat f_{0,n}(x) + \cdots + \beta_q(x) \hat f_{q,n}(x)   
	            = \hat f_{0,n}(x) - \sum_{\ell=1}^q \beta_{\ell}(x) (\hat f_{0,n}(x) - \hat f_{\ell,n}(x)),
\label{eq:convex-comb2}
\end{equation}
where $\beta_0(x) + \cdots + \beta_q(x) = 1$ for all $x\in\RR$.
The optimal coefficients can be estimated by standard control variate theory at selected values of $x$,
then for each $\ell\ge 1$, one can fit a smoothing spline to these estimated values, by least squares.
This provides estimated optimal coefficients that are smooth functions of $x$,
which can be used to obtain a final CDE.  
This type of strategy was used in \cite{vLEC08b} to estimate varying control variate coefficients.
The additional flexibility can improve the variance reduction in some situations.
\hpierre{Amal, Florian:  Please add short but clear explanations on how you have implemented these two 
 approaches.  In the first case, you may have used the integrals (w.r.t. $x$) of the covariances. 
 In the second case, you have to do something to approximate the optimal coefficients $\beta_\ell(x)$
 as functions of $x$, for $\ell \ge 1$.}

\hpierre{Around here, we may need to add a subsection that provides guidelines on how 
 to choose the conditioning. }

\subsection{A GLR density estimator (GLRDE)}
\label{sec:other-approaches}
\label{sec:GLRDE-estimator}

\hpierre{Here we briefly recall two other Monte Carlo-based density estimation approaches,
KDE+RQMC and the GLR, 
with which we will compare our method in the numerical examples.
\cite{vBEN21a} study the combination of a KDE with RQMC sampling.
They observe that RQMC reduces the IV under certain conditions, but that it does not affect
the bias, and as a result the gain is limited, especially when the dimension is moderate or large.}
 
The \emph{generalized likelihood ratio} (GLR) method, originally developed by \cite{oPEN18a}
to estimate the derivative of an expectation with respect to some model parameter,
can be adapted to density estimation, as shown in \cite{oPEN20a}.
We summarize briefly here how this method estimates the density $f(x)$ in our general setting,
so we can apply it in our examples and make numerical comparisons.
The assumptions stated below differ slightly from those in \cite{oPEN20a}.
In particular, here we do not have a parameter $\theta$,
the conditions on the estimator are required only in the area where $X \le x$,
and we add a condition to ensure finite variance.
As in Section~\ref{sec:small-examples}, we assume here that
$X = h(\bY) = h(Y_1,\dots,Y_d)$ where $Y_1,\dots,Y_d$ are independent continuous random variables, 
and $Y_j$ has cdf $F_j$ and density $f_j$.  
Let $P(x) = \{\by \in\RR^d : h(\by) \le x\}$.
For $j=1,\dots,d$, let $h_j(\by) := \partial h(\by) /\partial y_j$,
$h_{jj}(\by) := \partial^2 h(\by) /\partial y_j^2$, and 
\begin{equation}
 \Psi_j(\by) = \frac{\partial \log f_j(y_j) /\partial y_j - h_{jj}(\by) / h_j(\by)}{h_j(\by)}.
\end{equation}

\begin{assumption}
\label{ass:GLR1}
The Lebesgue measure of $h^{-1}((x-\epsilon, x+\epsilon))$ in $\RR^d$ goes to 0 
when $\epsilon\to 0$ (this means essentially that the density is bounded around $x$).
\end{assumption}

\begin{assumption}
\label{ass:GLR2}
The set $P(x)$ is measurable, the functions $h_j$, $h_{jj}$, and $\Psi_j$ 
are well defined  over it, 
and $\EE[\II[X \le x] \cdot\Psi_j^2(\bY)] < \infty$.
\hpierre{This is an adaptation and modification of the assumptions from \cite{oPEN20a}, 
 page 12, to our setting, specialized to a neighborhood of $x$.
 The last condition is to ensure finite variance.  Note that we do not care what happens
 when $X = h(\bY) > x$.  On the other hand, maybe we need some additional technical 
 assumptions made in \cite{oPEN20a}.  Perhaps we should sketch the proof of the proposition 
 under our assumptions.  If so, this would be like another new result in our paper,
 which could be cited!}
\end{assumption}

\begin{proposition}
\label{prop:GLR}
Under Assumptions~\ref{ass:GLR1} and \ref{ass:GLR2}, the GLRDE
$\II[X \le x] \cdot\Psi_j(\bY)$ is an unbiased and finite-variance estimator 
of the density $f(x)$ at $x$.
\end{proposition}

For the proof of Proposition~\ref{prop:GLR} and additional details, see \cite{oPEN20a}.

\section{Combining RQMC with the CMC density estimator}
\label{sec:rqmc}

We now discuss how RQMC can be used with the CDE,  
and under what conditions it can provide a convergence rate faster than $\cO(n^{-1})$
for the IV of the resulting unbiased estimator.
For this, we first recall some basic facts about QMC and RQMC.
More detailed coverages can be found in \cite{rNIE92b}, \cite{rDIC10a}, and
\cite{vLEC09f,vLEC18a}, for example.

For a function $g : [0,1)^s \to\RR$, the integration error by the average over a point set 
$P_n = \{\bu_1,\dots,\bu_{n}\}\subset [0,1]^s$ is defined by
\begin{equation}
\label{eq:En}
  E_n = \frac{1}{n} \sum_{i=1}^{n} g(\bu_i) - \int_{[0,1]^s} g(\bu)\d\bu.
\end{equation}
Classical QMC theory bounds this error as follows.
Let $\frakv\subseteq \cS := \{1,\dots,s\}$ denote an arbitrary subset of coordinates.
For any point $\bu = (u_1,\dots,u_s) \in[0,1]^s$, $\bu_{\frakv}$ denotes the projection of $\bu$ 
on the coordinates in $\frakv$ and $(\bu_\frakv,\bone)$ is the point $\bu$ in which $u_j$ is
replaced by 1 for each $j\not\in \frakv$. 
Let $g_{\frakv} := \partial^{|\frakv|} g/\partial\bu_{\frakv}$ denote the 
partial derivative of $g$ with respect to all the coordinates in $\frakv$. 
When $g_{\frakv}$ exists and is continuous for $\frakv = \cS$ (i.e., for all $\frakv \subseteq \cS$),
the \emph{Hardy-Krause (HK) variation} of $g$ can be written as
\begin{equation}
\label{eq:HK}
  V_{\rm HK}(g) = \sum_{\emptyset\not=\frakv\subseteq\cS} \int_{[0,1]^{|\frakv|}} 
	   \left|g_{\frakv}(\bu_{\frakv},\bone)\right| \d\bu_{\frakv}.
\end{equation}
On the other hand, the \emph{star-discrepancy} of $P_n$ is 
\[
  D^*(P_n) = \sup_{\bu\in[0,1]^s} 
    \left|\frac{|P_n\cap [\bzero,\bu)|}{n} - \vol[\bzero,\bu)\right|
\]
where $\vol[\bzero,\bu)$ is the volume of the box $[\bzero,\bu)$.
The classical \emph{Koksma-Hlawka (KH) inequality} bounds the absolute error by the product of 
these two quantities, one that involves only the function $g$ and the other that involves
only the point set $P_n$:
\begin{equation}
  |E_n| \le V_{\rm HK}(g) \cdot D^*(P_n).              \label{eq:kokla}
\end{equation}

There are explicit construction methods (e.g., digital nets, lattice rules, and polynomial lattice rules)
of deterministic point sets $P_n$ for which 
$D^*(P_n) = \cO((\log n)^{s-1} /n) = \cO(n^{-1 +\epsilon})$ for all $\epsilon > 0$.
This means that functions $g$ for which $V_{\rm HK}(g) < \infty$ can be integrated by QMC
with a worst-case error that satisfies $|E_n| = \cO(n^{-1+\epsilon})$.
There are also known methods to randomize these point sets $P_n$ in a way that 
each randomized point $\bu_i$ has the uniform distribution over $[0,1)^s$,
so $\EE[E_n] = 0$, and the $\cO(n^{-1 +\epsilon})$ discrepancy bound is preserved, which gives 
\begin{equation}
  \Var[E_n] = \EE[E_n^2] = \cO(n^{-2+\epsilon}).       \label{eq:var-bound}
\end{equation}

The classical definitions of variation and discrepancy given above are
only one pair among an infinite collection of possibilities.  
There are other versions of (\ref{eq:kokla}), with different definitions of the discrepancy and
the variation, such that there are known point set constructions for which the discrepancy converges as 
$\cO(n^{-\alpha+\epsilon})$ for $\alpha > 1$, but the conditions on $g$ to have finite variation
are more restrictive (more smoothness is required) \citep{rDIC10a}.

From a practical viewpoint, getting a good estimate or an upper bound on the variation of $g$
that can be useful to bound the RQMC variance is a notoriously difficult problem.
Even just showing that the variation is finite is not always easy.
However, finite variation is not a necessary condition.
In many realistic applications in which variation is known to be infinite, 
RQMC can nevertheless reduce the variance by a large factor \citep{vLEC09f,vLEC12a,vHE15b}.
The appropriate explanation for this depends on the application.
In many cases, part of the explanation is that the integrand $g$ can be written as a sum
of orthogonal functions (as in an ANOVA decomposition) and a set of terms in that sum 
have a large variance contribution and are smooth low-dimensional functions for which 
RQMC is very effective \citep{vLEC00b,vLEC09f,sLEM09a}.
Making such a decomposition and finding the important terms is difficult for 
realistic problems, but to apply RQMC in practice, this is not needed.
The usual approach in applications is to try it and compare the RQMC variance with the 
MC variance empirically.  
We will do that in Section~\ref{sec:experiments}.
\hpierre{To estimate the RQMC variance, we usually replicate the RQMC scheme $n_r$ times independently, using the same point set but with $n_r$ independent randomizations,
then we compute the empirical mean and variance of the $n_r$ independent realizations of 
the RQMC density estimator  $({1}/{n}) \sum_{i=1}^{n} g(\bU_i)$.}

To combine the CDE with RQMC, we must be able to write $F(x \mid \cG) = \tilde g(x,\bu)$ and 
$f(x \mid \cG) = \tilde g'(x,\bu) = \d \tilde g(x,\bu)/\d x$
for some function $\tilde g : [a,b]\times [0,1)^s$.
The function $\tilde g'(x,\cdot)$ will act as $g$ in (\ref{eq:En}).
The combined CDE+RQMC estimator $\hat f_{\cderqmc,n}(x)$ will be defined by
\begin{equation}
\label{eq:hatf-cde-rqmc}
  \hat f_{\cderqmc,n}(x) = \frac{1}{n} \sum_{i=1}^n \tilde g'(x,\bU_i),
\end{equation}
which is the RQMC version of (\ref{eq:hatf-cde}).
To estimate the RQMC variance, we can perform $n_r$ independent randomizations to obtain 
$n_r$ independent realizations of $\hat f_{\cderqmc,n}$ in (\ref{eq:hatf-cde-rqmc}) with RQMC,
and compute the empirical IV.  
By putting together the previous results, we obtain:

\begin{proposition}
\label{prop:hk-cde}
If $\sup_{x\in [a,b]} V_{\rm HK}(\tilde g'(x,\cdot)) < \infty$,
then with RQMC points sets $P_n$ with $D^*(P_n) = \cO((\log n)^{s-1} /n)$, 
for any $\epsilon > 0$, we have 
$\sup_{x\in [a,b]} \Var[\hat f_{\cderqmc,n}(x)] = \cO(n^{-2 +\epsilon})$, 
so the MISE of the CDE+RQMC estimator converges as $\cO(n^{-2 +\epsilon})$.
\end{proposition}

Although this is rarely done in practice, it is instructive to see how the HK variation 
of $\tilde g'(x,\cdot)$ can be bounded in our CDE setting,   
so that Proposition~\ref{prop:hk-cde} applies.
For this, we need to show that the integral of the partial derivative 
of $\tilde g'(x,\bu)$ with respect to each subset of coordinates of $\bu$ is finite.  
In \iflong Example~\ref{ex:sum-rqmc} below, 
\else Section~\ref{sec:sum-rqmc} of the Supplement, \fi  
we do it for Examples~\ref{ex:sum-asm18} to \ref{ex:sum2normal}.
When the variation is unbounded, RQMC may still reduce the IV, but there is no guarantee.
The GLRDE in Proposition~\ref{prop:GLR} is typically discontinuous
because of the indicator function, and therefore its HK variation is usually infinite.

\section{Examples and numerical experiments}
\label{sec:experiments}

We now examine larger instructive examples for which we show how to construct a CDE,
summarize the results of numerical experiments with the CDE and CDE+RQMC,
and make comparisons with the GLRDE and KDE, with MC and RQMC.
Section~\ref{sec:setting} gives the experimental framework used for all the 
numerical experiments.
In Section~\ref{sec:cantilever}, we use a three-dimensional real-life example to provide
further insight on the choice of conditioning and make comparisons between methods.
In Section~\ref{sec:san}, we estimate the density of the length $X$ of the longest path 
between the source and destination in a stochastic network.  
This length may represent the total time to execute a project,
the arrival time of a train at a given station, etc.  
The length of the shortest path can be handled in a similar way.  
In Section~\ref{sec:single-queue}, $X$ is the waiting time of a customer in a queuing system.
We consider a single queue in the example, but a similar conditioning would apply for 
larger queueing systems as well.
In Section~\ref{sec:function-multinormal}, $X$ is the payoff of a financial option.
We show that by using a clever conditioning with CDE+RQMC, the MISE can be reduced by huge factors.  
More examples are given in the Online Supplement.
In all these examples, estimating the density of $X$ has high practical relevance.
Larger problem instances can also be handled with the same methods.

\subsection{Experimental setting}
\label{sec:setting}

Since the CDE is unbiased, we measure its performance by the IV, 
which equals the MISE in this case.   
To approximate the IV estimator (\ref{eq:hat-iv}) for a given $n$, 
we first take a stratified sample $e_{1},\dots, e_{n_e}$ of $n_e$ \emph{evaluation points} 
at which the empirical variance will be computed.
We sample $e_j$ uniformly in $[a + (j-1)(b-a)/n_e,\, a + j(b-a)/n_e)$ for $j=1,\dots,n_e$.
Then we use the unbiased IV estimator
\[
  \widehat\IV = \frac{(b-a)}{n_e} \sum_{j=1}^{n_e} \widehat{\Var}[\hat f_n(e_j)],
\]
where $\widehat{\Var}[\hat f_n(e_j)]$ is the empirical variance of the CDE at $e_j$,
obtained as follows.
We repeat the following $n_r$ times, independently:
Generate $n$ observations of $X$ from the density $f$ with the given method (MC or RQMC), 
and compute the CDE at each evaluation point $e_j$. 
We then compute $\widehat{\Var}[\hat f_n(e_j)]$ as the empirical variance of the $n_r$ 
density estimates at $e_j$, for each~$j$.
In all our examples, we used $n_r=100$ and $n_e=128$.

To estimate the convergence rate of the IV as a function of $n$ with the different methods,
we fit a model of the form $\IV \approx K n^{-\nu}$.  
For the CDE with independent points (no RQMC), this model holds exactly with $\nu=1$.
We hope to observe $\nu > 1$ with RQMC.
The parameters $K$ and $\nu$ are estimated by linear regression in log-log scale, i.e.,
by fitting the model $\log \IV \approx \log K - \nu\log n$ to data.
Since $n$ is always taken as a power of 2, we report the logarithms in base 2.
We estimated the IV for $n = 2^{14},\dots, 2^{19}$ (6 values) to fit the regression model.
We also report the observed $-\log_2 \IV$ for $n=2^{19}$
and use $e19$ as a shorthand for this value in the tables. 
We use exactly the same procedure for the GLRDE.
For the KDE, these values are for the MISE instead of the IV.
In all cases, we used a normal kernel and a bandwidth $h$ 
selected by the methodology described in \cite{vBEN21a}.
For some examples, we tried CDEs based on different choices of $\cG$ and 
a convex combination as in Section~\ref{sec:combination}.

%
We report results with the following types of point sets:
\begin{verse}
(1) independent points (MC);\\
(2) a randomly-shifted lattice rule (Lat+s);\\
(3) a randomly-shifted lattice rule with a baker's transformation (Lat+s+b); \\
(4) Sobol' points with a left random matrix scramble and random digital shift (Sob+LMS).
\end{verse}
The short names in parentheses are used in the plots and tables.
For the definitions and properties of these RQMC point sets, 
see \cite{vLEC00b,vOWE03a,vLEC09f,vLEC18a}.
They are implemented in SSJ \citep{iLEC16j}, which we used for our experiments.
The parameters of the lattice rules were found with the Lattice Builder software of \cite{vLEC16a},
using a fast-CBC construction method with the $\mathcal{P}_2$ criterion and order dependent weights 
$\gamma_{\frakv} = \rho^{|\frakv|}$, with $\rho$ ranging from 0.05 to 0.8, depending on the example
(a larger $\rho$ was used when the dimension $s$ was smaller).
The baker's transformation sometimes improves the convergence rate by making the integrand periodic
\citep{vHIC02a}, but it can also increase the variation of the integrand, so its impact on
the variance can go either way.
\hpierre{The lattice parameters were found with Lattice Builder with $\mathcal{P}_2$,
  with weights $\gamma_\frakv = 0.6^{|\frakv|}$.}

\subsection{Displacement of a cantilever beam}
\label{sec:cantilever}

We consider the following (real-life) model for the displacement $X$ of a cantilever beam 
with horizontal and vertical loads, taken from \cite{sBIN17a}:
\begin{equation}       \label{eq:beam-displacement}
   X = h(Y_1,Y_2,Y_3) = \frac{4 \ell^3}{Y_1 wt}\sqrt{\frac{Y_2^2}{w^4} + \frac{Y_3^2}{t^4}}
\end{equation}
in which $\ell = 100$, $w = 4$ and $t = 2$ are constants (in inches), 
while $Y_1$ (Young's modulus), $Y_2$ (the horizontal load), and $Y_3$ (the vertical load), 
are independent normal random variables, $Y_j \sim \cN(\mu_j,\sigma_j^2)$, 
i.e., normal with mean $\mu_j$ and variance $\sigma_j^2$.
The parameter values are $\mu_1 = 2.9\times 10^7$, $\sigma_1 = 1.45 \times 10^{6}$,
$\mu_2 = 500$, $\sigma_2 = 100$,
$\mu_3 = 1000$, $\sigma_3 = 100$.
We will denote $\kappa = 4\ell^3/(wt) = 5\times 10^5$.
The goal is to estimate the density of $X$ over the interval $[3.1707,\, 5.6675]$, 
which covers about 99\% of the density (it clips $0.5\%$ on each side).
It is possible to have $X < 0$ in this model, but the probability is 
$\PP[Y_1 < 0] = \Phi(-20) = 2.8\times 10^{-89}$, which is negligible,
so we can assume that $Y_1 > 0$.
This example fits the framework of Section~\ref{sec:small-examples}, with $d=3$.
We can hide any of the three random variables for the conditioning, and we will examine each case.

Conditioning on $\cG_{-1}$ means hiding $Y_1$.  We have 
\[
  X = \frac{\kappa}{Y_1}\sqrt{\frac{Y_2^2}{w^4} + \frac{Y_3^2}{t^4}} \le x
\quad \mbox{ if and only if } \quad
  Y_1 \ge \frac{\kappa}{x}\sqrt{\frac{Y_2^2}{w^4} + \frac{Y_3^2}{t^4}}  \eqdef W_1(x).
\]
Note that $W_1(x) > 0$ if and only if $x > 0$.
For $x > 0$,
\[
  F(x \mid \cG_{-1}) 
	 = \PP[Y_1\ge W_1(x)\mid W_1(x)] = 1 - \Phi((W_1(x)-\mu_1)/\sigma_1)
\]
which is continuous and differentiable in $x$, and  
\[
  f(x\mid \cG_{-1}) = -{\phi((W_1(x)-\mu_1)/\sigma_1) W'_1(x)} / \sigma_1  
	        = {\phi((W_1(x)-\mu_1)/\sigma_1) W_1(x)} /( x \sigma_1).
\]
\hflorian{To check the bounded HK variation is elementary, but it involves quite complicated formulas. In the end, it once again boils down to the fact that $\phi$ times a polynomial is integrable.}%
If we condition on $\cG_{-2}$ instead, i.e., we hide $Y_2$, we have ${X}\le x$ if and only if 
\[
  Y_2^2 \le w^4 \left((x Y_1/\kappa)^2 - Y_3^2/t^4\right)  \eqdef W_2(x).
\]
If $W_2(x) \le 0$, then $f(x \mid \cG_{-2}) = F(x \mid \cG_{-2}) = \PP[X\le x\mid W_2(x)] = 0$. 
For $W_2(x) > 0$, we have 
\begin{eqnarray*}
  F(x \mid \cG_{-2}) &=& \PP[X\le x\mid W_2(x)] 
	 = \PP\left[-\sqrt{W_2(x)}\le Y_2 \le \sqrt{W_2(x)}\mid W_2(x)\right] \\
	 &=& \Phi((\sqrt{W_2(x)}-\mu_2)/\sigma_2) - \Phi(-(\sqrt{W_2(x)}+\mu_2)/\sigma_2),
\end{eqnarray*}
which is again continuous and differentiable in $x$, and
\[
  f(x\mid \cG_{-2}) = \frac{\phi((\sqrt{W_2(x)}-\mu_2)/\sigma_2) + \phi(-(\sqrt{W_2(x)}+\mu_2)/\sigma_2)}{
	                      (\sigma_2\sqrt{W_2(x)}) / ( w^4 x (Y_1 /\kappa)^2)  } > 0.
\]
If we condition on $\cG_{-3}$, the analysis is the same as for $\cG_{-2}$, by symmetry, and we get
\[
  f(x\mid \cG_{-3}) = \frac{\phi((\sqrt{W_3(x)}-\mu_3)/\sigma_3) + \phi(-(\sqrt{W_3(x)}+\mu_3)/\sigma_3)}{
	                       (\sigma_3\sqrt{W_3(x)}) / (t^4 x (Y_1 /\kappa)^2 ) } > 0
\]
for $W_3(x) > 0$, where $W_3(x)$ is defined in a similar way as $W_2(x)$.
\hpierre{In addition to testing these three ways of conditioning, which yield three different CDEs, 
we also tested a convex combination of the three,
as explained in Section~\ref{sec:combination}, with coefficients $\beta_\ell$ that do not depend on $x$.}
\hflorian{Moved and adapted this paragraph after the derivation of the GLRDE's.}

For the GLRDE, we write $h(\bY) = (\kappa/Y_1) S^{1/2}$ where 
$S = Y_2^2/w^4 + Y_3^2/t^4$, and denote 
$Z_j = (Y_j-\mu_j)/\sigma_j^2 = -\partial \log f_j(Y_j)/\partial Y_j$ for $j=1,2,3$.
With this notation, we obtain 
$h_1(\bY) = -h(\bY)/Y_1$, $h_{11}(\bY) = 2h(\bY)/Y_1^2$, 
$h_2(\bY) = (\kappa / Y_1)(Y_2/w^4) S^{-1/2} = h(\bY) Y_2 / (S w^4)$,
$h_{22}(\bY) = (\kappa/(Y_1 w^4)) (S^{-1/2} - S^{-3/2} Y_2^2 /w^4)$, 
$h_3(\bY) = (\kappa / Y_1)(Y_3/t^4) S^{-1/2} = h(\bY) Y_3 / (S t^4)$,
$h_{33}(\bY) = (\kappa /(Y_1 t^4)) (S^{-1/2} - S^{-3/2} Y_3^2 /t^4)$.
With a little calculation, this gives
\begin{gather*}
 \Psi_1(\bY) = \frac{Y_1 Z_1 - 2}{h(\bY)}, \qquad 
 \Psi_2(\bY) = - \frac{Y_2 Z_2 S + Y_3^2/t^4}{h(\bY)\,Y_2^2/w^4}, \qquad 
 \Psi_3(\bY) = - \frac{Y_3 Z_3 S + Y_2^2/w^4}{h(\bY)\,Y_3^2/t^4}.
\end{gather*}
In addition to testing the individual estimators derived above, 
we also tested convex combinations of the three CDEs and of the three GLRDEs,
as explained in Section~\ref{sec:combination}, with coefficients $\beta_\ell$ that do not depend on $x$.
\hpierre{Referee 2 asks us to also report on the GLRDE using $Y_2$ and $Y_3$, 
 and the linear combination. Then the table can have one panel for $\hat\nu$ and 
 one panel for e19. Of if the results with $Y_2$ and $Y_3$ and just worse than for $Y_1$,
 then perhaps we can just say that in one line, to save space.}
\begin{table}[!htbp]
  \caption{Values of $\hat\nu$ and e19 with a CDE for each choice of $\cG_{-k}$ and
	  for the best convex combination (CDE-c), for the GLRDE with each $\Psi_j$ and
	  for the best convex combination (GLRDE-c), and for the KDE, for the cantilever beam model.}
  \label{tab:cantilever-results}
  \centering
  \begin{tabular}{|l| c c c c | c c c c | c|}
     \cline{2-10}
  \multicolumn{1}{l|}{}	& $\cG_{-1}$	& $\cG_{-2}$	& $\cG_{-3}$	& CDE-c	
						& $\Psi_1$  	& $\Psi_2$ 		& $\Psi_3 $ 	& GLRDE-c 	& KDE \\
     \cline{2-10}
  \multicolumn{1}{l|}{} 	& \multicolumn{9}{c|}{e19}\\
	\hline
	MC			& 19.3		& 14.5		& 22.8		& 22.5  
				& 14.1 		& 4.5		& 15.8		& 16.3	&  15.8 \\
	Lat+s		& 39.8		& 25.2		& 41.6		& 41.9   
				& 23.4 		& -2.5 		& 26.4		& 26.5 	&   21.9	\\
	Lat+s+b		& 44.5		& 23.7		& 46.8		& 47.0 
				& 23.3 		& 5.7 		& 24.7 		& 25.1 	& 21.0	\\
	Sob+LMS		& 44.0		& 23.6		& 45.7		& 46.1  
				& 23.4 		& 2.8		& 25.5		& 25.9 	& 21.5		\\
	\hline
	 \multicolumn{1}{l|}{} &\multicolumn{9}{c|}{$\hat\nu$} \\
	\hline
	MC			& 0.97		& 0.98		& 0.99		& 0.98	
				& 1.02 		& 0.55		& 0.94 		& 0.95 	& 0.76 \\
	Lat+s		& 1.99		& 1.95		& 2.06		& 2.04 
				& 1.38 		&  ---		& 1.51 		& 1.52 	& 1.03	\\
	Lat+s+b		& 2.24		& 2.08		& 2.27		& 2.25 
				& 1.37 		&  --- 	& 1.24 		& 1.25 	& 0.93	\\
	Sob+LMS		& 2.21		& 2.03		& 2.21		& 2.21	
				& 1.32 		&  ---	& 1.31		& 1.32 	& 0.97		\\
	\hline
  \end{tabular}
 \end{table}

Table~\ref{tab:cantilever-results} summarizes the results.
\hflorian{For this example, the lattice parameters were obtained using $\mathcal{P}_2$
  with weights $\gamma_\frakv = 0.05^{|\frakv|}$.}%
The MISE is about ${2^{-47}}$ for the best CDE+RQMC compared with ${2^{-15.8}}$ for the usual KDE+MC,
a gain by a factor of over $2^{31} \approx$ 2 billions.
This is probably much better accuracy than required in practice for this particular application.
With RQMC, the convergence rate $\hat\nu$ is around 2 in all cases with the CDE methods,
and much less for GLRDE and KDE. 
The GLRDE using $\Psi_2$ behaves very badly (the estimates $\hat\nu$ with RQMC 
are meaningless), but with $\Psi_1$ and $\Psi_3$ (the best choice), 
it performs better that the KDE. Note that the denominator of $\Psi_2$ takes much smaller
values on average than that of $\Psi_3$, and this can explain its larger variance.

For the CDE with lattice rules, the baker's transformation helps significantly for the CDE. 
Conditioning on $\cG_{-2}$ does not give as much reduction as for the other choices.   
To provide visual insight, Figure~\ref{fig:canti-cde} shows  plots of five realizations
of the conditional density for $\cG_{-1}$, $\cG_{-2}$, and $\cG_{-3}$.
The realizations of $f(\cdot\mid \cG_{-2})$ have high narrow peaks, which explains the larger variance.
The average of the five realizations is shown in orange (dotted in the b/w version) and the true density in black.
In Figure~\ref{fig:canti-MCvsRQMC}, we zoom in on part of the estimated densities to show the difference 
between MC and RQMC.  In each panel one can see the CDE using MC (in orange), RQMC (in green), 
and the ``true density'' (black, dashed) estimated with RQMC using a large number of samples.
We have $\cG_{-1}$ with $n=2^{10}$ on the left and $\cG_{-2}$ with $n=2^{16}$ on the right.
In both cases, the RQMC estimate is closer to the true density, and on the left it oscillates less.
If we repeat this experiment several times, the orange curve would vary much more than
the green one across the realizations.

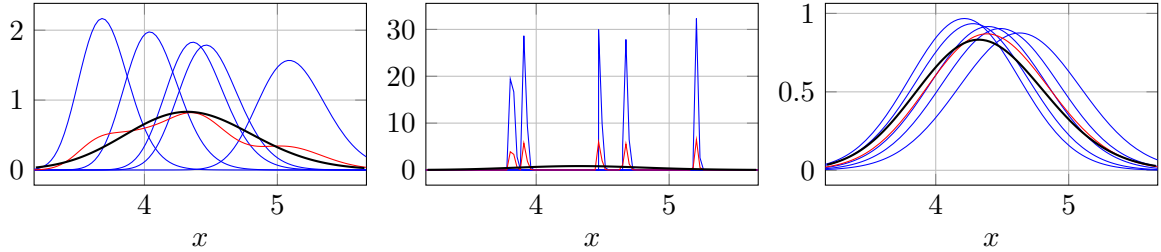
\begin{figure}[hbt]
\begin{center}
\begin{tikzpicture} 
  \begin{axis}[ 
  	cycle list name=canti,
  legend style={at={(0.6,1.3)}, anchor={north west}},
  xlabel=$x$,
  ylabel=,
  grid,
  xmin=3.17,
  xmax=5.67,
  width = 6cm,
  height = 4cm,
  ] 
  \addplot table[x=x,y=dens] { 
  	x  dens 
 3.17067	6.03389E-13
3.19518	2.95494E-12
3.21968	1.35157E-11
3.24625	6.52285E-11
3.27282	2.9229E-10
3.29762	1.11191E-9
3.32243	3.98428E-9
3.34675	1.31664E-8
3.37107	4.12486E-8
3.39745	1.34321E-7
3.42383	4.12673E-7
3.44845	1.1184E-6
3.47307	2.89187E-6
3.49975	7.69324E-6
3.52643	0.0000194456
3.55263	0.000046068
3.57882	0.000104269
3.60326	0.000214713
3.62769	0.000426125
3.65418	0.000860732
3.68068	0.00166992
3.70542	0.00299347
3.73015	0.00519375
3.7544	0.00864561
3.77865	0.0139768
3.80496	0.0227972
3.83127	0.0360112
3.85581	0.0536538
3.88036	0.0778904
3.90697	0.113413
3.93358	0.160511
3.95971	0.2198
3.98583	0.293431
4.01019	0.375772
4.03456	0.4714
4.06098	0.5894
4.08741	0.72047
4.13673	0.989346
4.19019	1.28788
4.21642	1.42444
4.24266	1.5474
4.26714	1.64571
4.29162	1.72491
4.31816	1.7864
4.3447	1.82072
4.36948	1.82771
4.39425	1.81105
4.41855	1.7731
4.44284	1.71579
4.46919	1.63473
4.49555	1.53765
4.52014	1.43615
4.54473	1.32746
4.59804	1.08175
4.65037	0.847784
4.69919	0.651805
4.72566	0.55762
4.75213	0.472735
4.77683	0.402011
4.80154	0.339361
4.82576	0.285454
4.84998	0.238529
4.87627	0.194884
4.90255	0.158076
4.92707	0.129215
4.95159	0.105005
4.97818	0.0833216
5.00476	0.0656937
5.02958	0.0523247
5.0544	0.0414608
5.07874	0.0328428
5.10308	0.0258956
5.12947	0.0199102
5.15587	0.0152308
5.18051	0.0118092
5.20514	0.00911887
5.23184	0.00686021
5.25854	0.00513811
5.28475	0.0038526
5.31096	0.00287732
5.33541	0.00218399
5.35986	0.00165246
5.38638	0.00121704
5.41289	0.000893265
5.43764	0.000667237
5.46239	0.000497026
5.48666	0.000371426
5.51092	0.000276887
5.53725	0.000200789
5.56358	0.00014522
5.58814	0.000107085
5.61271	0.000078798
5.64013	0.0000558177
5.66755	0.0000394444 
  }; 
  %
  \addplot table[x=x,y=dens] { 
  	x  dens 
3.17067	3.83213E-15
3.19518	2.23455E-14
3.21968	1.21125E-13
3.24625	6.99138E-13
3.27282	3.72757E-12
3.29762	1.66038E-11
3.32243	6.93729E-11
3.34675	2.65451E-10
3.37107	9.59299E-10
3.39745	3.63224E-9
3.42383	1.29212E-8
3.44845	4.00062E-8
3.47307	1.17778E-7
3.49975	3.59289E-7
3.52643	1.03747E-6
3.55263	2.79112E-6
3.57882	7.14973E-6
3.60326	0.0000164757
3.62769	0.0000364898
3.65418	0.0000827654
3.68068	0.000179758
3.70542	0.000357075
3.73015	0.000684801
3.7544	0.00125456
3.77865	0.00222701
3.80496	0.00401034
3.83127	0.00697621
3.85581	0.0113475
3.88036	0.0179472
3.90697	0.0286103
3.93358	0.0442293
3.95971	0.0659067
3.98583	0.0955423
4.01019	0.131884
4.03456	0.178028
4.06098	0.240552
4.08741	0.317163
4.11207	0.401866
4.13673	0.499143
4.16346	0.617731
4.19019	0.747992
4.24266	1.02512
4.29162	1.28472
4.31816	1.41469
4.3447	1.53082
4.36948	1.62294
4.39425	1.69633
4.41855	1.74801
4.44284	1.77838
4.46919	1.78676
4.49555	1.77011
4.52014	1.73334
4.54473	1.67804
4.57139	1.6
4.59804	1.50657
4.62421	1.40353
4.65037	1.29291
4.69919	1.07787
4.75213	0.850228
4.80154	0.657859
4.82576	0.573402
4.84998	0.496129
4.87627	0.420622
4.90255	0.353739
4.92707	0.298854
4.95159	0.250835
4.97818	0.205976
5.00476	0.16793
5.02958	0.13791
5.0544	0.112598
5.07874	0.0917946
5.10308	0.0744444
5.12947	0.0589739
5.15587	0.0464514
5.18051	0.0369902
5.20514	0.0293195
5.23184	0.0226771
5.25854	0.0174509
5.28475	0.0134296
5.31096	0.0102883
5.33541	0.00799269
5.35986	0.00618665
5.38638	0.00466781
5.41289	0.00350786
5.43764	0.00267738
5.46239	0.00203696
5.48666	0.00155341
5.51092	0.00118127
5.53725	0.000874888
5.56358	0.000645961
5.58814	0.0004854
5.61271	0.000363838
5.64013	0.000262982
5.66755	0.00018954
  }; 
  %
  \addplot table[x=x,y=dens] { 
  	x  dens 
3.17067	9.05389E-33
3.19518	1.79704E-31
3.21968	3.20991E-30
3.27282	1.17855E-27
3.32243	1.96063E-25
3.37107	2.0814E-23
3.42383	2.28115E-21
3.47307	1.32972E-19
3.52643	7.89211E-18
3.55263	5.2184E-17
3.57882	3.20917E-16
3.60326	1.63991E-15
3.62769	7.90094E-15
3.65418	4.07578E-14
3.68068	1.97095E-13
3.70542	8.11323E-13
3.73015	3.16863E-12
3.7544	1.14671E-11
3.77865	3.95792E-11
3.80496	1.44113E-10
3.83127	4.98061E-10
3.85581	1.5138E-9
3.88036	4.40961E-9
3.90697	1.34135E-8
3.93358	3.89397E-8
3.95971	1.06096E-7
3.98583	2.77157E-7
4.01019	6.54198E-7
4.03456	1.49207E-6
4.06098	3.51432E-6
4.08741	7.96993E-6
4.11207	0.0000165584
4.13673	0.0000333544
4.16346	0.0000688685
4.19019	0.000137436
4.21642	0.000262303
4.24266	0.00048547
4.26714	0.00083932
4.29162	0.0014151
4.31816	0.0024255
4.3447	0.00404365
4.36948	0.00636051
4.39425	0.00978031
4.41855	0.0145968
4.44284	0.0213421
4.46919	0.0315082
4.49555	0.0454706
4.52014	0.0627709
4.54473	0.0850543
4.57139	0.115833
4.59804	0.154537
4.62421	0.201153
4.65037	0.257002
4.67478	0.317777
4.69919	0.386978
4.72566	0.471211
4.75213	0.564177
4.80154	0.756097
4.84998	0.956231
4.90255	1.1676
4.92707	1.25791
4.95159	1.33982
4.97818	1.41681
5.00476	1.47943
5.02958	1.52346
5.0544	1.55263
5.07874	1.56638
5.10308	1.56547
5.12947	1.54842
5.15587	1.51569
5.18051	1.47227
5.20514	1.41793
5.23184	1.34861
5.25854	1.27063
5.31096	1.10095
5.35986	0.934735
5.41289	0.759038
5.46239	0.608178
5.51092	0.477816
5.53725	0.415254
5.56358	0.358584
5.58814	0.310957
5.61271	0.268246
5.64013	0.226099
5.66755	0.189409
  }; 
  %
  \addplot table[x=x,y=dens] { 
  	x  dens 
3.17067	0.0110383
3.19518	0.0194567
3.21968	0.0329484
3.24625	0.0558468
3.27282	0.0906543
3.29762	0.137279
3.32243	0.200815
3.34675	0.282321
3.37107	0.384926
3.39745	0.521165
3.42383	0.682596
3.44845	0.85324
3.47307	1.03857
3.52643	1.45878
3.55263	1.6546
3.57882	1.82968
3.60326	1.96605
3.62769	2.07007
3.65418	2.14134
3.68068	2.16686
3.70542	2.14964
3.73015	2.09527
3.7544	2.00996
3.77865	1.89817
3.80496	1.75382
3.83127	1.59317
3.88036	1.27569
3.93358	0.94523
3.95971	0.798997
3.98583	0.66667
4.01019	0.556804
4.03456	0.460235
4.06098	0.370137
4.08741	0.29435
4.11207	0.235383
4.13673	0.186545
4.16346	0.143591
4.19019	0.109463
4.21642	0.083108
4.24266	0.0625626
4.26714	0.0476504
4.29162	0.0360466
4.31816	0.0264399
4.3447	0.0192521
4.36948	0.0142264
4.39425	0.0104516
4.41855	0.00768333
4.44284	0.00561938
4.46919	0.00398007
4.49555	0.00280352
4.52014	0.00201199
4.54473	0.00143765
4.57139	0.000994051
4.59804	0.000684139
4.62421	0.000472037
4.65037	0.000324378
4.67478	0.000227815
4.69919	0.00015949
4.72566	0.000107975
4.75213	0.0000728535
4.77683	0.0000503144
4.80154	0.000034657
4.82576	0.0000239898
4.84998	0.0000165682
4.87627	0.0000110608
4.90255	7.36662E-6
4.92707	5.03146E-6
4.95159	3.43024E-6
4.97818	2.26008E-6
5.00476	1.48628E-6
5.02958	1.00332E-6
5.0544	6.76326E-7
5.07874	4.58843E-7
5.10308	3.10923E-7
5.12947	2.03601E-7
5.15587	1.33164E-7
5.18051	8.95071E-8
5.20514	6.01101E-8
5.23184	3.90102E-8
5.25854	2.52955E-8
5.28475	1.65202E-8
5.31096	1.07823E-8
5.33541	7.23817E-9
5.35986	4.85692E-9
5.38638	3.14997E-9
5.41289	2.0422E-9
5.43764	1.3623E-9
5.46239	9.08583E-10
5.48666	6.10739E-10
5.51092	4.10497E-10
5.53725	2.66741E-10
5.56358	1.73331E-10
5.58814	1.15934E-10
5.61271	7.75519E-11
5.64013	4.95156E-11
5.66755	3.16227E-11
  }; 
  %
  \addplot table[x=x,y=dens] { 
  	x  dens 
3.17067	4.76911E-7
3.19518	1.38589E-6
3.21968	3.81544E-6
3.24625	0.0000107865
3.27282	0.0000287596
3.29762	0.0000683006
3.32243	0.000154746
3.34675	0.000330181
3.37107	0.000675599
3.39745	0.00140344
3.42383	0.00278555
3.44845	0.00507705
3.47307	0.00892041
3.49975	0.0157891
3.52643	0.026856
3.55263	0.0435883
3.57882	0.0682857
3.60326	0.100661
3.62769	0.144215
3.65418	0.206471
3.68068	0.286568
3.70542	0.378822
3.73015	0.488389
3.7544	0.611998
3.77865	0.749932
3.83127	1.08343
3.88036	1.40294
3.90697	1.56287
3.93358	1.70408
3.95971	1.81826
3.98583	1.90335
4.01019	1.95361
4.03456	1.97451
4.06098	1.96413
4.08741	1.9212
4.11207	1.85466
4.13673	1.76626
4.16346	1.65064
4.19019	1.51994
4.24266	1.24092
4.29162	0.981034
4.31816	0.848846
4.3447	0.726012
4.36948	0.621149
4.39425	0.526487
4.41855	0.443806
4.44284	0.371001
4.46919	0.302697
4.49555	0.244733
4.52014	0.199122
4.54473	0.160828
4.57139	0.126583
4.59804	0.0988403
4.62421	0.0769534
4.65037	0.0594913
4.67478	0.0465096
4.69919	0.0361548
4.72566	0.0273447
4.75213	0.0205544
4.77683	0.0156626
4.80154	0.0118761
4.82576	0.00901261
4.84998	0.00680955
4.87627	0.00499984
4.90255	0.00365366
4.92707	0.00271544
4.95159	0.00201048
4.97818	0.00144544
5.00476	0.00103495
5.02958	0.000754939
5.0544	0.00054889
5.07874	0.000400363
5.10308	0.000291188
5.12947	0.000205509
5.15587	0.000144593
5.18051	0.000103871
5.20514	0.0000744353
5.23184	0.0000517387
5.25854	0.0000358694
5.28475	0.0000249741
5.31096	0.0000173491
5.33541	0.0000123269
5.35986	8.74304E-6
5.38638	6.01276E-6
5.41289	4.12742E-6
5.43764	2.90031E-6
5.46239	2.0351E-6
5.48666	1.43607E-6
5.51092	1.01212E-6
5.53725	6.91528E-7
5.56358	4.71882E-7
5.58814	3.29981E-7
5.61271	2.30528E-7
5.64013	1.54305E-7
5.66755	1.03178E-7
  }; 
  %
  \addplot table[x=x,y=dens] { 
  	x  dens 
3.17067	0.00220776
3.19518	0.00389161
3.21968	0.00659044
3.24625	0.0111715
3.27282	0.0181366
3.29762	0.0274696
3.32243	0.040194
3.34675	0.0565303
3.37107	0.0771202
3.39745	0.104514
3.42383	0.137076
3.44845	0.171664
3.47307	0.209499
3.52643	0.297132
3.55263	0.339648
3.57882	0.379616
3.60326	0.413389
3.62769	0.442949
3.65418	0.469751
3.68068	0.491056
3.70542	0.506363
3.73015	0.517907
3.7544	0.526371
3.77865	0.53286
3.80496	0.538616
3.83127	0.543918
3.85581	0.549087
3.88036	0.554894
3.90697	0.562218
3.93358	0.57081
3.95971	0.580592
3.98583	0.591799
4.01019	0.603615
4.03456	0.616834
4.06098	0.632845
4.08741	0.650638
4.11207	0.668796
4.13673	0.688265
4.19019	0.733082
4.21642	0.754905
4.24266	0.775298
4.26714	0.792041
4.29162	0.805625
4.31816	0.815759
4.3447	0.82017
4.36948	0.818479
4.39425	0.81082
4.41855	0.797439
4.44284	0.778427
4.46919	0.751933
4.49555	0.720153
4.52014	0.686679
4.54473	0.650564
4.59804	0.568476
4.65037	0.491502
4.67478	0.459376
4.69919	0.430593
4.72566	0.403682
4.75213	0.381554
4.77683	0.365266
4.80154	0.353046
4.82576	0.344708
4.84998	0.339543
4.87627	0.336909
4.90255	0.336616
4.92707	0.33774
4.95159	0.339535
4.97818	0.341511
5.00476	0.342818
5.02958	0.34289
5.0544	0.341447
5.07874	0.338284
5.10308	0.333221
5.12947	0.325503
5.15587	0.315503
5.18051	0.304235
5.20514	0.291289
5.23184	0.275641
5.25854	0.25865
5.31096	0.222827
5.35986	0.188517
5.41289	0.152689
5.46239	0.122143
5.51092	0.095855
5.53725	0.0832662
5.56358	0.0718751
5.58814	0.0623099
5.61271	0.0537378
5.64013	0.0452836
5.66755	0.0379277
  }; 
  %
  \addplot table[x=x,y=dens] { 
  	x  dens 
  	3.1831588899999996  0.03541944954473747 
  	3.2081276699999997  0.04085878094512657 
  	3.2330964499999997  0.04696621728929074 
  	3.2580652299999997  0.053795924416809755 
  	3.2830340099999997  0.061402190795344445 
  	3.3080027899999997  0.06983870844646409 
  	3.3329715699999993  0.07915777246889782 
  	3.3579403499999994  0.08940940549185124 
  	3.3829091299999994  0.10064041591176304 
  	3.4078779099999994  0.11289340132854127 
  	3.4328466899999994  0.12620571112788104 
  	3.4578154699999994  0.1406083845751079 
  	3.4827842499999995  0.15612508300829162 
  	3.5077530299999995  0.1727710366594544 
  	3.5327218099999995  0.19055202820834674 
  	3.5576905899999995  0.20946343630473932 
  	3.5826593699999996  0.22948936291256775 
  	3.6076281499999996  0.2506018683729689 
  	3.632596929999999  0.2727603375091096 
  	3.657565709999999  0.29591099887600086 
  	3.682534489999999  0.3199866173850802 
  	3.7075032699999992  0.344906378018264 
  	3.7324720499999993  0.3705759752234374 
  	3.7574408299999993  0.3968879189073853 
  	3.7824096099999993  0.42372206378931676 
  	3.8073783899999993  0.4509463643419893 
  	3.8323471699999994  0.4784178527391421 
  	3.8573159499999994  0.5059838322710104 
  	3.8822847299999994  0.5334832737177196 
  	3.9072535099999994  0.5607483973205487 
  	3.9322222899999995  0.5876064184039985 
  	3.9571910699999995  0.6138814305089652 
  	3.9821598499999995  0.6393963962274629 
  	4.0071286299999995  0.6639752128922177 
  	4.0320974099999995  0.6874448179666895 
  	4.057066189999999  0.7096372974749225 
  	4.08203497  0.7303919601591247 
  	4.107003749999999  0.7495573402794092 
  	4.131972529999999  0.7669930930760204 
  	4.156941309999999  0.7825717488716967 
  	4.181910089999999  0.7961802945493632 
  	4.206878869999999  0.8077215546231752 
  	4.231847649999999  0.8171153482333857 
  	4.256816429999999  0.8242994030233242 
  	4.281785209999999  0.8292300118761453 
  	4.306753989999999  0.8318824237616231 
  	4.331722769999999  0.8322509653306361 
  	4.356691549999999  0.8303488952567076 
  	4.381660329999999  0.8262079985230896 
  	4.406629109999999  0.8198779327649174 
  	4.431597889999999  0.8114253432805654 
  	4.456566669999999  0.8009327673288784 
  	4.481535449999999  0.7884973517410868 
  	4.506504229999999  0.7742294106420703 
  	4.531473009999999  0.7582508521473438 
  	4.556441789999999  0.7406935042616141 
  	4.581410569999999  0.7216973708500272 
  	4.606379349999999  0.7014088485044783 
  	4.631348129999999  0.6799789344185198 
  	4.656316909999999  0.657561454070704 
  	4.681285689999999  0.6343113356614026 
  	4.706254469999999  0.6103829559284821 
  	4.731223249999999  0.5859285792664803 
  	4.756192029999999  0.5610969090801613 
  	4.7811608099999985  0.5360317671062783 
  	4.806129589999999  0.510870913125332 
  	4.831098369999999  0.48574501414500676 
  	4.8560671499999986  0.46077676884545243 
  	4.881035929999999  0.43608018990974223 
  	4.906004709999999  0.41176004387986215 
  	4.930973489999999  0.3879114454359629 
  	4.955942269999999  0.36461960053522274 
  	4.980911049999999  0.34195969069869997 
  	5.005879829999999  0.319996888922119 
  	5.030848609999999  0.2987864962187401 
  	5.055817389999999  0.27837418668244285 
  	5.080786169999999  0.2587963481776362 
  	5.105754949999999  0.24008050530601197 
  	5.130723729999999  0.22224581114573966 
  	5.155692509999999  0.20530359438043408 
  	5.180661289999998  0.18925794880188443 
  	5.205630069999999  0.17410635274951383 
  	5.230598849999999  0.1598403068050655 
  	5.255567629999998  0.14644597895853026 
  	5.280536409999998  0.13390484746611075 
  	5.305505189999998  0.12219433269931725 
  	5.330473969999998  0.11128841040567204 
  	5.355442749999998  0.10115819993675756 
  	5.380411529999998  0.09177252212392818 
  	5.405380309999998  0.0830984225728834 
  	5.430349089999998  0.07510165718780594 
  	5.455317869999998  0.06774713770837665 
  	5.480286649999998  0.06099933593705869 
  	5.505255429999998  0.05482264614161529 
  	5.530224209999998  0.049181705833250086 
  	5.555192989999998  0.044041675741716725 
  	5.580161769999998  0.03936848033526607 
  	5.605130549999998  0.035129010667174894 
  	5.630099329999998  0.031291291676020475 
  	5.655068109999998  0.02782461632879747 
  }; 
  
  %
 
  \end{axis}
  \end{tikzpicture}
  \begin{tikzpicture} 
  \begin{axis}[ 
  	cycle list name=canti,
  xlabel=$x$,
  ylabel=,
  grid,
  xmin=3.17,
  xmax=5.67,
  width = 6cm,
  height = 4cm,
  ] 
  \addplot table[x=x,y=dens] { 
  	x  dens 
3.17067  0.0 
4.65	0.0
4.65037	5.32317
4.67478	27.8561
4.69919	6.45197
4.72566	0.297517
4.75213	0.00546088
4.77683	0.0000737473
4.80154	6.55342E-7
4.82576	4.60985E-9
4.84998	2.47513E-11
4.87627	6.55824E-14
4.90255	1.37094E-16
4.92707	3.58455E-19
4.95159	7.94648E-22
5.00476	8.45136E-28
5.0544	1.29192E-33
5.10308	1.63249E-39
5.15587	4.1333E-46
5.20514	1.94381E-52
5.25854	1.81405E-59
5.31096	1.55793E-66
5.35986	2.89835E-73
5.41289	1.04674E-80
5.46239	8.73693E-88
5.51092	7.6574E-95
5.56358	1.25685E-102
5.61271	5.22614E-110
5.64013	3.55135E-114
5.66755	2.23224E-118
  }; 
  %
  \addplot table[x=x,y=dens] { 
  	x  dens 
3.17067  0.0 
4.469 	0.0
4.46919	29.9629
4.49555	9.95006
4.52014	0.493799
4.54473	0.00879738
4.57139	0.0000519815
4.59804	1.71775E-7
4.62421	4.06277E-10
4.65037	6.74178E-13
4.67478	1.31286E-15
4.69919	2.03623E-18
4.72566	1.45561E-21
4.75213	8.40597E-25
4.80154	4.55658E-31
4.84998	1.86417E-37
4.90255	1.24695E-44
4.95159	1.57027E-51
5.00476	3.23909E-59
5.0544	1.42575E-66
5.10308	6.07078E-74
5.15587	4.16456E-82
5.20514	7.06363E-90
5.25854	1.86172E-98
5.31096	4.92428E-107
5.35986	3.60407E-115
5.41289	3.9179E-124
5.46239	1.24158E-132
5.51092	4.38687E-141
5.56358	2.19133E-150
5.61271	3.46548E-159
5.64013	3.77735E-164
5.66755	3.79238E-169
  }; 
  %
  \addplot table[x=x,y=dens] { 
  	x  dens 
 3.17067  0.0 
 5.18	0.0
 5.18051	2.68493
5.20514	32.3233
5.23184	2.61168
5.25854	0.0306639
5.28475	0.000135162
5.31096	2.92368E-7
5.33541	5.90183E-10
5.35986	8.22268E-13
5.38638	4.64834E-16
5.41289	1.93194E-19
5.43764	1.0577E-22
5.46239	4.70743E-26
5.51092	7.44797E-33
5.56358	1.56628E-40
5.61271	6.12678E-48
5.64013	3.63149E-52
5.66755	1.8667E-56
  }; 
  %
  \addplot table[x=x,y=dens] { 
  	x  dens 
    3.17067  0.0 
    3.778 	0.0
 3.77865	0.0131306
3.80496	19.4247
3.83127	16.6565
3.85581	2.81143
3.88036	0.198969
3.90697	0.00584195
3.93358	0.000103795
3.95971	1.35601E-6
3.98583	1.30049E-8
4.01019	1.34625E-10
4.03456	1.14148E-12
4.06098	5.28906E-15
4.08741	2.03064E-17
4.11207	9.67573E-20
4.13673	4.01695E-22
4.19019	1.79977E-27
4.24266	6.01151E-33
4.29162	3.08067E-38
4.3447	3.72121E-44
4.39425	7.68402E-50
4.44284	1.48212E-55
4.49555	6.61874E-62
4.54473	5.74805E-68
4.59804	1.11745E-74
4.65037	2.0915E-81
4.69919	8.40958E-88
4.75213	7.25767E-95
4.80154	1.4093E-101
4.84998	2.89262E-108
4.90255	1.21767E-115
4.95159	1.24166E-122
5.00476	2.49352E-130
5.0544	1.25905E-137
5.10308	6.99899E-145
5.15587	7.30155E-153
5.20514	2.04129E-160
5.25854	1.02781E-168
5.31096	5.65329E-177
5.35986	8.86032E-185
5.41289	2.37884E-193
5.46239	1.87686E-201
5.51092	1.72069E-209
5.56358	2.57299E-218
5.61271	1.19545E-226
5.64013	2.43018E-231
5.66755	4.61902E-236
  }; 
  %
  \addplot table[x=x,y=dens] { 
  	x  dens 
   3.17067  0.0 
   3.17	0.0
 3.88036	0.140937
3.90697	28.6437
3.93358	9.35235
3.95971	0.510137
3.98583	0.010258
4.01019	0.000145768
4.03456	1.33036E-6
4.06098	5.47622E-9
4.08741	1.60002E-11
4.11207	5.30019E-14
4.13673	1.40573E-16
4.16346	1.81592E-19
4.19019	1.902E-22
4.24266	1.55793E-28
4.29162	1.84027E-34
4.3447	3.95195E-41
4.39425	1.46287E-47
4.44284	4.80423E-54
4.49555	2.89114E-61
4.54473	3.6062E-68
4.59804	7.98535E-76
4.65037	1.66098E-83
4.69919	8.15753E-91
4.75213	6.81789E-99
4.80154	1.4266E-106
4.84998	3.15706E-114
4.90255	1.1316E-122
4.95159	1.11105E-130
5.00476	1.68819E-139
5.0544	7.3489E-148
5.10308	3.56146E-156
5.15587	2.53018E-165
5.20514	5.55261E-174
5.25854	1.70486E-183
5.31096	5.78473E-193
5.35986	6.51403E-202
5.41289	9.69622E-212
5.46239	4.96506E-221
5.51092	3.01932E-230
5.56358	2.29546E-240
5.61271	6.40327E-250
5.64013	2.67135E-255
5.66755	1.03186E-260
  }; 
  %
  \addplot table[x=x,y=dens] { 
  	x  dens 
   3.17067  0.0 
 3.7544	0.
3.77865	0.00262612
3.80496	3.88494
3.83127	3.33129
3.85581	0.562286
3.88036	0.0679812
3.90697	5.72991
3.93358	1.87049
3.95971	0.102028
3.98583	0.0020516
4.01019	0.0000291536
4.03456	2.66072E-7
4.06098	1.09524E-9
4.08741	3.20004E-12
4.11207	1.06004E-14
4.13673	2.81147E-17
4.16346	3.63185E-20
4.19019	3.80403E-23
4.24266	3.11597E-29
4.29162	3.68115E-35
4.3447	7.91134E-42
4.39425	2.9411E-48
4.41855	1.78317E-51
4.44284	9.90489E-55
4.46919	5.99258
4.49555	1.99001
4.52014	0.0987598
4.54473	0.00175948
4.57139	0.0000103963
4.59804	3.4355E-8
4.62421	8.12554E-11
4.65037	1.06463
4.67478	5.57122
4.69919	1.29039
4.72566	0.0595035
4.75213	0.00109218
4.77683	0.0000147495
4.80154	1.31068E-7
4.82576	9.2197E-10
4.84998	4.95025E-12
4.87627	1.31165E-14
4.90255	2.74188E-17
4.92707	7.16911E-20
4.95159	1.5893E-22
5.00476	1.69027E-28
5.0544	2.58384E-34
5.10308	3.26499E-40
5.12947	1.73888E-43
5.15587	8.26661E-47
5.18051	0.536985
5.20514	6.46467
5.23184	0.522336
5.25854	0.00613279
5.28475	0.0000270325
5.31096	5.84736E-8
5.33541	1.18037E-10
5.35986	1.64454E-13
5.38638	9.29669E-17
5.41289	3.86388E-20
5.43764	2.1154E-23
5.46239	9.41486E-27
5.51092	1.48959E-33
5.56358	3.13256E-41
5.61271	1.22536E-48
5.64013	7.26298E-53
5.66755	3.73341E-57
  }; 
 \addplot table[x=x,y=dens] { 
	x  dens 
  	3.1831588899999996  0.03541944954473747 
  	3.2081276699999997  0.04085878094512657 
  	3.2330964499999997  0.04696621728929074 
  	3.2580652299999997  0.053795924416809755 
  	3.2830340099999997  0.061402190795344445 
  	3.3080027899999997  0.06983870844646409 
  	3.3329715699999993  0.07915777246889782 
  	3.3579403499999994  0.08940940549185124 
  	3.3829091299999994  0.10064041591176304 
  	3.4078779099999994  0.11289340132854127 
  	3.4328466899999994  0.12620571112788104 
  	3.4578154699999994  0.1406083845751079 
  	3.4827842499999995  0.15612508300829162 
  	3.5077530299999995  0.1727710366594544 
  	3.5327218099999995  0.19055202820834674 
  	3.5576905899999995  0.20946343630473932 
  	3.5826593699999996  0.22948936291256775 
  	3.6076281499999996  0.2506018683729689 
  	3.632596929999999  0.2727603375091096 
  	3.657565709999999  0.29591099887600086 
  	3.682534489999999  0.3199866173850802 
  	3.7075032699999992  0.344906378018264 
  	3.7324720499999993  0.3705759752234374 
  	3.7574408299999993  0.3968879189073853 
  	3.7824096099999993  0.42372206378931676 
  	3.8073783899999993  0.4509463643419893 
  	3.8323471699999994  0.4784178527391421 
  	3.8573159499999994  0.5059838322710104 
  	3.8822847299999994  0.5334832737177196 
  	3.9072535099999994  0.5607483973205487 
  	3.9322222899999995  0.5876064184039985 
  	3.9571910699999995  0.6138814305089652 
  	3.9821598499999995  0.6393963962274629 
  	4.0071286299999995  0.6639752128922177 
  	4.0320974099999995  0.6874448179666895 
  	4.057066189999999  0.7096372974749225 
  	4.08203497  0.7303919601591247 
  	4.107003749999999  0.7495573402794092 
  	4.131972529999999  0.7669930930760204 
  	4.156941309999999  0.7825717488716967 
  	4.181910089999999  0.7961802945493632 
  	4.206878869999999  0.8077215546231752 
  	4.231847649999999  0.8171153482333857 
  	4.256816429999999  0.8242994030233242 
  	4.281785209999999  0.8292300118761453 
  	4.306753989999999  0.8318824237616231 
  	4.331722769999999  0.8322509653306361 
  	4.356691549999999  0.8303488952567076 
  	4.381660329999999  0.8262079985230896 
  	4.406629109999999  0.8198779327649174 
  	4.431597889999999  0.8114253432805654 
  	4.456566669999999  0.8009327673288784 
  	4.481535449999999  0.7884973517410868 
  	4.506504229999999  0.7742294106420703 
  	4.531473009999999  0.7582508521473438 
  	4.556441789999999  0.7406935042616141 
  	4.581410569999999  0.7216973708500272 
  	4.606379349999999  0.7014088485044783 
  	4.631348129999999  0.6799789344185198 
  	4.656316909999999  0.657561454070704 
  	4.681285689999999  0.6343113356614026 
  	4.706254469999999  0.6103829559284821 
  	4.731223249999999  0.5859285792664803 
  	4.756192029999999  0.5610969090801613 
  	4.7811608099999985  0.5360317671062783 
  	4.806129589999999  0.510870913125332 
  	4.831098369999999  0.48574501414500676 
  	4.8560671499999986  0.46077676884545243 
  	4.881035929999999  0.43608018990974223 
  	4.906004709999999  0.41176004387986215 
  	4.930973489999999  0.3879114454359629 
  	4.955942269999999  0.36461960053522274 
  	4.980911049999999  0.34195969069869997 
  	5.005879829999999  0.319996888922119 
  	5.030848609999999  0.2987864962187401 
  	5.055817389999999  0.27837418668244285 
  	5.080786169999999  0.2587963481776362 
  	5.105754949999999  0.24008050530601197 
  	5.130723729999999  0.22224581114573966 
  	5.155692509999999  0.20530359438043408 
  	5.180661289999998  0.18925794880188443 
  	5.205630069999999  0.17410635274951383 
  	5.230598849999999  0.1598403068050655 
  	5.255567629999998  0.14644597895853026 
  	5.280536409999998  0.13390484746611075 
  	5.305505189999998  0.12219433269931725 
  	5.330473969999998  0.11128841040567204 
  	5.355442749999998  0.10115819993675756 
  	5.380411529999998  0.09177252212392818 
  	5.405380309999998  0.0830984225728834 
  	5.430349089999998  0.07510165718780594 
  	5.455317869999998  0.06774713770837665 
  	5.480286649999998  0.06099933593705869 
  	5.505255429999998  0.05482264614161529 
  	5.530224209999998  0.049181705833250086 
  	5.555192989999998  0.044041675741716725 
  	5.580161769999998  0.03936848033526607 
  	5.605130549999998  0.035129010667174894 
  	5.630099329999998  0.031291291676020475 
  	5.655068109999998  0.02782461632879747 
}; 
  
  %

  \end{axis}
  \end{tikzpicture}
  \begin{tikzpicture} 
  \begin{axis}[ 
  	cycle list name=canti,
  xlabel=$x$,
  ylabel=,
  grid,
  xmin=3.17,
  xmax=5.67,
  width = 6cm,
  height = 4cm,
  ] 
  \addplot table[x=x,y=dens] { 
  	x  dens 
 3.17067	0.038352
3.19518	0.0445746
3.21968	0.0516193
3.24625	0.0602724
3.27282	0.0700772
3.29762	0.0803585
3.32243	0.0918077
3.34675	0.10424
3.37107	0.117936
3.39745	0.134295
3.42383	0.152287
3.47307	0.190434
3.52643	0.238697
3.57882	0.293079
3.62769	0.349729
3.68068	0.416851
3.73015	0.483756
3.77865	0.551926
3.83127	0.62684
3.88036	0.695531
3.93358	0.766086
3.98583	0.828717
4.01019	0.854927
4.03456	0.878885
4.06098	0.902056
4.08741	0.922037
4.11207	0.937601
4.13673	0.950021
4.16346	0.95978
4.19019	0.965574
4.21642	0.967339
4.24266	0.965194
4.26714	0.959692
4.29162	0.950871
4.31816	0.937665
4.3447	0.92083
4.36948	0.902015
4.39425	0.880409
4.41855	0.856732
4.44284	0.830814
4.49555	0.768062
4.54473	0.703391
4.59804	0.629295
4.65037	0.555105
4.69919	0.486738
4.75213	0.415493
4.80154	0.353177
4.84998	0.297032
4.90255	0.242377
4.95159	0.1976
5.00476	0.155865
5.02958	0.138737
5.0544	0.123049
5.07874	0.109012
5.10308	0.0962432
5.12947	0.0837513
5.15587	0.0725858
5.18051	0.0632802
5.20514	0.054973
5.23184	0.0470097
5.25854	0.0400336
5.28475	0.0340551
5.31096	0.0288539
5.33541	0.0246319
5.35986	0.0209548
5.38638	0.0175163
5.41289	0.0145824
5.43764	0.0122434
5.46239	0.010243
5.48666	0.00856993
5.51092	0.00714565
5.53725	0.00584416
5.56358	0.00476052
5.58814	0.00391712
5.61271	0.00321187
5.64013	0.00256287
5.66755	0.00203611
  }; 
  %
  \addplot table[x=x,y=dens] { 
  	x  dens 
3.17067	0.0103501
3.19518	0.012215
3.21968	0.0143703
3.24625	0.0170774
3.27282	0.0202193
3.29762	0.0235935
3.32243	0.027442
3.34675	0.0317242
3.37107	0.036561
3.39745	0.0424959
3.42383	0.0492145
3.44845	0.0562548
3.47307	0.0640987
3.49975	0.0735754
3.52643	0.08414
3.55263	0.0956404
3.57882	0.108324
3.60326	0.121277
3.62769	0.135357
3.68068	0.169941
3.73015	0.207411
3.77865	0.249063
3.83127	0.299608
3.88036	0.351409
3.93358	0.41188
3.98583	0.47455
4.03456	0.534727
4.08741	0.60028
4.13673	0.660067
4.19019	0.721312
4.24266	0.775788
4.26714	0.798708
4.29162	0.819775
4.31816	0.840316
4.3447	0.858259
4.36948	0.872503
4.39425	0.884191
4.41855	0.893064
4.44284	0.899298
4.46919	0.903013
4.49555	0.903519
4.52014	0.901091
4.54473	0.895888
4.57139	0.887171
4.59804	0.875346
4.62421	0.860837
4.65037	0.843605
4.67478	0.825231
4.69919	0.804798
4.72566	0.780518
4.75213	0.754261
4.80154	0.700826
4.84998	0.644263
4.90255	0.580113
4.95159	0.519377
5.00476	0.454348
5.0544	0.395806
5.10308	0.341543
5.15587	0.287111
5.20514	0.241052
5.25854	0.196674
5.31096	0.158801
5.35986	0.128444
5.38638	0.113907
5.41289	0.100653
5.43764	0.0893864
5.46239	0.0791333
5.48666	0.0700119
5.51092	0.0617565
5.53725	0.0537145
5.56358	0.0465553
5.58814	0.0406086
5.61271	0.0353129
5.64013	0.0301034
5.66755	0.0255644
  }; 
  %
  \addplot table[x=x,y=dens] { 
  	x  dens 
3.17067	0.0172313
3.19518	0.0201751
3.21968	0.0235459
3.24625	0.0277385
3.27282	0.0325545
3.29762	0.037675
3.32243	0.0434579
3.34675	0.049829
3.37107	0.0569541
3.39745	0.0656052
3.42383	0.0752906
3.44845	0.0853294
3.47307	0.0963951
3.49975	0.109613
3.52643	0.124172
3.55263	0.139829
3.57882	0.156887
3.62769	0.192581
3.68068	0.237095
3.73015	0.284046
3.77865	0.33484
3.83127	0.394671
3.88036	0.454055
3.93358	0.520992
3.98583	0.587663
4.03456	0.649006
4.08741	0.712613
4.13673	0.767305
4.16346	0.794393
4.19019	0.81934
4.21642	0.841501
4.24266	0.861126
4.26714	0.876968
4.29162	0.890281
4.31816	0.901717
4.3447	0.909912
4.36948	0.914558
4.39425	0.916256
4.41855	0.915043
4.44284	0.910992
4.46919	0.903434
4.49555	0.892665
4.52014	0.879822
4.54473	0.864404
4.57139	0.844952
4.59804	0.822853
4.62421	0.798807
4.65037	0.772672
4.69919	0.719218
4.75213	0.656053
4.80154	0.594145
4.84998	0.532448
4.90255	0.46617
4.95159	0.406431
5.00476	0.345317
5.0544	0.292623
5.10308	0.245667
5.15587	0.200372
5.20514	0.163491
5.23184	0.14565
5.25854	0.129271
5.28475	0.114566
5.31096	0.101168
5.33541	0.0897934
5.35986	0.079448
5.38638	0.0693264
5.41289	0.0602714
5.43764	0.0527126
5.46239	0.0459537
5.48666	0.0400444
5.51092	0.0347873
5.53725	0.0297571
5.56358	0.0253618
5.58814	0.0217764
5.61271	0.0186387
5.64013	0.0156086
5.66755	0.0130196
  }; 
  %
  \addplot table[x=x,y=dens] { 
  	x  dens 
3.17067	0.031784
3.19518	0.0368459
3.21968	0.0425729
3.24625	0.0496058
3.27282	0.0575763
3.29762	0.0659393
3.32243	0.0752615
3.34675	0.085398
3.37107	0.0965848
3.39745	0.109977
3.42383	0.124748
3.44845	0.139834
3.47307	0.156223
3.52643	0.19638
3.57882	0.242115
3.62769	0.290355
3.68068	0.348377
3.73015	0.407253
3.77865	0.468464
3.83127	0.537416
3.88036	0.602556
3.93358	0.671988
3.98583	0.736696
4.03456	0.791878
4.06098	0.819036
4.08741	0.843884
4.11207	0.86476
4.13673	0.883199
4.16346	0.900223
4.19019	0.913986
4.21642	0.924171
4.24266	0.930946
4.26714	0.934124
4.29162	0.934236
4.31816	0.930892
4.3447	0.923983
4.36948	0.914383
4.39425	0.90184
4.41855	0.886806
4.44284	0.869203
4.46919	0.847394
4.49555	0.822991
4.52014	0.798105
4.54473	0.771409
4.59804	0.708451
4.65037	0.641839
4.69919	0.577493
4.75213	0.507434
4.80154	0.443532
4.84998	0.383688
4.90255	0.323111
4.95159	0.271514
5.00476	0.22152
5.0544	0.180639
5.10308	0.145965
5.12947	0.129327
5.15587	0.114149
5.18051	0.101244
5.20514	0.0895009
5.23184	0.078014
5.25854	0.0677363
5.28475	0.0587405
5.31096	0.050748
5.33541	0.0441257
5.35986	0.0382421
5.38638	0.0326246
5.41289	0.0277252
5.43764	0.0237349
5.46239	0.0202508
5.48666	0.0172758
5.51092	0.0146903
5.53725	0.0122761
5.56358	0.0102197
5.58814	0.00858328
5.61271	0.0071851
5.64013	0.00586861
5.66755	0.00477363
  }; 
  %
  \addplot table[x=x,y=dens] { 
  	x  dens 
3.17067	0.00500819
3.19518	0.00595271
3.21968	0.00705429
3.24625	0.00845148
3.27282	0.0100901
3.29762	0.0118682
3.32243	0.0139173
3.34675	0.0162213
3.37107	0.0188516
3.39745	0.0221162
3.42383	0.0258574
3.44845	0.0298254
3.47307	0.0343
3.49975	0.0397761
3.52643	0.0459656
3.55263	0.0527986
3.57882	0.0604437
3.60326	0.0683623
3.62769	0.0770925
3.65418	0.0875342
3.68068	0.0990496
3.70542	0.11082
3.73015	0.123618
3.77865	0.15184
3.83127	0.187347
3.88036	0.225176
3.93358	0.271239
3.98583	0.321285
4.03456	0.371781
4.08741	0.429925
4.13673	0.486398
4.19019	0.548648
4.24266	0.609263
4.29162	0.663818
4.3447	0.719033
4.39425	0.765269
4.41855	0.785601
4.44284	0.804181
4.46919	0.822184
4.49555	0.837782
4.52014	0.850035
4.54473	0.859957
4.57139	0.867983
4.59804	0.873092
4.62421	0.875221
4.65037	0.874468
4.67478	0.871177
4.69919	0.865415
4.72566	0.856432
4.75213	0.844693
4.77683	0.83135
4.80154	0.815821
4.82576	0.798603
4.84998	0.779549
4.87627	0.756971
4.90255	0.732613
4.95159	0.683153
5.00476	0.625099
5.0544	0.568352
5.10308	0.511816
5.15587	0.451006
5.20514	0.395987
5.25854	0.339431
5.31096	0.287921
5.35986	0.244032
5.41289	0.201348
5.46239	0.166235
5.51092	0.136205
5.56358	0.108338
5.58814	0.0969232
5.61271	0.0864617
5.64013	0.0758527
5.66755	0.0663068
  }; 
  %
  \addplot table[x=x,y=dens] { 
  	x  dens 
3.17067	0.0205451
3.19518	0.0239527
3.21968	0.0278325
3.24625	0.0326291
3.27282	0.0381035
3.29762	0.0438869
3.32243	0.0503773
3.34675	0.0574824
3.37107	0.0653775
3.39745	0.0748979
3.42383	0.0854796
3.44845	0.0963697
3.47307	0.10829
3.49975	0.122424
3.52643	0.137871
3.57882	0.17217
3.62769	0.209023
3.68068	0.254263
3.73015	0.301217
3.77865	0.351227
3.83127	0.409177
3.88036	0.465745
3.93358	0.528437
3.98583	0.589782
4.03456	0.645255
4.08741	0.701748
4.13673	0.749398
4.16346	0.772631
4.19019	0.793772
4.21642	0.812303
4.24266	0.828463
4.26714	0.841274
4.29162	0.851796
4.31816	0.860528
4.3447	0.866403
4.36948	0.86927
4.39425	0.869593
4.41855	0.867449
4.44284	0.862898
4.46919	0.855295
4.49555	0.845004
4.52014	0.833074
4.54473	0.81901
4.57139	0.801501
4.59804	0.781808
4.62421	0.760535
4.65037	0.737538
4.69919	0.690733
4.75213	0.635587
4.80154	0.5815
4.84998	0.527396
4.90255	0.468877
4.95159	0.415615
5.00476	0.36043
5.0544	0.312094
5.10308	0.268247
5.15587	0.225045
5.20514	0.189001
5.25854	0.154629
5.31096	0.125498
5.35986	0.102224
5.38638	0.0910824
5.41289	0.0809162
5.43764	0.0722593
5.46239	0.0643632
5.48666	0.0573173
5.51092	0.0509169
5.53725	0.0446536
5.56358	0.039047
5.58814	0.0343617
5.61271	0.0301621
5.64013	0.0259992
5.66755	0.0223401
  }; 
  %
  \addplot table[x=x,y=dens] { 
  	x  dens 
  	3.1831588899999996  0.03541944954473747 
  	3.2081276699999997  0.04085878094512657 
  	3.2330964499999997  0.04696621728929074 
  	3.2580652299999997  0.053795924416809755 
  	3.2830340099999997  0.061402190795344445 
  	3.3080027899999997  0.06983870844646409 
  	3.3329715699999993  0.07915777246889782 
  	3.3579403499999994  0.08940940549185124 
  	3.3829091299999994  0.10064041591176304 
  	3.4078779099999994  0.11289340132854127 
  	3.4328466899999994  0.12620571112788104 
  	3.4578154699999994  0.1406083845751079 
  	3.4827842499999995  0.15612508300829162 
  	3.5077530299999995  0.1727710366594544 
  	3.5327218099999995  0.19055202820834674 
  	3.5576905899999995  0.20946343630473932 
  	3.5826593699999996  0.22948936291256775 
  	3.6076281499999996  0.2506018683729689 
  	3.632596929999999  0.2727603375091096 
  	3.657565709999999  0.29591099887600086 
  	3.682534489999999  0.3199866173850802 
  	3.7075032699999992  0.344906378018264 
  	3.7324720499999993  0.3705759752234374 
  	3.7574408299999993  0.3968879189073853 
  	3.7824096099999993  0.42372206378931676 
  	3.8073783899999993  0.4509463643419893 
  	3.8323471699999994  0.4784178527391421 
  	3.8573159499999994  0.5059838322710104 
  	3.8822847299999994  0.5334832737177196 
  	3.9072535099999994  0.5607483973205487 
  	3.9322222899999995  0.5876064184039985 
  	3.9571910699999995  0.6138814305089652 
  	3.9821598499999995  0.6393963962274629 
  	4.0071286299999995  0.6639752128922177 
  	4.0320974099999995  0.6874448179666895 
  	4.057066189999999  0.7096372974749225 
  	4.08203497  0.7303919601591247 
  	4.107003749999999  0.7495573402794092 
  	4.131972529999999  0.7669930930760204 
  	4.156941309999999  0.7825717488716967 
  	4.181910089999999  0.7961802945493632 
  	4.206878869999999  0.8077215546231752 
  	4.231847649999999  0.8171153482333857 
  	4.256816429999999  0.8242994030233242 
  	4.281785209999999  0.8292300118761453 
  	4.306753989999999  0.8318824237616231 
  	4.331722769999999  0.8322509653306361 
  	4.356691549999999  0.8303488952567076 
  	4.381660329999999  0.8262079985230896 
  	4.406629109999999  0.8198779327649174 
  	4.431597889999999  0.8114253432805654 
  	4.456566669999999  0.8009327673288784 
  	4.481535449999999  0.7884973517410868 
  	4.506504229999999  0.7742294106420703 
  	4.531473009999999  0.7582508521473438 
  	4.556441789999999  0.7406935042616141 
  	4.581410569999999  0.7216973708500272 
  	4.606379349999999  0.7014088485044783 
  	4.631348129999999  0.6799789344185198 
  	4.656316909999999  0.657561454070704 
  	4.681285689999999  0.6343113356614026 
  	4.706254469999999  0.6103829559284821 
  	4.731223249999999  0.5859285792664803 
  	4.756192029999999  0.5610969090801613 
  	4.7811608099999985  0.5360317671062783 
  	4.806129589999999  0.510870913125332 
  	4.831098369999999  0.48574501414500676 
  	4.8560671499999986  0.46077676884545243 
  	4.881035929999999  0.43608018990974223 
  	4.906004709999999  0.41176004387986215 
  	4.930973489999999  0.3879114454359629 
  	4.955942269999999  0.36461960053522274 
  	4.980911049999999  0.34195969069869997 
  	5.005879829999999  0.319996888922119 
  	5.030848609999999  0.2987864962187401 
  	5.055817389999999  0.27837418668244285 
  	5.080786169999999  0.2587963481776362 
  	5.105754949999999  0.24008050530601197 
  	5.130723729999999  0.22224581114573966 
  	5.155692509999999  0.20530359438043408 
  	5.180661289999998  0.18925794880188443 
  	5.205630069999999  0.17410635274951383 
  	5.230598849999999  0.1598403068050655 
  	5.255567629999998  0.14644597895853026 
  	5.280536409999998  0.13390484746611075 
  	5.305505189999998  0.12219433269931725 
  	5.330473969999998  0.11128841040567204 
  	5.355442749999998  0.10115819993675756 
  	5.380411529999998  0.09177252212392818 
  	5.405380309999998  0.0830984225728834 
  	5.430349089999998  0.07510165718780594 
  	5.455317869999998  0.06774713770837665 
  	5.480286649999998  0.06099933593705869 
  	5.505255429999998  0.05482264614161529 
  	5.530224209999998  0.049181705833250086 
  	5.555192989999998  0.044041675741716725 
  	5.580161769999998  0.03936848033526607 
  	5.605130549999998  0.035129010667174894 
  	5.630099329999998  0.031291291676020475 
  	5.655068109999998  0.02782461632879747 
  }; 
  
  %

  \end{axis}
  \end{tikzpicture}
\end{center}
\caption{Five realizations of the density conditional on $\cG_{-k}$ (blue), 
 their average (orange, dotted in the b/w version),
 and  the true density  (thick black)  for $k=1$ (left), $k=2$ (middle), and $k=3$ (right),
 for the cantilever example.}
\label{fig:canti-cde}
\end{figure}

\begin{figure}[hbt]

\begin{center}
\begin{tikzpicture} 
    \begin{axis}[ 
      xlabel=$x$,
       grid,
       xmin=4,
       xmax=4.9,
       line width = 0.8pt,
       width = 0.45\textwidth,
       height = 4cm,
      ] 
      \addplot[red] table[x=var1,y=var2] { 
      var1  var2 
      4.001143100098419  0.6703636500819347 
      4.011866748088571  0.6804949609632656 
      4.020782674140249  0.688717273667733 
      4.034432621766344  0.7009290332971149 
      4.037994669242038  0.7040374761749869 
      4.049800558491264  0.7140975070829467 
      4.05832696782984  0.7211232861175113 
      4.0662000389149435  0.7274253666983769 
      4.073223895693563  0.7328931519626151 
      4.087802670134454  0.7437616686955486 
      4.095179997870103  0.7490071771081409 
      4.102690576842437  0.7541672185745567 
      4.110936967114892  0.7596195111454765 
      4.119163402491013  0.7648323681401168 
      4.131490566841569  0.7722131230985441 
      4.143137628265337  0.7787040427394998 
      4.146690774490517  0.7805897254328148 
      4.153307390474001  0.7839825491191135 
      4.170697825647459  0.7921578872879008 
      4.17229145866647  0.7928529862938432 
      4.186781874047041  0.7987547247321339 
      4.190633786831943  0.8001964802736121 
      4.198724430173328  0.8030509032304257 
      4.2089595097947345  0.8063247127249772 
      4.222109519624333  0.8099799189975636 
      4.226731171819401  0.8111180802208744 
      4.238721261337405  0.8137180185913622 
      4.244855630642495  0.8148525020530913 
      4.257589149531802  0.8167888398948367 
      4.267923660190571  0.8179496786238497 
      4.277997695335541  0.8187321514628476 
      4.28691138626462  0.8191411643476623 
      4.295591570391796  0.8192875014675411 
      4.303537029681192  0.8192066166109325 
      4.31053489964182  0.8189676740198306 
      4.315714750625348  0.8186908954038521 
      4.3243750394389515  0.8180408084647399 
      4.341401122768977  0.8160930886284775 
      4.350799818992932  0.8146467367348893 
      4.355309067943666  0.8138611340750113 
      4.365284716478368  0.8119155894193515 
      4.375060931854107  0.8097364762890763 
      4.378018128341243  0.8090250845275914 
      4.395474556768711  0.8043425535696889 
      4.399522006912758  0.8031410290119929 
      4.409959686634041  0.7998461407420006 
      4.420503400019302  0.7962358864759308 
      4.42899217196842  0.7931279445464079 
      4.433223634014725  0.791512804170599 
      4.445840342199149  0.7864416536208141 
      4.4506922119925685  0.7843912958581724 
      4.466121426341882  0.777509890146276 
      4.4703135052305205  0.7755469582270674 
      4.478433797175351  0.7716334977400797 
      4.490182036473384  0.7657164305631949 
      4.495797542268152  0.7627834238606974 
      4.505527181291639  0.7575441716574347 
      4.519052569603998  0.7499357402083127 
      4.528936769489862  0.7441415720653936 
      4.536691075170353  0.739460467257373 
      4.543849972219158  0.7350347624635564 
      4.5569775892275715  0.7266644157010277 
      4.558865653078085  0.7254338447870251 
      4.575616285207027  0.7142284724358772 
      4.5792911456952705  0.7117020061583043 
      4.591999464171447  0.7027803762455147 
      4.600793225198051  0.6964422577672191 
      4.6052793785783255  0.6931579992818224 
      4.619623644104162  0.6824311705767188 
      4.627332494289885  0.6765275274618358 
      4.635531969600043  0.6701443964792839 
      4.645219386676474  0.6624683572216803 
      4.656297263722828  0.6535169658488034 
      4.664174935653663  0.6470421659707162 
      4.672569040596336  0.6400459822012731 
      4.680698401467471  0.6331779343771873 
      4.689593464171152  0.6255620880750093 
      4.6971081978288725  0.6190486708462936 
      4.706649654971931  0.6106778230401422 
      4.718990720124932  0.5996914885226033 
      4.727471479079004  0.5920430104047382 
      4.73356145891449  0.5865036386040917 
      4.742651413153739  0.5781659638789974 
      4.749804741083501  0.5715490032348989 
      4.761884782342947  0.5602711599697956 
      4.76665536814656  0.555783845279905 
      4.779359155975447  0.5437498701442552 
      4.789228308112707  0.5343239461555755 
      4.799794192485789  0.5241670759545769 
      4.802172699950617  0.5218722274846544 
      4.811208659878094  0.5131289324112438 
      4.827171747945884  0.49760039513287424 
      4.829378818295849  0.49544661711189125 
      4.839161759933375  0.48588459407455664 
      4.853606007338513  0.47173384662348966 
      4.862077439924494  0.4634253653423199 
      4.872463245867602  0.45323928783641176 
      4.873371183995904  0.4523490942293568 
      4.884743987491529  0.44120689212266423 
      4.897833147414796  0.42841364076336114 
 }; 
%
      \addplot[green!80!black,very thick] table[x=var1,y=var2] { 
      var1  var2 
      4.001143100098419  0.658290125452424 
      4.011866748088571  0.6686252667519186 
      4.020782674140249  0.6770583486774475 
      4.034432621766344  0.689669490049286 
      4.037994669242038  0.6928979806175524 
      4.049800558491264  0.7034042768626135 
      4.05832696782984  0.7107995594158015 
      4.0662000389149435  0.7174786374675516 
      4.073223895693563  0.7233120528649404 
      4.087802670134454  0.735028016382853 
      4.095179997870103  0.7407477251142482 
      4.102690576842437  0.7464214722794937 
      4.110936967114892  0.7524731642261916 
      4.119163402491013  0.7583199940945222 
      4.131490566841569  0.7667155912737916 
      4.143137628265337  0.774232949821655 
      4.146690774490517  0.7764440849001764 
      4.153307390474001  0.7804574611377207 
      4.170697825647459  0.7903462265460105 
      4.17229145866647  0.7912037652772944 
      4.186781874047041  0.7986182167070351 
      4.190633786831943  0.8004717266345814 
      4.198724430173328  0.804201934611102 
      4.2089595097947345  0.8086009944228536 
      4.222109519624333  0.8137209375485949 
      4.226731171819401  0.8153766517554896 
      4.238721261337405  0.8193198359047156 
      4.244855630642495  0.8211391757128095 
      4.257589149531802  0.8244840409041102 
      4.267923660190571  0.8267675683933451 
      4.277997695335541  0.8286194690666129 
      4.28691138626462  0.8299489245409357 
      4.295591570391796  0.8309639212613004 
      4.303537029681192  0.8316508257820315 
      4.31053489964182  0.8320640395525734 
      4.315714750625348  0.8322542706467031 
      4.3243750394389515  0.8323528416126509 
      4.341401122768977  0.8317481267860413 
      4.350799818992932  0.8309630596114187 
      4.355309067943666  0.8304731003881529 
      4.365284716478368  0.8291295469734498 
      4.375060931854107  0.8274681270314449 
      4.378018128341243  0.8268987951998646 
      4.395474556768711  0.822912632724674 
      4.399522006912758  0.8218370773249357 
      4.409959686634041  0.8188040654452102 
      4.420503400019302  0.8153652341218911 
      4.42899217196842  0.8123268484008729 
      4.433223634014725  0.8107235539426529 
      4.445840342199149  0.8055979469969523 
      4.4506922119925685  0.8034911069465954 
      4.466121426341882  0.796301163774668 
      4.4703135052305205  0.7942211298166313 
      4.478433797175351  0.7900414330030274 
      4.490182036473384  0.7836497508497718 
      4.495797542268152  0.7804538042250956 
      4.505527181291639  0.7747060696696569 
      4.519052569603998  0.7662852315018794 
      4.528936769489862  0.759825158872799 
      4.536691075170353  0.7545822327445866 
      4.543849972219158  0.7496092447946586 
      4.5569775892275715  0.7401703750870519 
      4.558865653078085  0.7387797680062764 
      4.575616285207027  0.726093644220159 
      4.5792911456952705  0.7232292403506786 
      4.591999464171447  0.7131092437593527 
      4.600793225198051  0.705919487394799 
      4.6052793785783255  0.7021952004322458 
      4.619623644104162  0.690043366069226 
      4.627332494289885  0.6833668619651477 
      4.635531969600043  0.6761600018886019 
      4.645219386676474  0.667512612769769 
      4.656297263722828  0.6574589023935627 
      4.664174935653663  0.6502094248062427 
      4.672569040596336  0.6423993108634871 
      4.680698401467471  0.6347569603097475 
      4.689593464171152  0.6263125398498575 
      4.6971081978288725  0.6191164998340679 
      4.706649654971931  0.6099047088151193 
      4.718990720124932  0.5978784194150433 
      4.727471479079004  0.5895495505487556 
      4.73356145891449  0.5835400430322428 
      4.742651413153739  0.5745308312186233 
      4.749804741083501  0.5674119000540818 
      4.761884782342947  0.5553417056626284 
      4.76665536814656  0.5505612202094837 
      4.779359155975447  0.5378026995327735 
      4.789228308112707  0.5278713968780323 
      4.799794192485789  0.5172301287447093 
      4.802172699950617  0.5148343699456245 
      4.811208659878094  0.5057350324854643 
      4.827171747945884  0.4896829441149675 
      4.829378818295849  0.4874673064969752 
      4.839161759933375  0.4776615131445843 
      4.853606007338513  0.46324010202681515 
      4.862077439924494  0.4548211005388801 
      4.872463245867602  0.44454654812663574 
      4.873371183995904  0.4436510243707192 
      4.884743987491529  0.4324738205564437 
      4.897833147414796  0.41971033572920174 
 }; 
%
      \addplot[black,dashed] table[x=var1,y=var2] { 
      var1  var2 
      4.001143100098419  0.6581729858698636 
      4.011866748088571  0.6685145692328184 
      4.020782674140249  0.6769530607090263 
      4.034432621766344  0.6895724857667656 
      4.037994669242038  0.6928031255105671 
      4.049800558491264  0.7033164707108364 
      4.05832696782984  0.7107167480486927 
      4.0662000389149435  0.7174003473903449 
      4.073223895693563  0.723237712087848 
      4.087802670134454  0.7349615869487907 
      4.095179997870103  0.7406851399152005 
      4.102690576842437  0.7463626850891512 
      4.110936967114892  0.7524184098993871 
      4.119163402491013  0.758269118774046 
      4.131490566841569  0.7666702620583977 
      4.143137628265337  0.7741925770536594 
      4.146690774490517  0.7764051718503233 
      4.153307390474001  0.7804212044006834 
      4.170697825647459  0.7903165984657733 
      4.17229145866647  0.7911747214457963 
      4.186781874047041  0.7985943338017075 
      4.190633786831943  0.8004491750721777 
      4.198724430173328  0.8041821341647922 
      4.2089595097947345  0.8085846038368382 
      4.222109519624333  0.8137088541627071 
      4.226731171819401  0.8153660725188198 
      4.238721261337405  0.8193131614616393 
      4.244855630642495  0.8211345113443774 
      4.257589149531802  0.82448360651737 
      4.267923660190571  0.8267706516555233 
      4.277997695335541  0.8286260792625284 
      4.28691138626462  0.8299587524040153 
      4.295591570391796  0.8309769822232356 
      4.303537029681192  0.8316669407715893 
      4.31053489964182  0.8320829242744167 
      4.315714750625348  0.8322752558382573 
      4.3243750394389515  0.8323774372444793 
      4.341401122768977  0.8317801888481874 
      4.350799818992932  0.8309994521799545 
      4.355309067943666  0.8305116220249159 
      4.365284716478368  0.8291728922935269 
      4.375060931854107  0.8275163421874916 
      4.378018128341243  0.8269485091446513 
      4.395474556768711  0.8229714039007876 
      4.399522006912758  0.8218979919306247 
      4.409959686634041  0.8188705626017857 
      4.420503400019302  0.815437424154107 
      4.42899217196842  0.812403635650527 
      4.433223634014725  0.810802630242128 
      4.445840342199149  0.8056838052875651 
      4.4506922119925685  0.8035795442959225 
      4.466121426341882  0.7963976263377579 
      4.4703135052305205  0.7943197124541765 
      4.478433797175351  0.7901440292846268 
      4.490182036473384  0.7837578962822281 
      4.495797542268152  0.7805644765169657 
      4.505527181291639  0.7748208988641474 
      4.519052569603998  0.7664053095724133 
      4.528936769489862  0.7599486311051695 
      4.536691075170353  0.7547080792162657 
      4.543849972219158  0.7497370419673519 
      4.5569775892275715  0.7403011059635862 
      4.558865653078085  0.7389108490851862 
      4.575616285207027  0.7262269972182641 
      4.5792911456952705  0.7233628844327824 
      4.591999464171447  0.7132432966910055 
      4.600793225198051  0.706053268669455 
      4.6052793785783255  0.7023286655395696 
      4.619623644104162  0.690175007826889 
      4.627332494289885  0.683497011587055 
      4.635531969600043  0.6762881736372458 
      4.645219386676474  0.6676379351018183 
      4.656297263722828  0.6575803009237564 
      4.664174935653663  0.6503276142382082 
      4.672569040596336  0.6425137121028062 
      4.680698401467471  0.6348673451717131 
      4.689593464171152  0.6264181580961288 
      4.6971081978288725  0.6192178057853742 
      4.706649654971931  0.6100001896969014 
      4.718990720124932  0.5979658407187227 
      4.727471479079004  0.5896311296686123 
      4.73356145891449  0.5836172930811595 
      4.742651413153739  0.5746014383082061 
      4.749804741083501  0.5674771486475602 
      4.761884782342947  0.5553977003882818 
      4.76665536814656  0.5506135064846212 
      4.779359155975447  0.5378450193024616 
      4.789228308112707  0.5279059362263515 
      4.799794192485789  0.5172563641220577 
      4.802172699950617  0.5148587456774595 
      4.811208659878094  0.5057523947825256 
      4.827171747945884  0.4896882029966557 
      4.829378818295849  0.4874709292200483 
      4.839161759933375  0.47765801662037827 
      4.853606007338513  0.463226554170017 
      4.862077439924494  0.45480195290167275 
      4.872463245867602  0.4445208699452721 
      4.873371183995904  0.4436247938356383 
      4.884743987491529  0.4324409383749107 
      4.897833147414796  0.41967044223427175 
 }; 
	\end{axis}
	\end{tikzpicture}
\begin{tikzpicture} 
    \begin{axis}[ 
      xlabel=$x$,
       grid,
       xmin=4.5,
       xmax=4.7,
       line width = 0.8pt,
       width = 0.45\textwidth,
       height = 4cm,
      ] 
      \addplot[red] table[x=var1,y=var2] { 
      var1  var2 
      4.500254022244093  0.762644302395525 
      4.502637055130793  0.7602614932766419 
      4.5046183720311666  0.7588769127043203 
      4.507651693725855  0.7578804508926208 
      4.508443259831564  0.7578391828310377 
      4.511066790775836  0.7582809893965438 
      4.512961548406631  0.7590806216575867 
      4.514711119758876  0.7600910452247484 
      4.516271976820792  0.7611507764777955 
      4.519511704474323  0.7635775455437936 
      4.521151110637801  0.764804219552759 
      4.522820128187208  0.7659666690035182 
      4.524652659358865  0.7670707144810064 
      4.526480756109114  0.7679132365330484 
      4.529220125964793  0.7685116322239226 
      4.531808361836742  0.7681546958723668 
      4.532597949886782  0.7678420171991055 
      4.534068308994223  0.7669794476099949 
      4.53793285014388  0.7628842431851534 
      4.538286990814771  0.7623733244558089 
      4.541507083121565  0.7567872649898958 
      4.5423630637404315  0.7550544761107283 
      4.544160984482962  0.7511744397915386 
      4.546435446621052  0.7460111718754827 
      4.54935767102763  0.7394619660695736 
      4.550384704848756  0.7372891328885548 
      4.55304916918609  0.7321638227384252 
      4.554412362364999  0.7298777670254831 
      4.557242033229289  0.7259243899167166 
      4.55953859115346  0.7235078006881998 
      4.56177726563012  0.7218033067036332 
      4.563758085836582  0.7208036177611067 
      4.565687015642621  0.720257295189709 
      4.567452673262487  0.7200866562036695 
      4.56900775547596  0.7201563702474647 
      4.570158833472299  0.7203245595044829 
      4.572083342097545  0.7207879665317073 
      4.575866916170884  0.722196574613233 
      4.577955515331762  0.7230886734174539 
      4.578957570654148  0.7234896190681973 
      4.581174381439637  0.7241739164219684 
      4.5833468737453575  0.7243724137348936 
      4.584004028520276  0.7243047398009023 
      4.587883234837491  0.7223408108609038 
      4.588782668202835  0.7214683031225145 
      4.591102152585343  0.7185303202035865 
      4.59344520000429  0.7147391506051362 
      4.59533159377076  0.7112934403032882 
      4.596271918669939  0.709501472802819 
      4.599075631599811  0.7041149553474254 
      4.600153824887237  0.7020964487284035 
      4.603582539187085  0.6961780284190382 
      4.604514112273449  0.6947359801195165 
      4.606318621594522  0.6921641766983846 
      4.60892934143853  0.6889217164537035 
      4.610177231615145  0.6875369602751935 
      4.612339373620364  0.6853256412253516 
      4.615345015467555  0.6824731951044373 
      4.61754150433108  0.6803757210730537 
      4.619264683371189  0.6786274237677177 
      4.620855549382036  0.6768993344044002 
      4.623772797606127  0.6734539107301686 
      4.624192367350686  0.6729323249945592 
      4.627914730046006  0.6681908350141088 
      4.628731365710061  0.6671672323115917 
      4.631555436482544  0.6638967105483142 
      4.633509605599567  0.6620564418578316 
      4.634506528572961  0.6613062472808782 
      4.637694143134258  0.6599952601813432 
      4.6394072209533075  0.6600495444313357 
      4.641229326577787  0.6607219692678051 
      4.643382085928105  0.662337757028448 
      4.64584383638285  0.6652190010636434 
      4.647594430145259  0.6678616157797868 
      4.649459786799186  0.6711102980625472 
      4.651266311437216  0.6745662988802791 
      4.653242992038034  0.6785276576398321 
      4.654912932850861  0.6818789805705208 
      4.657033256660429  0.685921544139756 
      4.659775715583319  0.6903297868411311 
      4.661660328684223  0.692479995358749 
      4.6630136575365535  0.6934288058948883 
      4.665033647367498  0.6937485081678137 
      4.666623275796334  0.6929788608968572 
      4.669307729409544  0.6895170748122504 
      4.670367859588125  0.6874083788508429 
      4.6731909235500995  0.6799137583137507 
      4.675384068469491  0.6724920501913231 
      4.67773204277462  0.6634895913246412 
      4.678260599989026  0.6613721090845364 
      4.680268591084021  0.6532240633636179 
      4.683815943987974  0.6392727063734241 
      4.684306404065744  0.6374611471760663 
      4.686480391096306  0.6299555954683351 
      4.689690223853003  0.6208754529814194 
      4.691572764427666  0.6168923721059394 
      4.693880721303912  0.6134714633903406 
      4.694082485332423  0.6132475013952636 
      4.696609774998118  0.6113550539626682 
      4.699518477203288  0.610847756868902 
 }; 
%
      \addplot[green!80!black,very thick] table[x=var1,y=var2] { 
      var1  var2 
      4.500254022244093  0.7752545138786381 
      4.502637055130793  0.7741129403976861 
      4.5046183720311666  0.773145514473238 
      4.507651693725855  0.7715692812911283 
      4.508443259831564  0.7711360031673657 
      4.511066790775836  0.7696436787792355 
      4.512961548406631  0.7685336823306672 
      4.514711119758876  0.7674977956219837 
      4.516271976820792  0.7665804356679392 
      4.519511704474323  0.7647433083997253 
      4.521151110637801  0.7638682290993394 
      4.522820128187208  0.7630210797991562 
      4.524652659358865  0.7621371083626836 
      4.526480756109114  0.7613070837005078 
      4.529220125964793  0.7601352956766785 
      4.531808361836742  0.7590674327657113 
      4.532597949886782  0.7587416844198961 
      4.534068308994223  0.7581172193956768 
      4.53793285014388  0.7562794183021322 
      4.538286990814771  0.7560907341250098 
      4.541507083121565  0.754187190666676 
      4.5423630637404315  0.753621399838045 
      4.544160984482962  0.7523589238330247 
      4.546435446621052  0.7506389151226632 
      4.54935767102763  0.7482822635298715 
      4.550384704848756  0.7474323969862285 
      4.55304916918609  0.7452144325696738 
      4.554412362364999  0.7440759956590481 
      4.557242033229289  0.7417265856646469 
      4.55953859115346  0.739826765405591 
      4.56177726563012  0.7379859146350941 
      4.563758085836582  0.7363656669259189 
      4.565687015642621  0.7347930390893392 
      4.567452673262487  0.7333588897396222 
      4.56900775547596  0.7321002691605032 
      4.570158833472299  0.7311675711195638 
      4.572083342097545  0.7295983078261035 
      4.575866916170884  0.7264833243392713 
      4.577955515331762  0.7247743138416007 
      4.578957570654148  0.7239622515942267 
      4.581174381439637  0.7221781857512601 
      4.5833468737453575  0.7204472079913957 
      4.584004028520276  0.7199272722561113 
      4.587883234837491  0.7168504298232811 
      4.588782668202835  0.7161249970727941 
      4.591102152585343  0.7142236170740232 
      4.59344520000429  0.7122414962832923 
      4.59533159377076  0.7105838144633393 
      4.596271918669939  0.7097328857240425 
      4.599075631599811  0.7071055511215953 
      4.600153824887237  0.7060645292125263 
      4.603582539187085  0.7026987528799837 
      4.604514112273449  0.701791983247406 
      4.606318621594522  0.7000688554069932 
      4.60892934143853  0.6976802011824595 
      4.610177231615145  0.696587558775449 
      4.612339373620364  0.6947754085437127 
      4.615345015467555  0.6923964324770701 
      4.61754150433108  0.6907155064886606 
      4.619264683371189  0.6893914497142803 
      4.620855549382036  0.6881301985664504 
      4.623772797606127  0.685637445872628 
      4.624192367350686  0.6852537322390027 
      4.627914730046006  0.6815448627652777 
      4.628731365710061  0.6806694507706288 
      4.631555436482544  0.6775492781709521 
      4.633509605599567  0.6753637276469728 
      4.634506528572961  0.6742620192504891 
      4.637694143134258  0.6708997280080368 
      4.6394072209533075  0.6692273349488196 
      4.641229326577787  0.6675654451600731 
      4.643382085928105  0.6657664912715292 
      4.64584383638285  0.6638894690765955 
      4.647594430145259  0.6626394883477021 
      4.649459786799186  0.6613580972995335 
      4.651266311437216  0.6601309012246466 
      4.653242992038034  0.6587495168384736 
      4.654912932850861  0.6575169136556303 
      4.657033256660429  0.6558407699429305 
      4.659775715583319  0.6534951863350954 
      4.661660328684223  0.6518024070013541 
      4.6630136575365535  0.6505851061719754 
      4.665033647367498  0.6488220348768116 
      4.666623275796334  0.6475172166594275 
      4.669307729409544  0.6455077011738884 
      4.670367859588125  0.6447723364388559 
      4.6731909235500995  0.642920851117479 
      4.675384068469491  0.641533962187193 
      4.67773204277462  0.6400368290775353 
      4.678260599989026  0.6396929410709903 
      4.680268591084021  0.6383467590676443 
      4.683815943987974  0.6357195097080941 
      4.684306404065744  0.6353225453000421 
      4.686480391096306  0.6334414352134439 
      4.689690223853003  0.6302533604113894 
      4.691572764427666  0.6281473833275066 
      4.693880721303912  0.6253754351864877 
      4.694082485332423  0.6251262681880957 
      4.696609774998118  0.6219743477658127 
      4.699518477203288  0.6184412314751052 
 }; 
%
      \addplot[dashed,black] table[x=var1,y=var2] { 
      var1  var2 
      4.500254022244093  0.7780465040384773 
      4.502637055130793  0.7765937726933926 
      4.5046183720311666  0.7753789073871894 
      4.507651693725855  0.773510304101935 
      4.508443259831564  0.7730211013456731 
      4.511066790775836  0.7713959789021372 
      4.512961548406631  0.7702181966449377 
      4.514711119758876  0.7691265937253691 
      4.516271976820792  0.7681464480844841 
      4.519511704474323  0.7660873629481774 
      4.521151110637801  0.7650287099451398 
      4.522820128187208  0.7639411886438447 
      4.524652659358865  0.7627359975975573 
      4.526480756109114  0.7615240982436375 
      4.529220125964793  0.7596965348273748 
      4.531808361836742  0.7579601836585597 
      4.532597949886782  0.7574293475680396 
      4.534068308994223  0.7564396272337076 
      4.53793285014388  0.7538278392675561 
      4.538286990814771  0.7535877617187521 
      4.541507083121565  0.751393205329637 
      4.5423630637404315  0.7508057321038172 
      4.544160984482962  0.7495652960508655 
      4.546435446621052  0.7479794806034359 
      4.54935767102763  0.7459098456016321 
      4.550384704848756  0.7451731069130516 
      4.55304916918609  0.743239202558005 
      4.554412362364999  0.7422387290658721 
      4.557242033229289  0.7401416856065363 
      4.55953859115346  0.738422791839624 
      4.56177726563012  0.7367335523847505 
      4.563758085836582  0.7352316245756558 
      4.565687015642621  0.7337644964811789 
      4.567452673262487  0.7324220049646961 
      4.56900775547596  0.7312397658350431 
      4.570158833472299  0.7303655102040739 
      4.572083342097545  0.7289047910648283 
      4.575866916170884  0.7260258054815253 
      4.577955515331762  0.7244220259120263 
      4.578957570654148  0.7236472135863178 
      4.581174381439637  0.7219210031385751 
      4.5833468737453575  0.720215417544533 
      4.584004028520276  0.7196975176708614 
      4.587883234837491  0.7166204397902597 
      4.588782668202835  0.7159020621169611 
      4.591102152585343  0.7140377297123202 
      4.59344520000429  0.7121303713360317 
      4.59533159377076  0.710576277085466 
      4.596271918669939  0.7097959635928448 
      4.599075631599811  0.7074539009291433 
      4.600153824887237  0.7065493863361928 
      4.603582539187085  0.7036693673223541 
      4.604514112273449  0.702887940061097 
      4.606318621594522  0.7013755344365451 
      4.60892934143853  0.6991912168755212 
      4.610177231615145  0.6981464629026853 
      4.612339373620364  0.6963310388706967 
      4.615345015467555  0.6937915203313091 
      4.61754150433108  0.6919234114262178 
      4.619264683371189  0.6904506439929261 
      4.620855549382036  0.689086278744159 
      4.623772797606127  0.6865783822109964 
      4.624192367350686  0.6862170703727686 
      4.627914730046006  0.6830076786281489 
      4.628731365710061  0.6823017948690623 
      4.631555436482544  0.6798570325085739 
      4.633509605599567  0.6781609619329783 
      4.634506528572961  0.6772935035072906 
      4.637694143134258  0.6745064114587842 
      4.6394072209533075  0.6729969977913435 
      4.641229326577787  0.6713790131278597 
      4.643382085928105  0.669447411002888 
      4.64584383638285  0.667210128159921 
      4.647594430145259  0.6656021157547982 
      4.649459786799186  0.6638758681367684 
      4.651266311437216  0.6621945290283252 
      4.653242992038034  0.6603497047467571 
      4.654912932850861  0.6587897946254004 
      4.657033256660429  0.6568137507857088 
      4.659775715583319  0.6542668603461974 
      4.661660328684223  0.6525219464636577 
      4.6630136575365535  0.651271609370125 
      4.665033647367498  0.6494112351579387 
      4.666623275796334  0.6479515454895631 
      4.669307729409544  0.6454952933985904 
      4.670367859588125  0.6445254115700122 
      4.6731909235500995  0.6419327473422648 
      4.675384068469491  0.6399014652251687 
      4.67773204277462  0.6377061909877298 
      4.678260599989026  0.6372089266579156 
      4.680268591084021  0.6353093850259015 
      4.683815943987974  0.6319212596446576 
      4.684306404065744  0.6314502523609202 
      4.686480391096306  0.6293575288671432 
      4.689690223853003  0.6262632070854004 
      4.691572764427666  0.6244534434893595 
      4.693880721303912  0.622244630196098 
      4.694082485332423  0.6220521489218517 
      4.696609774998118  0.6196427800558999 
      4.699518477203288  0.6168759716955438 
 }; 
	\end{axis}
	\end{tikzpicture}
	\end{center}
	\caption{The CDE under MC (red), under RQMC  (green, very close to the dashed line) 
	 and the true density (black, dashed) 
	 for $\cG_{-1}$ with $n=2^{10}$ (left) and for $\cG_{-2}$ with $n=2^{16}$ (right),
	 for the cantilever example.}
	\label{fig:canti-MCvsRQMC}
\end{figure}
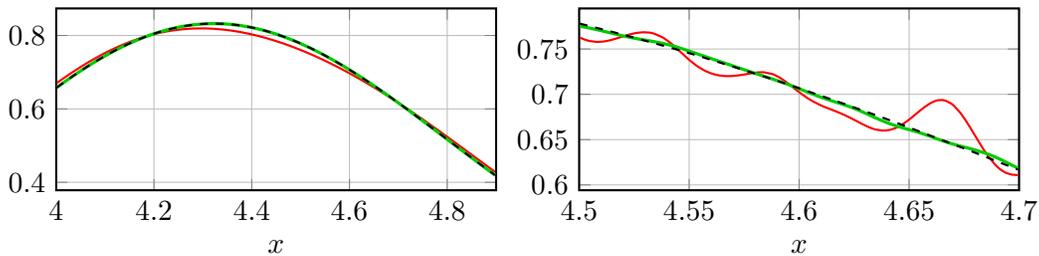

\subsection{A stochastic activity network}
\label{sec:san}

In this example, the conditioning for the CDE must hide more than one random variable.
We consider an acyclic directed graph $G = (\cN, \cA)$ where $\cN$ is a finite set of nodes
and $\cA = \{a_j = (\alpha_j, \beta_j),\, j=1,\dots,d\}$ a finite set of arcs (directed links) 
where $a_j$ goes from $\alpha_j$ to $\beta_j$.  There is a source node having only outcoming arcs,
a sink node having only incoming arcs, and each arc belongs to at least one path going
from the source to the sink.  There can be at most one arc for each pair $(\alpha_j, \beta_j)$
(no parallel arcs).  Each arc $j$ has random length $Y_j$.
These $Y_j$ are assumed independent with continuous cdf's $F_j$, density $f_j$, 
and can be generated by inversion: $Y_j = F_j^{-1}(U_j)$ where $U_j\sim U(0,1)$.
The length of the longest path from the source to the sink is a random variable $X$
and the goal is to estimate the density of $X$.

This general model has several applications.  
The arcs $a_j$ may represent activities having random durations and the graph represents
precedence relationships between all activities of a project.  
Activity $a_j$ cannot start before all activities $j'$ with $\beta_{j'} = \alpha_j$ are completed.  
Then $X$ represents the duration of the project if all activities are started as soon as allowed.
This type of \emph{stochastic activity network} (SAN) is widely used in project management
for all types of projects (e.g., construction, software, etc.), communication, transportation, etc.
For example, the graph may represent a large railway network in which each activity corresponds 
to a train stopping at a station, or a train covering a given segment of its route, 
or a minimal spacing between trains, etc.
Precedence relationships are needed because railways are shared, there are ordering and 
distancing rules between trains, passengers have connections between trains, trains are merged
or split at certain points, etc.  The travel time of one passenger in this network turns out to 
be the length $X$ of the longest path in a subnetwork whose source and sink are the origin and 
destination of this passenger. 

For our numerical experiments, we use a small example from \cite{vAVR96a,vAVR98a}, 
who showed how to use CMC to estimate $\EE[X]$ and some quantiles of the distribution of $X$.
\cite{vLEC00b} and \cite{vLEC12a} used this same example to test the combination
of CMC with RQMC to estimate $\EE[X]$.
The network is depicted in Fig.~\ref{fig:san13} and the cdf's $F_j$ are given in \cite{vAVR96a}.  
Much larger networks can be handled in the same way.
We will estimate the density of $X$ over $[a,b] = [22,\, 106.24]$, which covers 
about 95\% of the density.

\begin{figure}[hbt]
\begin{center}
\begin{tikzpicture}[node distance=1.8cm,>=stealth',auto,
    cuta/.style={color=blue!60,dashed},
    cutb/.style={color=red!50!yellow,dashed}
]
\tikzstyle{node}=[circle,thick,draw=black,fill=yellow!20,minimum size=6mm]
\tikzstyle{nodeb}=[circle,thick,draw=black,fill=blue!10,minimum size=6mm]
    \node [nodeb] (n0) [label=left:{\small\color{blue} source}]  {0};
    \node [node] (n1) [right of=n0] {1}
          edge [pre,thick] node[below] {$Y_1$} (n0);
    \node [node] (n2) [above of=n1] {2}
          edge [pre,thick] node[above left] {$Y_2$} (n0)
          edge [pre,thick] node[left] {$Y_3$} (n1);
    \node [node] (n3) [right of=n1] {3}
          edge [pre,thick] node[below] {$Y_4$} (n1);
    \node [node] (n4) [above of=n3] {4}
          edge [pre,thick] node[right] {$Y_8$} (n3);
    \node [node] (n5) [above of=n4] {5}
          edge [pre,cuta] node[right] {$Y_{10}$} (n4)
          edge [pre,cuta] node[below right] {$Y_5$} (n1)
          edge [pre,cuta] node[above left] {$Y_6$} (n2);
    \node [node] (n6) [right of=n3] {6}
          edge [pre,cuta] node[below] {$Y_7$} (n3);
    \node [node] (n7) [above of=n6] {7}
          edge [pre,thick] node[right] {$Y_{12}$} (n6)
          edge [pre,cuta] node[below] {$Y_9$} (n4);
    \node [nodeb] (n8) [above of=n7, label=right:{\small\color{blue} sink}] {8}
          edge [pre,thick] node[right] {$Y_{13}$} (n7)
          edge [pre,thick] node[below] {$Y_{11}$} (n5);
\end{tikzpicture}
\end{center}
\caption{A stochastic activity network, with the cut $\cL = \{ 5, 6, 7, 9, 10 \}$ shown in dashed light blue}
\label{fig:san13}
\end{figure}
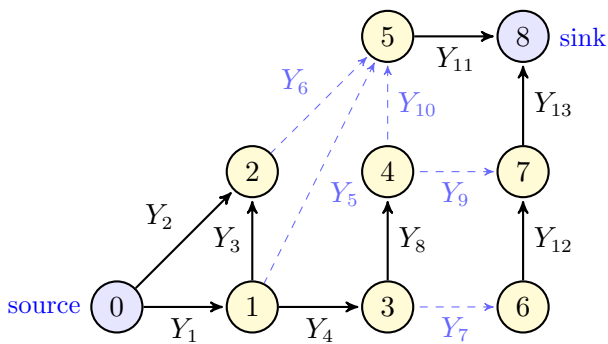

Here, $X$ is defined as the maximum length over several paths, and if we hide only a single 
random variable $Y_j$ to implement the CDE, we run into the same problem as in 
Example~\ref{ex:minmax}: Assumption~\ref{ass:cmc-dct} does not hold, because $F(\cdot\mid \cG)$
has a jump.  This means that we must hide more information (condition on less).
Following \cite{vAVR96a,vAVR98a}, we select a 
\emph{uniformly directed cut} $\cL$, which is a set of activities such that each path from 
the source to the sink contains exactly one activity from $\cL$, and let $\cG$ represent 
$\{Y_j,\, j \not\in {\cL}\}$.   
In Figure~\ref{fig:san13}, $\{1, 2\}$, $\{11, 13\}$, $\{ 5, 6, 7, 9, 10 \}$, and
$\{ 2, 3, 5, 8, 9, 13 \}$, are all valid choices of $\cL$.
The corresponding conditional cdf is 
\begin{equation}
\label{eq:cmc-estimator}
  F(x\mid\cG) = \PP\left[X \le x \mid \{Y_j : j\not\in {\cL}\}\right] 
              = \prod_{j\in\cL} \PP[Y_j \le x - P_j]
							= \prod_{j\in\cL} F_j(x - P_j)
\end{equation}
where $P_j$ is the length of the longest path that goes through arc $j$ 
when we exclude $Y_j$ from that length.  The conditional density is 
\[
  f(x\mid \cG) = \frac{\d}{\d x} F(x\mid \cG)  
	 = \sum_{j\in\cL} f_j(x - P_j)  \prod_{l\in\cL,\, l\not=j} F_l(x - P_j).
\]
Under this conditioning, if the $Y_j$'s are continuous variables with bounded variance,
Assumption~\ref{ass:cmc-dct} holds, 
so $f(x\mid \cG)$ is an unbiased density estimator with uniformly bounded variance.

For our numerical experiments, we use the same cut $\cL = \{ 5, 6, 7, 9, 10 \}$
as \cite{vAVR96a}, indicated in light blue in Figure~\ref{fig:san13},
even though there are other cuts with six links, 
which could possibly perform better because they hide more links.
We could also compute the CDE with several choices of $\cL$ and then take a convex combination.
\new{This approach scales nicely and works in exactly the same way for very large networks,
with thousands of links. A simple adaptation also works for a stochastic max-flow problem, in which we
want the density of the capacity of the minimal cut having the smallest capacity \citep{vLEC20m}.}

The GLRDE method 
described in Section~\ref{sec:GLRDE-estimator} does not work for this example.
Indeed, with $X = h(\bY)$ defined as the length of the longest path, for any $j$, 
the derivative $h_j(\bY)$ is zero whenever arc $j$ is not on the longest path,
so we would need to select an arc $j$ that is guaranteed to be on the longest path.
But there is no such arc in general.
We could perhaps apply a modified GLRDE that selects a cut instead of a single coordinate $Y_j$,
but this is beyond the scope of this paper.

Table~\ref{tab:san-results} and Figure \ref{fig:san-plots} summarize our results.
\hflorian{Here, we used a multilevel rank-1-lattice obtained with
  the $\mathcal{P}_2$ criterion and $\gamma_\frakv = 0.3^{|\frakv|}$.}%
We see that for $n = 2^{19}$, the CDE outperforms the KDE by a factor of about 20 with MC, 
and by a factor of about $2^8 \approx 250$ with RQMC.
\hpierre{Again, the lattice with $\log_2 n = 17$ looks weaker than the other ones.}
\hflorian{Checked with larger and smaller lattice parameters. All results were even more erratic.}

\begin{table}[!htbp]	
	\caption{Values of $\hat\nu$ and e19 for the SAN example. }
	\label{tab:san-results}
	\centering
	\small
	\begin{tabular}{c|l| c c }
	  \hline
		&& $\hat\nu$	& e19 \\
		\hline
		\multirow{4}{*}{CDE}
		& MC 		    & 0.96		& 25.6\\
		& Lat+s  	  & 1.31		& 30.9\\
		& Lat+s+b  	& 1.17		& 29.6\\
		& Sob+LMS 	& 1.27		& 29.9\\
		\hline
	\multirow{3}{*}{KDE}
	  & MC 			  & 0.78		& 20.9\\
		& Lat+s   	& 0.95		& 22.7\\
		& Lat+s+b   & 0.93    & 22.0 \\
		& Sob+LMS 	& 0.74 		& 21.9\\
		\hline
	\end{tabular}
\end{table}

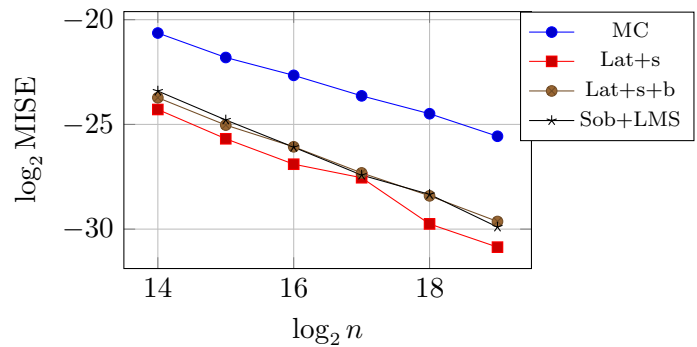
\begin{figure}[!htbp]
	\centering
	\begin{tikzpicture} 
	\begin{axis}[ 
	xlabel=$\log_2n$,
	ylabel=$\log_2\MISE$,
	grid,
	width = 7cm,
	height =5cm,
	cycle list name=defaultcolorlist,
	legend style={at={(1.4,1)}, anchor={north east}, font=\scriptsize},
	] 
	
	\addplot table[x=logN,y=logIV] { 	
		logN  logIV 
	14.0	 -20.635818
	15.0	 -21.804929
	16.0	 -22.657728
	17.0	 -23.636585
	18.0	 -24.489658
	19.0	 -25.565705
	}; 
	\addlegendentry{MC}

	\addplot table[x=logN,y=logIV] { 
		logN  logIV 
      14.0  -24.28725447828006 
      15.0  -25.686405090832118 
      16.0  -26.90008452939896 
      17.0  -27.54718029931705 
      18.0  -29.753949936558215 
      19.0  -30.859588503361 
		
	}; 
	\addlegendentry{Lat+s}
	\addplot table[x=logN,y=logIV] { 
		logN  logIV 
      14.0  -23.723448489062058 
      15.0  -25.032503023189808 
      16.0  -26.06927861033965 
      17.0  -27.315729287045095 
      18.0  -28.410876554460962 
      19.0  -29.63534519539921 
		
	}; 
	\addlegendentry{Lat+s+b}
		\addplot table[x=logN,y=logIV] { 
		logN  logIV 
		
	14.0	 -23.407210
	15.0	 -24.787950
	16.0	 -26.087181
	17.0	 -27.420209
	18.0	 -28.342173
	19.0	 -29.899237
		
	}; 
	\addlegendentry{Sob+LMS}
	%
	%
	%
	\end{axis}
	\end{tikzpicture}
\caption{MISE vs $n$ in log-log scale, for the SAN example.}
	\label{fig:san-plots}
\end{figure}

\subsection{Density of waiting times in a single queue}
\label{sec:single-queue}

We adapted this example from \cite{oPEN20a}, mainly to compare our approach with
the GLRDE proposed in their paper.   What we use here is not ordinary CMC,
but extended CMC, in which we condition on different information for each customer.
This type of strategy would work for much larger queueing systems and many other 
types of systems that involve random delays.
To estimate the density of the waiting times in a queueing system, the general idea is  
to hide sufficient information, for each customer in the system, so that its exact 
waiting time is unknown, \new{but it has a density conditional on the known information,
and we can compute this density easily. This is often easy to do even for large queueing systems.}
The hidden information can be the arrival time of the customer, the departure time of the 
previous customer, or something similar, selected so that we can compute the conditional density.

\subsubsection{Model with independent days.}

We consider a single-server FIFO queue in which customers arrive from an arbitrary 
arrival process (not necessarily stationary Poisson)
and the service times are independent, with continuous cdf $G$ and density $g$.
If $W$ denotes the waiting time of a ``random'' customer,
we want to estimate $p_0 = \PP[W=0]$ and the density $f$ of $W$ over $(0,\infty)$.  

We first consider a system that starts empty and evolves over a fixed time horizon $\tau$, 
which we call a \emph{day}.
Let $T_j$ be the arrival time of the $j$th customer, $T_0=0$, 
$A_{j} = T_{j} - T_{j-1}$ the $j$th interarrival time, 
$S_j$ the service time of customer $j$,
and $W_j$ the waiting time of customer $j$.
Since the system starts empty, we have $W_1=0$, and 
the Lindley recurrence gives us that $W_j = \max(0,\, W_{j-1} + S_{j-1} - A_j)$ for $j\ge 2$.
At time $\tau$, the arrival process stops, but service continues until all customers 
already arrived are served.   The number of customers handled in a day is 
the random variable $N = \max\{ j \ge 1 : T_j < \tau\}$.
The cdf of $W$ can be written  as $F(0) = p_0$ and for $x > 0$, 
$F(x) = \PP[W \le x]  = \EE[\II(W \le x)]$.
The sequence of waiting times of all customers over an infinite number
of independent successive days is a regenerative process that regenerates at the 
beginning of each day, so we can apply the renewal reward theorem, which gives
\begin{equation}                                  \label{eq:queue-renewal-reward}
 F(x) = \EE[\II(W \le x)] = \frac{\EE\left[\II[W_1\le x] + \cdots + \II[W_N \le x]\right]}{\EE[N]}.
\end{equation}
Since $\EE[N]$ does not depend on $x$, we see that for $x > 0$, 
the density $f(x)$ is the derivative of the numerator with respect to $x$, divided by $\EE[N]$.

To obtain a differentiable cdf estimator, we want to replace each indicator in the 
numerator by a conditional expectation.  One simple way of doing this is to hide the service
time $S_{j-1}$ of the previous customer; that is, replace $\II[W_j\le x]$ by 
\[
  P_j(x) = \PP[W_j\le x \mid W_{j-1} - A_j] = \PP[S_{j-1} \le x + A_j - W_{j-1}] 
	       = G(x + A_j - W_{j-1})   \quad \mbox{ for } x\ge 0.
\]
This gives $P_j(0) = G(A_j - W_{j-1})$ (there is a probability mass at 0), whereas 
for $x > 0$, we have $P'_j(x) = \d P_j(x)/\d x = g(x + A_j - W_{j-1})$ and then,
since $N$ does not change when we change $x$,
\begin{equation}           \label{eq:queue-cde0}
 f(x) = \frac{\EE[D(x)]}{\EE[N]} \qquad\mbox{where } 
        D(x) = \sum_{j=1}^N g(x + A_j - W_{j-1}).
\end{equation}
Note that we are not conditioning on the same information for all terms of the sum,
so what we do is not exactly CMC, but \emph{extended CMC} \citep{sBRA87a}.
It nevertheless provides the required smoothing and an unbiased density estimator
for the numerator of (\ref{eq:queue-renewal-reward}).
\new{In a multiserver queue, such as a call center with a large number of agents, one possibility
would be to hide the arrival time $A_j$ of the call, and compute the density of its waiting 
time conditional on the other information.}

Often, for example if the arrival process is Poisson, $\EE[N]$ can be computed exactly,
in which case we only need to estimate $\EE[D(x)]$ and we get an unbiased density estimator.
Otherwise, the denominator $\EE[N]$ can be estimated in the usual way,
and we are then in the standard setting of estimating a ratio of expectations \citep{sASM07a},
for which we have unbiased estimators for the numerator and the denominator. 
We simulate $n$ days, independently (with MC) or with $n$ RQMC points,
to obtain $n$ realizations of $(N, D(x))$, say $(N_1, D_1(x)),\dots, (N_n, D_n(x))$.
The ratio estimator (CDE) of $f(x)$ is 
\[
  \hat f(x) = \frac{\sum_{i=1}^n D_i(x)}{\sum_{i=1}^n N_i}.
\]
It can be computed at any $x\in [0,\infty)$.
For independent realizations (with MC), the variance of $\hat f(x)$ 
can be estimated using the delta method for ratio estimators \citep{sASM07a}:
\[
 n \Var[\hat f(x)] \to \frac{\Var[D_i(x)] + \Var[N_i] f^2(x) - 2\Cov[D_i(x), N_i] f(x)}{(\EE[N_i])^2}
\]
asymptotically, when $n\to\infty$.
This variance can be estimated by replacing the unknown quantities in this expression by their 
empirical values.   
This is consistent because the $n$ pairs $(D_i(x), N_i)$, $i=1,\dots,n$, are independent.
Alternatively, a confidence interval on $f(x)$ can also be computed with a 
bootstrap approach \citep{sCHO99a}.

In the RQMC case, the pairs $(D_i(x), N_i)$ are no longer independent.
Then, to obtain an estimator of $f(x)$ for which we can estimate the variance, we make $n_r$
independent replicates of the RQMC estimator of the pair $(\EE[D(x)], \EE[N])$, say
$(\bar D_1(x), \bar N_1), \dots, (\bar D_{n_r}(x), \bar N_{n_r})$, where each 
$(\bar D_j(x), \bar N_j)$ is the average of $n$ pairs $(D_i(x), N_i)$ sampled by RQMC.
We estimate the density $f(x)$ by the ratio of the two grand sums 
\[
   \hat f_{{\rm rqmc},n_r}(x) = \frac{\sum_{j=1}^{n_r} \bar D_j(x)}{\sum_{j=1}^{n_r} \bar N_j}.
\]
To estimate the variance, we use that
\[
  \Var[\hat f_{{\rm rqmc},n_r}(x)] \approx 
	\frac{\Var[\bar D_j(x)] + \Var[\bar N_j] f^2(x) - 2\Cov[\bar D_j(x), \bar N_j] f(x)}{n_r (\EE[N])^2}
\]
and we replace all the unknown quantities in this expression by their empirical values.

Here, the required dimension of the RQMC points is the (random) total number of inter-arrival 
times $A_j$ and service times $S_j$ that we need to generate during the day.  
It is approximately twice the number of customers that arrive during the day.
This number is unbounded, so the RQMC points must have unbounded (or infinite) dimension,
and one must be able to generate the points without first selecting a maximal dimension.
Recurrence-based RQMC point sets have this property;
they can be provided for instance by ordinary or polynomial Korobov lattice rules
\citep{vLEC00b,vLEC02a}, which    
are available in the hups package of SSJ \citep{iLEC16j}.

\subsubsection{Steady-state model.}

In a slightly different setting, we can assume that the single queue evolves in steady-state
over an infinite time horizon, under the additional assumptions that the $A_j$'s are i.i.d.
and the $S_j$'s are also i.i.d. 
Again, we want to estimate the density of the waiting time $W$ 
of a random customer.  In this case, the system regenerates whenever a new customer arrives
in an empty system. The regenerative cycles can be much shorter on average than 
for the previous case, unless the day is very short or the utilization factor of the system is close to 1.
The CDE has exactly the same form, apart from the different definition of regenerative cycle.
In this case $n$ represents the number of regenerative cycles, 
$N_i$ is the number of customers in the $i$th cycle and $D_i(x)$ is the realization of
$D(x)$ over the $i$th cycle.

In both settings, one could also hide $A_j$ instead of $S_{j-1}$.  
The density estimator is similar and easy to derive.  
Intuition says that this should be a better choice if $A_j$ has more variance than $S_{j-1}$.

\subsubsection{The GLRDE estimator.}

\hpierre{I think in the paper, we should add a section that describes this density estimation approach 
  in a more general setting, with our notation. This would be worthwhile, because things are
	not very clear and easy in Peng's papers.}

\cite{oPEN20a}, Section~4.2.2., show how to construct a GLRDE for the 
density of the \emph{sojourn time} of customer $j$ in this single-queue model.
The density of the \emph{waiting time} can be estimated as follows.
If the service times $S_j$ are lognormal with parameters $(\mu, \sigma^2)$, we can write 
\[
  X = W_j = \max(0,\, W_{j-1} + S_{j-1} - A_j) 
	    = \max(0,\, W_{j-1} + \exp[\sigma Z_{j-1} + \mu] - A_j) 
			=: h(\bY) 
\]
where $Z_{j-1}$ has the standard normal density $\phi$, and 
$\bY = (Y_1,Y_2,Y_3) = (Z_{j-1}, A_j, W_{j-1})$. 
When $W_j > 0$, taking the derivative of $h$ with respect to $Y_1=Z_{j-1}$
gives $h_1(\bY) = \exp[\sigma Z_{j-1} + \mu] \sigma = S_{j-1} \sigma$,
$h_{11}(\bY) = S_{j-1} \sigma^2$, and these derivatives are 0 when $W_j=0$.
We also have $\partial \log \phi(x) / \partial x = -x$, and therefore for $x > 0$, 
$f(x) = {\EE[L(x)]}/{\EE[N]}$ where $L(x) = \sum_{j=1}^N \II[W_j \le x] \cdot \Psi_j$ and
$\Psi_j = - (Z_{j-1} + \sigma) / (S_{j-1} \sigma)$.
We can do $n$ runs to estimate each of the two expectations in the ratio.
This provides a very similar density estimator as with the CDE in (\ref{eq:queue-cde0}),
but here $L(x)$ is discontinuous in $x$, whereas $D(x)$ in (\ref{eq:queue-cde0}) 
is continuous.

\subsubsection{Numerical results.}

For a numerical illustration, suppose the time is in minutes, let
the arrival process be Poisson with constant rate $\lambda = 1$, and 
the service times $S_j$ lognormal with parameters $(\mu, \sigma^2) = (-0.7, 0.4)$. 
This gives $\EE[S_j] = e^{-0.5} \approx 0.6065$ and $\Var[S_j] = e^{-1}(e^{0.4}-1) \approx 0.18093$.
For RQMC, we use infinite-dimensional RQMC points defined by Korobov lattice rules \citep{vLEC00b}
selected with \texttt{Lattice Builder} \citep{vLEC16a} using order-dependent weights 
$\gamma_{k} = 0.005^{k}$ for projections of order $k$. 
We do not use Sobol' points because 
with the available software, there is an upper bound on the dimension.

\paragraph{Finite-horizon case.}
For the finite-horizon case, take $\tau = 60$, {so $\EE[N]=60$, we only need
to estimate the numerator, and we have an unbiased density estimator all over $[0,\infty)$.}
%
\hpierre{What are the units here?   Arrival rate is $\lambda=1$ customer per hour, or per minute,
 or per second? Never said.  Also, what are the units for the service times? 
 And for the waiting times? }%
\hpierre{I think in this example, we know $\EE[N]$ exactly, and therefore we can replace it 
 by its exact value in the ratio estimator.   After that, we no longer have a ratio estimator,
 but an \emph{unbiased} estimator of the ratio! Amal, Florian: Is this what you have done?}%
\hpierre{The first thing to do is to make experiments with the CDE for this example
 combined with Monte Carlo only (no RQMC) and see what happens, for the two settings.
 We want to see if the MISE converges as $\cO(1/n)$ as expected.
 We also want to compare with the GLRDE.}%
\hpierre{For this example, RQMC might not help much when simulating days (first setting),
 but it is likely to help more for the steady-state setting, because the regenerative cycles 
 will then be much shorter on average.  
 Do RQMC with ``LCG point sets'' in SSJ (they are Korobov lattice points).  
 Good parameters were found using \texttt{Latnet Builder}.}
\hpierre{We estimate the density of the waiting time over the interval $[a,b]=[0.045,2.2]$. 
The left boundary was selected so that we could guarantee that for the smallest $n$ 
considered ($n=2^{14}$) there would still be one data point to the left of $a$. 
An experiment with $2^{20}$ samples showed that we cut off a mere $0.005\%$ of the mass on the left side and that the right boundary $b$ is the $99\%$ percentile. We also performed an experiment with $a$ being the $1\%$ percentile, 
but the results were almost indistinguishable.}
The results for $(a,b] = (0,2.2]$ are in Table~\ref{tab:queue-finite-horizon}.
\hpierre{We also want to compare with GLRDE with and without RQMC. We want to show that the CDE idea
  is good just by itself, not only RQMC.}%
Due to the large and random dimensionality of the required RQMC points, and more importantly the 
discontinuity of the derivative of the CDE with respect to the underlying uniforms
(because of the max, the HK variation is infinite), 
it was unclear if RQMC could bring any significant gain for this example.
The good surprise is that although RQMC does not improve $\tilde\nu$ significantly, 
it improves the IV itself by a factor of about $2^{7.5} \approx 180$
for $n=2^{19}$, which is quite significant.
We also see that CDE beats GLRDE by a factor of about 500 with MC and about 200 with RQMC.
\hflorian{I know, this plot is rather meaningless, but I wanted to guarantee that everything 
 works correctly, so I estimated the variance for every $n=2^{14},\ldots,2^{19}$ individually. 
 This is what is contained in the plot.}

\begin{table}[!htbp]	
	\centering
	\caption{Values of $\hat\nu$ and e19 for the single queue example, finite-horizon case.}
	\label{tab:queue-finite-horizon}
\small
	\begin{tabular}{l | l| c c }
	\hline
	&       	  & $\hat\nu$	&  e19 \\
	\hline
	\multirow{3}{*}{CDE}
	& MC 		    & 1.00	& 24.8\\
	& Lat+s  	  & 0.99	& 32.3\\
	& Lat+s+b	  & 1.02	& 32.3\\
	\hline
	\multirow{3}{*}{GLRDE}
	& MC 		    & 1.00	& 15.8 \\
	& Lat+s  	  & 1.03	& 24.6 \\
	& Lat+s+b	  & 1.08	& 25.0 \\
	\hline
	\end{tabular}
\end{table}

\ifplots   

\begin{figure}[!htbp]
  \begin{tikzpicture} 
    \begin{axis}[ 
      title ={},
		  width=.45\columnwidth,
		  height=.30\columnwidth,
      xlabel=$x$,
      cycle list name=dens,
       grid,
      ] 
      \addplot table[x=x,y=y] { 
      x  y 
      0.002794244685024697  0.41282111409741357 
      0.05100760643872949  0.43273563102311124 
      0.09480209234283196  0.4499388862925675 
      0.1501686309843965  0.46442981722169113 
      0.18087585814720453  0.46751174004078616 
      0.23173469853420214  0.46509792119086385 
      0.2745770324729436  0.4575192082719382 
      0.3158223173476388  0.44687290961814663 
      0.3549917450287107  0.4347479350273885 
      0.41262874921755405  0.4148923269254528 
      0.4526622170158059  0.4005793314919085 
      0.49302141005928946  0.3860256487355689 
      0.5351792529475141  0.3711095680692716 
      0.5772883172002552  0.3564109132046836 
      0.6294213856127225  0.3390334389702375 
      0.6798919802041565  0.32304954519088047 
      0.7105774487545969  0.31377636159236 
      0.7487513989364474  0.30239523594559947 
      0.8132613515826783  0.2844652979188286 
      0.8391568989624817  0.27758144980765465 
      0.89657791433721  0.26321279288179117 
      0.927993701144748  0.25567810473772595 
      0.969770829312581  0.2460244262962898 
      1.016789912831574  0.2356645048824472 
      1.070934381303925  0.22456853329996515 
      1.1042317533363135  0.2179480673514368 
      1.1555408610469893  0.20827369763523002 
      1.1925359860149876  0.2015562900856097 
      1.2456623655221826  0.1923843554353138 
      1.2929245026880636  0.18466436198560907 
      1.339549921931323  0.1774130528931131 
      1.3833389442024058  0.17088008409328664 
      1.4265571720688355  0.16472537702868378 
      1.4679794058873588  0.15905284195422323 
      1.5070853102355601  0.15392596567063785 
      1.541747168195297  0.14955321027419316 
      1.5849167630729935  0.14430279767023022 
      1.6485360778797198  0.136940727625966 
      1.6935106686493893  0.13205168883677923 
      1.7265332771956272  0.12853110011655644 
      1.7729181958360098  0.12385436344078886 
      1.8188156111989284  0.11946268370870343 
      1.8480443137230393  0.11673487404654535 
      1.912715583212404  0.11092988576537006 
      1.9446093502311843  0.10821047703416542 
      1.992123678438767  0.10430887574466767 
      2.0398972000471822  0.10057069129559998 
      2.08264753147836  0.09736248877472524 
      2.1149911053693273  0.09501210513696762 
      2.1678319475979184  0.09134568564384005 
 }; 
       \addlegendentry{density}
  \end{axis}
  \end{tikzpicture}
\hfill
   \begin{tikzpicture} 
    \begin{axis}[ 
      width = 0.5\textwidth,
	    height = 5cm,
      xlabel=$\log n$,
      ylabel=$\log{\rm IV}$,
       grid,
       legend style={at={(1.25,1.1)},anchor=north east},   
       cycle list name=comparison,
      ] 
      \addplot table[x=logN,y=logIV] { 
      logN  logIV 
      14.0  -19.853407349363017 
      15.0  -20.837730535905706 
      16.0  -21.835504415895194 
      17.0  -22.843850664973367 
      18.0  -23.842744090719222 
      19.0  -24.842239203877426 
 }; 
      \addlegendentry{MC-CDE}
%
%
      \addplot table[x=logN,y=logIV] { 
      logN  logIV 
      14.0  -10.855182045494894 
      15.0  -11.838606936573392 
      16.0  -12.83926039468427 
      17.0  -13.83398707870161 
      18.0  -14.843645560130184 
      19.0  -15.839419424762527 
 }; 

      \addlegendentry{MC-GLRDE}
      \addplot table[x=logN,y=logIV] { 
      logN  logIV 
      14.0  -27.29067223962574 
      15.0  -28.07666259760639 
      16.0  -29.294888069803527 
      17.0  -30.021630892383126 
      18.0  -31.12253797408318 
      19.0  -32.25043212022527 
 }; 
      \addlegendentry{Lat+s-CDE}
      \addplot table[x=logN,y=logIV] {
      logN  logIV 
       14.0	 -19.654037
       15.0	 -20.350068
       16.0	 -21.478259
       17.0	 -22.856442
       18.0	 -23.549005
       19.0  -24.636864 
 }; 
      \addlegendentry{Lat+s-GLRDE}
    \end{axis}
  \end{tikzpicture}
  \caption{Estimated density (left) and $\log\IV$ as a function of $\log n$ (right) 
    for the single queue over a finite-horizon.}
  \label{fig:queue-finite-horizon}
\end{figure}

\fi  

\paragraph{Steady-state case.}

We performed a similar experiment using regenerative simulation for the steady-state model.
The density is similar but not exactly the same as in the finite-horizon case.
The results are in Table~\ref{tab:queue-steady-state}.
They are similar to those of the finite-horizon case, with similar empirical convergence rates, 
and the IV for $n = 2^{19}$ is again about 180 times smaller with CDE+RQMC compared to CDE+MC.
The IV for GLRDE with $n = 2^{19}$ is roughly {300} times larger than with CDE with MC and 
{200} times larger than with CDE with RQMC. 
The only important difference is that here, the IV is about $30$ times \emph{larger}
than in the finite-horizon case, for all the methods.  
The explanation is that in the finite-horizon case, we simulate $n$ runs with about 60 customers
per run, whereas in the steady-state case, we have about 2.5 customers per regenerative cycle on average,
so we simulate about 25 times fewer customers.
\hpierre{Is that right?  How many customers per cycle on average?}  
\hflorian{60 in the finite-horizon case and 2.5 in the steady state case (the medians are 60 and 1 resp.). 
 In both cases I estimated with $n=2^{19}$ points. Interestingly the largest cycle is 98 for finite-horizon
 and 114 for the steady-state. I.e., we need higher dimensional point sets in the steady state case.}%
Interestingly, the fact that we use much more coordinates of the RQMC points in the finite-horizon case
(on average) makes no significant difference.  
A similar observation was made by \cite{vLEC00b}, Section 10.3, 
who compared finite-horizon runs of 5000 customers each on average, with regenerative simulation,
in the context of estimating the probability of a large waiting time using RQMC.
The reason why RQMC performs well even for a very large time horizon is that the integrand has 
\emph{low effective dimension in the successive-dimensions sense}
(as defined by these authors).
%
Appendix C of the Supplement provides additional plots for this example.

\begin{table}[!htbp]	
	\centering
	\caption{Values of $\hat\nu$ and e19 for the single queue example, steady-state case.}
	\label{tab:queue-steady-state}
\small
	\begin{tabular}{l | l| c c }
	\hline
	&       	  & $\hat\nu$	&  e19 \\
	\hline
	\multirow{3}{*}{CDE}
	& MC 		    & 0.99	& 19.9\\
	& Lat+s  	  & 1.04	& 27.6\\
	& Lat+s+b	  & 1.08	& 27.8\\
	\hline
	\multirow{3}{*}{GLRDE}
	& MC 		    & 0.99	& 11.5\\
	& Lat+s  	  & 1.20	& 20.1\\
	& Lat+s+b	  & 1.21	& 20.4\\
	\hline
	\end{tabular}
\end{table}

\ifplots  

\begin{figure}[!htbp]
  \begin{tikzpicture} 
    \begin{axis}[ 
      title ={},
		width=.45\columnwidth,
		height=.35\columnwidth,
      xlabel=$x$,
      cycle list name=dens,
       grid,
      ] 
      \addplot table[x=x,y=y] { 
      x  y 
      0.005906017175165837  0.39742752587153185 
      0.10781153179095096  0.43589591609120293 
      0.2003771497246221  0.44892750275994364 
      0.3174018791261108  0.42720191770134747 
      0.3823057910838641  0.4066324644046126 
      0.4898028855382  0.3703009946763053 
      0.5803560004541763  0.33985152408428554 
      0.6675335343938729  0.31233946371649896 
      0.7503234610834112  0.28782182578232635 
      0.872147129028012  0.2557573228592747 
      0.9567633223288625  0.23534128033385596 
      1.042067980352589  0.21661513772725566 
      1.1311743300936095  0.19963333655108015 
      1.220177579536903  0.1836563945861569 
      1.3303679286814363  0.16533184260031605 
      1.43704441270424  0.15001706772845466 
      1.5019023348676708  0.14143204714151175 
      1.5825881841156728  0.13088798567881774 
      1.7189387658452064  0.11504780412967734 
      1.7736725364434272  0.1088141861309927 
      1.8950396825763756  0.09725014657781669 
      1.9614412319650356  0.09211788155390024 
      2.049742889228864  0.08526292123636937 
      2.1491241339394636  0.07823271660814425 
      2.263565851392387  0.07137039105683644 
      2.3339443877335717  0.06755916979791599 
      2.442393183576591  0.06184584019973621 
      2.520587424986224  0.05748247173296238 
      2.6328772725809766  0.051925064249434244 
      2.7327722443179527  0.04785302190670782 
      2.8313214259002955  0.04402821664601987 
      2.923875495700539  0.04054204361942753 
      3.015223113690948  0.03810055665773806 
      3.1027746533528267  0.03532923654129244 
      3.1854303148160703  0.03314924537992229 
      3.2586928782309683  0.030895804771667913 
      3.3499377037679183  0.02926389232797021 
      3.484405800973044  0.0265271200409409 
      3.579465731463482  0.02448879174596624 
      3.6492635177089396  0.023049062342575997 
      3.7473043684715663  0.021253165221502675 
      3.844314814579553  0.019502573188140145 
      3.906093663096424  0.018393221307815982 
      4.042785209971672  0.016385551548199307 
      4.110197035715912  0.015529812949185463 
      4.210625047609212  0.01430736523687535 
      4.311600900099727  0.013176668284291716 
      4.401959555170171  0.012294833956339717 
      4.470322109076078  0.011505996756587677 
      4.5820084346956005  0.01078568327727437 
 }; 
       \addlegendentry{density}
  \end{axis}
  \end{tikzpicture}
\hfill		
		\begin{tikzpicture} 
    \begin{axis}[ 
		width=.45\columnwidth,
		height=.35\columnwidth,
      xlabel=$\log n$,
      ylabel=$\log{\rm IV}$,
      cycle list name=comparison,
	  legend style={at={(1.45,1.)},anchor=north east},     
       grid,
	] 
      \addplot table[x=logN,y=logIV] { 
      logN  logIV 
      14.0  -14.90092233824423 
      15.0  -15.922674757355777 
      16.0  -16.91721009636318 
      17.0  -17.877637604940823 
      18.0  -18.87594158390892 
      19.0  -19.882423945433235 
 };
      \addlegendentry{MC-CDE}
%
      \addplot table[x=logN,y=logIV] { 
      logN  logIV 
      14.0  -6.620403056650534 
      15.0  -7.563311157521403 
      16.0  -8.471188538507484 
      17.0  -9.459099061707633 
      18.0  -10.512135806985476 
      19.0  -11.478168194523814 
 }; 
      \addlegendentry{MC-GLRDE}
      \addplot table[x=logN,y=logIV] { 
      logN  logIV 
      14.0  -22.392693569981542 
      15.0  -23.382865497835372 
      16.0  -24.7233343106602 
      17.0  -25.55620435940835 
      18.0  -26.517461119406576 
      19.0  -27.647943633942436 
 }; 
      \addlegendentry{Lat+s-CDE}
      \addplot table[x=logN,y=logIV] { 
      logN  logIV 
      14.0  -14.136963821198073 
      15.0  -15.036836033546011 
      16.0  -16.562985946524233 
      17.0  -17.554005390299235 
      18.0  -18.722183207944685 
      19.0  -20.13907459572688 
 }; 
      \addlegendentry{Lat+s-GLRDE}
    \end{axis}
  \end{tikzpicture}
  \caption{Estimated density (left) and $\log\IV$ as a function of $\log n$ (right) 
    for the single queue in steady-state.}
  \label{fig:queue-steady-state}
\end{figure}

\fi  

\subsection{Making a change of variable}
\label{sec:change-variable}

In many situations, $X = h(\bY)$ for a random vector $\bY$
and hiding a single coordinate of $\bY$ does not provide a very effective CDE.
But sometimes, after an appropriate change of variable $\bY = g(\bZ)$, 
hiding one coordinate of the random vector $\bZ$ can provide a much more effective CDE.
We will use this technique in Section~\ref{sec:function-multinormal}.
We describe it here in a separate subsection because it can be useful for a much
wider range of applications.

Specifically, let $\bZ_{-j}$ denote the vector $\bZ$ with $Z_j$ (the $j$th coordinate) removed,
and let $\gamma(z) = \gamma(z; \bZ_{-j}) = h(g(z; \bZ_{-j}))$ denote the value of 
$h(\bY)$ as a function of $Z_j = z$ when $\bZ_{-j}$ is fixed.
We assume in the following that for almost any realization of $\bZ_{-j}$, 
$\gamma(z; \bZ_{-j})$ is a monotone non-decreasing and differentiable function of $z$, 
so that $\gamma^{-1}(x) = \inf \{z \in\RR : \gamma(z) \ge x\}$
is well defined for any $x$.  
We also assume that $Z_j$ has density $\varphi$ and is independent of $\bZ_{-j}$ (to simplify).
Conditional on $\bZ_{-j}$, we have
\[
   \PP[x < h(\bY) \le x+\delta \mid \bZ_{-j}] 
 = \PP[x < \gamma(Z_j) \le x+\delta \mid \bZ_{-j}] 
 = \PP[z < Z_j \le z + \Delta \mid \bZ_{-j}]
 \approx \varphi(z) \Delta 
\]
where $z = \gamma^{-1}(x)$ and $z + \Delta = \gamma^{-1}(x+\delta)$.
Taking the limit gives
\[
  f(x \mid \bZ_{-j}) 
	     = \lim_{\delta\to 0} \frac{\PP[z < Z_j \le z + \Delta \mid \bZ_{-j}]}{\delta} 
	     = \lim_{\delta\to 0} \frac{\varphi(z) \Delta}{\delta}
			 = \frac{\varphi(z)}{\gamma'(z)}
			 = \frac{\varphi(\gamma^{-1}(x))}{\gamma'(\gamma^{-1}(x))},
\]
assuming that the latter is well defined.
In case there are closed-form formulas for $\gamma^{-1}$ and $\gamma'$, 
this CDE can be evaluated directly. 
Otherwise, $z= \gamma^{-1}(x)$ can often be computed by a 
few iterations of a root-finding algorithm.  
Since $\gamma$ and its inverse $\gamma^{-1}$ depend on $\bZ_{-j}$, 
this could mean inverting a different function for each sample realization.  
Our next example will show that the approach could nevertheless bring a huge benefit.

\subsection{A function of a multivariate normal vector}
\label{sec:function-multinormal}

We consider a multivariate normal vector $\bY = (Y_1,\dots,Y_s)^\tr$ 
\mpierre{(where $^{\tr}$ means transposed)} defined via
$Y_j = Y_{j-1} + \mu_j + \sigma_j Z_j$ with $Y_0=0$, the $\mu_j$ and $\sigma_j > 0$ are constants,
and the $Z_j$ are independent $\cN(0,1)$ random variables, with cdf $\Phi$ and density $\phi$.
Let $X = \bar S = (S_1 + \cdots + S_s)/s$ where $S_j = S_0 e^{Y_j}$ for some constant $S_0>0$.
We want to estimate the density of $X$ over some interval $(a,b) = (K, K+c)$ 
where $K\ge 0$ and $c > 0$.  
This is the same as estimating the density of $\max(0, \bar S-K)$, which may represent the 
payoff of a financial contract, for example \citep{fGLA04a}.
A simple way to define the CDE here is to hide $Z_s$.
The conditional cdf is 
$\PP[X \le x \mid \bZ_{-s}] = \PP[Z_s \le W(x)] = \Phi(W(x))$ where
\[
 W(x) = (\ln[sx - (S_1 + \cdots + S_{s-1})/S_0] - {\ln S_0} - Y_{s-1} - {\mu_s}) / {\sigma_s}.
\]
Taking the derivative with respect to $x$ gives the unbiased CDE
\begin{eqnarray}
  f(x \mid \bZ_{-s}) = \frac{\partial}{\partial x} \PP[\bar S \le x \mid \bZ_{-s}] 
	&=& \phi(W(x)) W'(x)  ~=~  \frac{\phi(W(x)) s}{[sx - (S_1 + \cdots + S_{s-1})/S_0] {\sigma_s}}.	
\label{eq:asian-sequential-cde}
\end{eqnarray}
Unfortunately, this \emph{sequential} CDE is usually rather spiky, 
because hiding only this $Z_s$ does not remove much information, 
and then the conditional density has a large variance.

We now describe a less obvious but more effective conditioning approach.
The goal is to hide a variable that contains more information.
For this, we generate the vector $\bY$ using a Brownian bridge construction in which
the $Z_j$'s are used in a different way, as follows \citep{vCAF97a,fGLA04a}.
Let $\bar\mu_j = \mu_1 + \cdots + \mu_j$ and $\bar\sigma_j = \sigma_1 + \cdots + \sigma_j$,
for $j=1,\dots,s$.
With this construction, we first sample $Y_s = \bar\mu_s + \bar\sigma_s Z_s$.
Then, given $Y_s = y_s$, we put $j_2 = \lfloor s/2\rfloor$, and we sample 
$Y_{j_2}$ from its normal distribution conditional on $Y_s=y_s$, 
which is normal with mean $y_s \bar\mu_{j_2} / \bar\mu_s$ and variance 
$(\bar\sigma_s - \bar\sigma_{j_2}) \bar\sigma_{j_2} / \bar\sigma_s$.
This uses the fact that if $X_1$ and $X_2$ are independent and normal, then conditional
on $X_1+X_2 = \bar x$, $X_1$ is normal with mean $\bar x \EE[X_1]/\EE[X_1+X_2]$ and 
variance $\Var[X_1]\Var[X_2]/\Var[X_1+X_2]$.
Then we put $j_3 = \lfloor j_2/2\rfloor$ and we sample $Y_{j_3}$ conditionally on $Y_{j_2}$,
then we put $j_4 = \lfloor (j_2+s)/2\rfloor$ and we sample $Y_{j_4}$ conditionally on $(Y_{j_2}, Y_s)$,
and so on, until all the $Y_j$'s are known.

For the CDE, we hide again $Z_s$, but now $Z_s$ has much more impact on the payoff,
because all the $Y_j$'s depend on $Z_s$.
This makes the conditional density much less straightforward to compute, 
but we can proceed as follows.
To avoid sampling $Z_s$, we sample $Y_1,\dots,Y_{s-1}$ conditional on $Z_s = z_s = 0$, 
which will give say $Y_1^0,\dots,Y_{s-1}^0$,
and then write $X$ as a function of $z = z_s$ conditional on these values,
that is, conditional on $\bZ_{-s} = (Z_1,\dots,Z_{s-1})$.
We have $Y_s = Y_s^0 + \bar\sigma_s Z_s$
and $Y_j = Y_j^0 + (\bar\mu_j/\bar\mu_s) \bar\sigma_s Z_s$.  Then,
\[
 X = \bar S = \frac{S_0}{s} \sum_{j=1}^s e^{Y_j} 
   =  \frac{S_0}{s} \sum_{j=1}^s \exp[Y_j^0 + Z_s (\bar\mu_j/\bar\mu_s) \bar\sigma_s].
\]
This fits the framework of Section~\ref{sec:change-variable}, with $j=s$, 
\[
  \gamma(z) =  \frac{S_0}{s} \sum_{j=1}^s \exp[Y_j^0 + z (\bar\mu_j/\bar\mu_s) \bar\sigma_s]
\quad\mbox{ and }\quad
  \gamma'(z) = \frac{S_0}{s} \sum_{j=1}^s \exp[Y_j^0 + z (\bar\mu_j/\bar\mu_s) \bar\sigma_s] 
	                                      (\bar\mu_j/\bar\mu_s) \bar\sigma_s.
\]
The CDE at $x = \gamma(z)$ is then $f(x \mid \bZ_{-s}) = \phi(z) / \gamma'(z)$.
We call it the \emph{bridge} CDE.

To compute this density at a specified $x$ we need $z = \gamma^{-1}(x)$,
We have no explicit formula for $\gamma^{-1}$ in this case, 
but we can compute a root of $\gamma(z) - x = 0$ numerically.
To evaluate the density at the $n_e$ evaluation points $e_{1},\dots, e_{n_e}$ in $(a,b)$,
we  first compute $x_* = \gamma(0)$ and let $j_*$ be the smallest $j$ for which $e_{j} \ge x_*$.  
We compute $z = w_{j_*}$ such that $\gamma(w_{j_*}) = e_{j_*}$.  
This can be done via Newton iteration, 
$z_k = z_{k-1} - (\gamma(z_{k-1})-e_{j^*}) / \gamma'(z_{k-1})$, starting with $z_0 = 0$.
Then, for $j = j_*+1,\dots, n_e$, we use again Newton iteration to find $z = w_{j}$ 
such that $\gamma(w_{j}) = e_{j}$, starting at $z_0 = w_{j-1}$.
We do the same to find $z = w_{j}$ such that $\gamma(w_{j}) = e_{j}$ for $j = j_*-1,\dots, 1$,
starting at $z_0 = w_{j+1}$.
This provides the point $w_j$ required to evaluate the conditional density at $e_j$, for each $j$.
We must repeat this procedure for each realization of $\bZ_{-j}$, because the function $\gamma$
depends on $\bZ_{-j}$.  However, the gain in accuracy is more significant than the cost of 
additional computations.
This conditioning differs from the simpler ones used by 
\cite{fBOY97a} and \cite{oHEI15a} for barrier options.

For a numerical illustration, we take $S_0=100$, $s=12$, $\mu_j = 0.00771966$
and $\sigma_j=0.035033$ for all $j$, and $K=101$. 
We estimate the density of the payoff over $[a,b]=[101,\, 128.13]$. 
To approximate the root of $\gamma(z)-x = 0$ for the bridge CDE, we use five Newton iterations;
doing more makes no significant difference.
The results are in Table~\ref{tab:asianBridge},
with additional plots in the Supplement.
RQMC with the bridge CDE performs extremely well.
For example, for Sob+LMS, the MISE with $n=2^{19}$ is approximately $2^{-46.9}$,
which is about $2^{19}$ (half a million) times smaller than for the same CDE with MC, 
and it decreases as $\cO(n^{-2})$.
With a KDE, the MISE with $n=2^{19}$ is about $2^{21} \approx 2$ million times larger 
with the same Sobol' points and $2^{26}\approx  67$ million times larger with MC.
With the sequential CDE, RQMC is ineffective and the IV of the MC estimator is also 
quite large, as expected.
To illustrate the behavior of the sequential and bridge CDEs, 
Figure~\ref{fig:asian-cde} plots five single realizations of each, using the same 
horizontal scale. The sequential CDE has much more spiky realizations than the bridge CDE,
and this explains why the latter performs much better.

\begin{table}[!htbp]	
	\centering
	\caption{Values of $\hat\nu$ and e19 for the Asian option, with sequential 
	  and bridge CDE constructions.}
	\label{tab:asianBridge}
\small
	\begin{tabular}{l | l| c c }
	\hline
  	&     & $\hat\nu$	&  e19 \\
	\hline
	\multirow{2}{*}{sequential KDE}
		& MC 			 & 0.78		&  20.4 \\
		& Sob+LMS   &   0.76		&  20.6 \\
 \hline
	\multirow{4}{*}{sequential CDE}
		& MC 			 & 1.00		& 19.9 \\
		& Lat+s	   & 1.07 		& 20.3 \\
		& Lat+s+b	 & 1.01  		& 20.1 \\
		& Sob+LMS   &  1.00 		& 20.0\\
 \hline
	\multirow{4}{*}{bridge CDE}
		& MC 			 & 1.04		& 27.9 \\
		& Lat+s	   & 1.60		& 40.0\\
		& Lat+s+b	 & 1.74		& 45.0\\
		& Sob+LMS 	& 2.01 		& 46.9\\
	\hline
	\end{tabular}
\end{table}

\begin{figure}[hbt]
\begin{center}
\begin{tikzpicture} 
  \begin{axis}[ 
  	cycle list name=canti,
  legend style={at={(0.6,1.3)}, anchor={north west}},
  xlabel=$x$,
  ylabel=,
  grid,
  xmin=101,
  xmax=128.13,
  width = 0.5\textwidth,
  height = 5cm,
  ] 
     \addplot table[x=x,y=density] { 
      x  density 
      101.05298828125  1.2129642531673865 
      101.15896484375  1.2861792185105183 
      101.26494140625  1.2133794530608826 
      101.37091796875  1.0226287270798182 
      101.47689453125  0.7729905479498063 
      101.58287109375  0.5260195872847049 
      101.68884765625  0.3234220676636093 
      101.79482421875  0.18029392624962115 
      101.90080078125  0.09142784731754207 
      102.00677734375  0.04231025835857486 
      102.11275390625  0.017923010879018656 
      102.21873046875  0.00697027481036857 
      102.32470703125  0.0024956735722511483 
      102.43068359375  8.24901507192927E-4 
      102.53666015625  2.5236249071068254E-4 
      102.64263671875  7.16381671318098E-5 
      102.74861328125  1.891508410423375E-5 
      102.85458984375  4.656118060796985E-6 
      102.96056640625  1.0709314950994514E-6 
      103.06654296875  2.306511208776754E-7 
      103.17251953125  4.661283047990821E-8 
      103.27849609375  8.856833512979716E-9 
      103.38447265625  1.5852958541432336E-9 
      103.49044921875  2.677976617020637E-10 
      103.59642578125  4.2770563214020014E-11 
      103.70240234375  6.469557592827194E-12 
      103.80837890625  9.283656528774248E-13 
      103.91435546875  1.265833872606632E-13 
      104.02033203125  1.6425619021851825E-14 
      104.12630859375  2.031446074360993E-15 
      104.23228515625  2.3980397639528096E-16 
      104.33826171875  2.705719557782043E-17 
      104.44423828125  2.9219443011480014E-18 
      104.55021484375  3.0240692655121437E-19 
      104.65619140625  3.003248917495191E-20 
      128.07701171875  0.0 
 }; 
      \addlegendentry{}
      \label{}
      \addplot table[x=x,y=density] { 
      x  density 
      101.05298828125  1.9123622117564208E-23 
      101.15896484375  2.4189708315487618E-21 
      101.26494140625  2.247392861167212E-19 
      101.37091796875  1.554945962757224E-17 
      101.47689453125  8.11704544845673E-16 
      101.58287109375  3.236442914541436E-14 
      101.68884765625  9.971734623430413E-13 
      101.79482421875  2.4003609785404838E-11 
      101.90080078125  4.5614107043123076E-10 
      102.00677734375  6.910551919447181E-9 
      102.11275390625  8.424955067323752E-8 
      102.21873046875  8.338824089615543E-7 
      102.32470703125  6.757242614597531E-6 
      102.43068359375  4.518801123909052E-5 
      102.53666015625  2.5127849899893176E-4 
      102.64263671875  0.0011702902616062493 
      102.74861328125  0.004596359535231257 
      102.85458984375  0.015323211727757797 
      102.96056640625  0.04363148731649812 
      103.06654296875  0.10674245934262828 
      103.17251953125  0.22563932727621114 
      103.27849609375  0.4143562390327186 
      103.38447265625  0.6644304728840253 
      103.49044921875  0.934924946571858 
      103.59642578125  1.1598341128387142 
      103.70240234375  1.274259457322874 
      103.80837890625  1.2451710534045426 
      103.91435546875  1.0866645830735506 
      104.02033203125  0.8502911058757182 
      104.12630859375  0.5988026607065099 
      104.23228515625  0.3809025363742633 
      104.33826171875  0.21961571140073533 
      104.44423828125  0.1151532504778794 
      104.55021484375  0.05508555004960808 
      104.65619140625  0.02411447385149379 
      104.76216796874999  0.009688880859999844 
      104.86814453125  0.0035830431310695005 
      104.97412109375  0.001222903231376095 
      105.08009765625  3.862133404922128E-4 
      105.18607421875  1.131481997449974E-4 
      105.29205078125  3.082491723875005E-5 
      105.39802734375  7.827071626399807E-6 
      105.50400390625  1.856570213520015E-6 
      105.60998046875  4.1226241370329413E-7 
      105.71595703125  8.587913757974793E-8 
      105.82193359375  1.6815959389531106E-8 
      105.92791015625  3.101078687502157E-9 
      106.03388671875  5.395953784041564E-10 
      106.13986328125  8.874960190952247E-11 
      106.24583984375  1.3821601834265845E-11 
      106.35181640625  2.041591532641467E-12 
      106.45779296875  2.864820588625409E-13 
      106.56376953125  3.8248955097060064E-14 
      106.66974609374999  4.866174422552312E-15 
      106.77572265625  5.907883855216414E-16 
      106.88169921875  6.854273008075048E-17 
      106.98767578125  7.609649204134026E-18 
      107.09365234375  8.094885778880432E-19 
      107.19962890625  8.261329672229851E-20 
      128.07701171875  1.9227520787021994E-308 
 }; 
      \addlegendentry{}
      \label{}
      \addplot table[x=x,y=density] { 
      x  density 
      101.05298828125  0.0 
      116.94947265625  4.2862060865061176E-20 
      117.05544921875  1.5390811758802903E-18 
      117.16142578124999  4.549358280153773E-17 
      117.26740234374999  1.1148184469157806E-15 
      117.37337890625  2.280065988170592E-14 
      117.47935546875  3.917153039175808E-13 
      117.58533203125  5.687722049910855E-12 
      117.69130859375  7.02096866639913E-11 
      117.79728515625  7.409357401604227E-10 
      117.90326171875  6.720736495753675E-9 
      118.00923828124999  5.266640660943203E-8 
      118.11521484375  3.5831397407190193E-7 
      118.22119140625  2.1264239691165688E-6 
      118.32716796874999  1.1057313293739238E-5 
      118.43314453125  5.0598748337558464E-5 
      118.53912109375  2.0460695834355601E-4 
      118.64509765625  7.340383248552545E-4 
      118.75107421875  0.002345270010674671 
      118.85705078125  0.006697844488359861 
      118.96302734375  0.017158335318223906 
      119.06900390625  0.03956247137203921 
      119.17498046875  0.08237084037004092 
      119.28095703125  0.15534779924147016 
      119.38693359375  0.26618681152276763 
      119.49291015624999  0.41560146458589864 
      119.59888671875  0.5929106466455817 
      119.70486328125  0.7749809041702144 
      119.81083984374999  0.9304813392529754 
      119.91681640624999  1.028783315015325 
      120.02279296875  1.0499908913091525 
      120.12876953125  0.9915192920848788 
      120.23474609375  0.8682482864440423 
      120.34072265625  0.7065663456308718 
      120.44669921875  0.5354695635953234 
      120.55267578125  0.37867348302387105 
      120.65865234374999  0.25037378447585323 
      120.76462890625  0.1550685310884436 
      120.87060546875  0.0901278374003366 
      120.97658203124999  0.04924453105921065 
      121.08255859375  0.025337219445487617 
      121.18853515625  0.01229633579698173 
      121.29451171875  0.005637661514229238 
      121.40048828124999  0.002445666367434996 
      121.50646484375  0.0010053501614964312 
      121.61244140625  3.921803575018718E-4 
      121.71841796875  1.4538155508263866E-4 
      121.82439453125  5.1283232600020526E-5 
      121.93037109375  1.723667222046597E-5 
      122.03634765625  5.5270722613427596E-6 
      122.14232421874999  1.6929158141435315E-6 
      122.24830078125  4.958982687954993E-7 
      122.35427734375  1.39081591099274E-7 
      122.46025390624999  3.7389894301403564E-8 
      122.56623046874999  9.645383292810721E-9 
      122.67220703125  2.3901464491075376E-9 
      122.77818359375  5.695277299053933E-10 
      122.88416015625  1.306240382764362E-10 
      122.99013671875  2.8864839865252724E-11 
      123.09611328125  6.151213798650586E-12 
      123.20208984375  1.265302046114903E-12 
      123.30806640624999  2.514514286441656E-13 
      123.41404296875  4.831860521823652E-14 
      123.52001953125  8.985420206235061E-15 
      123.62599609374999  1.618375408064165E-15 
      123.73197265625  2.825402628930583E-16 
      123.83794921875  4.7849282488037427E-17 
      123.94392578124999  7.866637765321306E-18 
      124.04990234374999  1.256426995193329E-18 
      124.15587890625  1.9508635965243468E-19 
      124.26185546875  2.9468425962570736E-20 
      128.07701171875  8.626767902882596E-56 
 }; 
      \addlegendentry{}
      \label{}
      \addplot table[x=x,y=density] { 
      x  density 
      101.05298828125  0.0 
      105.82193359375  3.8750462794379746E-19 
      105.92791015625  2.0549749268586696E-17 
      106.03388671875  8.484746217276715E-16 
      106.13986328125  2.755857883309227E-14 
      106.24583984375  7.110643601517244E-13 
      106.35181640625  1.4710397320105633E-11 
      106.45779296875  2.461655854787127E-10 
      106.56376953125  3.3600789981683904E-9 
      106.66974609374999  3.7708766412412965E-8 
      106.77572265625  3.505807740978612E-7 
      106.88169921875  2.719631650647321E-6 
      106.98767578125  1.772483284339488E-5 
      107.09365234375  9.768717037889876E-5 
      107.19962890625  4.5811772645868316E-4 
      107.30560546875  0.0018389798405399108 
      107.41158203124999  0.006354708495608 
      107.51755859375  0.019005555911009525 
      107.62353515625  0.049450930224275165 
      107.72951171875  0.11249181256823727 
      107.83548828125  0.2247863227159032 
      107.94146484375  0.39635351286516235 
      108.04744140625  0.6193509998738359 
      108.15341796875  0.8612519845190878 
      108.25939453125  1.0700042549621827 
      108.36537109375  1.1922128654013795 
      108.47134765625  1.1956896483910875 
      108.57732421875  1.0831745306134226 
      108.68330078125  0.8893065124346851 
      108.78927734375  0.663861020847576 
      108.89525390625  0.4519806259160682 
      109.00123046875  0.2814952556568507 
      109.10720703125  0.16083201420172852 
      109.21318359375  0.08453136981733692 
      109.31916015624999  0.04097846055482922 
      109.42513671875  0.01836921844072283 
      109.53111328125  0.0076328485577616635 
      109.63708984375  0.002946915563050372 
      109.74306640625  0.0010595471582579335 
      109.84904296875  3.55545287294336E-4 
      109.95501953125  1.1158548292635903E-4 
      110.06099609374999  3.28202460518417E-5 
      110.16697265625  9.064610112118872E-6 
      110.27294921875  2.3553275528434836E-6 
      110.37892578125  5.768213563029721E-7 
      110.48490234375  1.333781167222257E-7 
      110.59087890625  2.9168901105556154E-8 
      110.69685546875  6.043124143104728E-9 
      110.80283203125  1.1879506924059336E-9 
      110.90880859375  2.2192029468215543E-10 
      111.01478515625  3.945493227077127E-11 
      111.12076171875  6.685499309295688E-12 
      111.22673828125  1.081176815818724E-12 
      111.33271484375  1.6709840045793037E-13 
      111.43869140625  2.4712915936442706E-14 
      111.54466796874999  3.501846299851649E-15 
      111.65064453125  4.76013582753606E-16 
      111.75662109375  6.214436820631115E-17 
      111.86259765625  7.800801304381408E-18 
      111.96857421874999  9.42564256048443E-19 
      112.07455078125  1.0974415136411111E-19 
      112.18052734375  1.2325383467987164E-20 
      128.07701171875  1.0507260422669712E-225 
 }; 
      \addlegendentry{}
      \label{}
      \addplot table[x=x,y=density] { 
      x  density 
      101.05298828125  0.0 
      109.63708984375  5.351551393263135E-20 
      109.74306640625  2.8757182366729238E-18 
      109.84904296875  1.2132752916078831E-16 
      109.95501953125  4.058460332275914E-15 
      110.06099609374999  1.0863686391401207E-13 
      110.16697265625  2.3476235144587916E-12 
      110.27294921875  4.129956257130876E-11 
      110.37892578125  5.961830106172325E-10 
      110.48490234375  7.115642471431213E-9 
      110.59087890625  7.072521275256156E-8 
      110.69685546875  5.894360422945088E-7 
      110.80283203125  4.146085136110147E-6 
      110.90880859375  2.4767601488454582E-5 
      111.01478515625  1.264022627533958E-4 
      111.12076171875  5.542615485585199E-4 
      111.22673828125  0.0020995052116472524 
      111.33271484375  0.0069057419464684635 
      111.43869140625  0.019821891490465433 
      111.54466796874999  0.04988592453009286 
      111.65064453125  0.11058025924998334 
      111.75662109375  0.21683461701593512 
      111.86259765625  0.37769093270561027 
      111.96857421874999  0.5867211387774589 
      112.07455078125  0.8159679929105828 
      112.18052734375  1.0196520239633275 
      112.28650390625  1.1489345021855362 
      112.39248046875  1.171301803620516 
      112.49845703125  1.083877813816658 
      112.60443359375  0.9132343777921231 
      112.71041015624999  0.7027072037740568 
      112.81638671875  0.49522949353187945 
      112.92236328125  0.32053873135605365 
      113.02833984375  0.19105258395412347 
      113.13431640625  0.10513245074453155 
      113.24029296875  0.05354313717018276 
      113.34626953125  0.025297998305228157 
      113.45224609375  0.011114201358607884 
      113.55822265625  0.004550284045874561 
      113.66419921875  0.0017397660037686198 
      113.77017578125  6.224812296313695E-4 
      113.87615234375  2.0883547083803164E-4 
      113.98212890625  6.581944398285153E-5 
      114.08810546875  1.952435007929225E-5 
      114.19408203124999  5.460648962430144E-6 
      114.30005859375  1.4424594657404688E-6 
      114.40603515625  3.604754444606403E-7 
      114.51201171875  8.536054234654223E-8 
      114.61798828124999  1.918321692568471E-8 
      114.72396484375  4.097505158549027E-9 
      114.82994140625  8.330688862207482E-10 
      114.93591796875  1.6144127472818225E-10 
      115.04189453125  2.986136317063646E-11 
      115.14787109375  5.278811870807179E-12 
      115.25384765625  8.929889595240359E-13 
      115.35982421874999  1.4473502078686206E-13 
      115.46580078125  2.2502781215185336E-14 
      115.57177734375  3.359966573671375E-15 
      115.67775390624999  4.823423979807187E-16 
      115.78373046875  6.664508681851074E-17 
      115.88970703125  8.872154138815819E-18 
      115.99568359375  1.1391485013715866E-18 
      116.10166015624999  1.412051681142368E-19 
      116.20763671875  1.6914359298927245E-20 
 }; 
      \addlegendentry{}
      \label{}
      \addplot table[x=x,y=density] { 
      x  density 
      101.05298828125  0.2425928506334773 
      101.15896484375  0.2572358437021037 
      101.26494140625  0.2426758906121765 
      101.37091796875  0.20452574541596363 
      101.47689453125  0.1545981095899614 
      101.58287109375  0.10520391745694746 
      101.68884765625  0.06468441353292129 
      101.79482421875  0.03605878525472495 
      101.90080078125  0.018285569554736626 
      102.00677734375  0.008462053053825356 
      102.11275390625  0.0035846190257138658 
      102.21873046875  0.0013942217385555064 
      102.32470703125  5.004861629731492E-4 
      102.43068359375  1.740179036864035E-4 
      102.53666015625  1.0072819794192285E-4 
      102.64263671875  2.4838568574761185E-4 
      102.74861328125  9.230549238670981E-4 
      102.85458984375  0.003065573569163719 
      102.96056640625  0.008726511649598645 
      103.06654296875  0.021348537998749834 
      103.17251953125  0.04512787477780832 
      103.27849609375  0.08287124957791042 
      103.38447265625  0.13288609489386424 
      103.49044921875  0.18698498936793112 
      103.59642578125  0.23196682257629692 
      103.70240234375  0.25485189146586873 
      103.80837890625  0.2490342106810942 
      103.91435546875  0.21733291661473544 
      104.02033203125  0.17005822117514693 
      104.12630859375  0.11976053214130238 
      104.23228515625  0.0761805072748527 
      104.33826171875  0.04392314228014707 
      104.44423828125  0.02303065009557588 
      104.55021484375  0.011017110009921616 
      104.65619140625  0.004822894770298758 
      104.76216796874999  0.001937776171999969 
      104.86814453125  7.166086262139E-4 
      104.97412109375  2.44580646275219E-4 
      105.08009765625  7.724266809844257E-5 
      105.18607421875  2.2629639948999483E-5 
      105.29205078125  6.16498344775001E-6 
      105.39802734375  1.5654143252799614E-6 
      105.50400390625  3.7131404270400297E-7 
      105.60998046875  8.245248274065883E-8 
      105.71595703125  1.7175827515950713E-8 
      105.82193359375  3.363191877983722E-9 
      105.92791015625  6.202157416103813E-10 
      106.03388671875  1.0791924537575563E-10 
      106.13986328125  1.7755432097671113E-11 
      106.24583984375  2.906533238883514E-12 
      106.35181640625  3.3503977705494196E-12 
      106.45779296875  4.9290413507515054E-11 
      106.56376953125  6.720234494246975E-10 
      106.66974609374999  7.541754255717478E-9 
      106.77572265625  7.011615493772992E-8 
      106.88169921875  5.439263301431727E-7 
      106.98767578125  3.544966568680498E-6 
      107.09365234375  1.9537434075779914E-5 
      107.19962890625  9.162354529173665E-5 
      107.30560546875  3.6779596810798215E-4 
      107.41158203124999  0.0012709416991215998 
      107.51755859375  0.003801111182201905 
      107.62353515625  0.009890186044855033 
      107.72951171875  0.022498362513647455 
      107.83548828125  0.04495726454318064 
      107.94146484375  0.07927070257303247 
      108.04744140625  0.12387019997476718 
      108.15341796875  0.17225039690381755 
      108.25939453125  0.21400085099243654 
      108.36537109375  0.2384425730802759 
      108.47134765625  0.2391379296782175 
      108.57732421875  0.21663490612268452 
      108.68330078125  0.17786130248693702 
      108.78927734375  0.1327722041695152 
      108.89525390625  0.09039612518321363 
      109.00123046875  0.05629905113137014 
      109.10720703125  0.032166402840345706 
      109.21318359375  0.01690627396346738 
      109.31916015624999  0.008195692110965843 
      109.42513671875  0.003673843688144566 
      109.53111328125  0.0015265697115523327 
      109.63708984375  5.893831126100744E-4 
      109.74306640625  2.1190943165158728E-4 
      109.84904296875  7.110905745889147E-5 
      109.95501953125  2.2317096586083498E-5 
      110.06099609374999  6.564049232095713E-6 
      110.16697265625  1.8129224919484772E-6 
      110.27294921875  4.7107377048121097E-7 
      110.37892578125  1.1548350786271787E-7 
      110.48490234375  2.8098751838731385E-8 
      110.59087890625  1.9978822771623542E-8 
      110.69685546875  1.190958332875227E-7 
      110.80283203125  8.294546173605105E-7 
      110.90880859375  4.953564681749852E-6 
      111.01478515625  2.528046044166561E-5 
      111.12076171875  1.1085231104880384E-4 
      111.22673828125  4.1990104254568586E-4 
      111.33271484375  0.0013811483893271122 
      111.43869140625  0.003964378298098029 
      111.54466796874999  0.009977184906019272 
      111.65064453125  0.02211605184999676 
      111.75662109375  0.043366923403187034 
      111.86259765625  0.07553818654112206 
      111.96857421874999  0.11734422775549178 
      112.07455078125  0.16319359858211654 
      112.18052734375  0.2039304047926655 
      112.28650390625  0.22978690043710723 
      112.39248046875  0.23426036072410322 
      112.49845703125  0.2167755627633316 
      112.60443359375  0.18264687555842463 
      112.71041015624999  0.14054144075481137 
      112.81638671875  0.0990458987063759 
      112.92236328125  0.06410774627121073 
      113.02833984375  0.0382105167908247 
      113.13431640625  0.02102649014890631 
      113.24029296875  0.010708627434036552 
      113.34626953125  0.005059599661045631 
      113.45224609375  0.002222840271721577 
      113.55822265625  9.100568091749122E-4 
      113.66419921875  3.4795320075372394E-4 
      113.77017578125  1.244962459262739E-4 
      113.87615234375  4.1767094167606325E-5 
      113.98212890625  1.3163888796570305E-5 
      114.08810546875  3.90487001585845E-6 
      114.19408203124999  1.0921297924860287E-6 
      114.30005859375  2.8849189314809375E-7 
      114.40603515625  7.209508889212805E-8 
      114.51201171875  1.7072108469308448E-8 
      114.61798828124999  3.836643385136942E-9 
      114.72396484375  8.195010317098054E-10 
      114.82994140625  1.6661377724414963E-10 
      114.93591796875  3.228825494563645E-11 
      115.04189453125  5.972272634127292E-12 
      115.14787109375  1.0557623741614359E-12 
      115.25384765625  1.785977919048072E-13 
      115.35982421874999  2.8947004157372414E-14 
      115.46580078125  4.500556243037067E-15 
      115.57177734375  6.719933147342751E-16 
      115.67775390624999  9.646847959614375E-17 
      115.78373046875  1.3329017363702148E-17 
      115.88970703125  1.7744308277631636E-18 
      115.99568359375  2.2782970027431733E-19 
      116.10166015624999  2.8241033622847364E-20 
      117.05544921875  3.0781623522254684E-19 
      117.16142578124999  9.098716560311801E-18 
      117.26740234374999  2.2296368938315653E-16 
      117.37337890625  4.560131976341184E-15 
      117.47935546875  7.834306078351615E-14 
      117.58533203125  1.137544409982171E-12 
      117.69130859375  1.404193733279826E-11 
      117.79728515625  1.4818714803208454E-10 
      117.90326171875  1.344147299150735E-9 
      118.00923828124999  1.0533281321886406E-8 
      118.11521484375  7.166279481438039E-8 
      118.22119140625  4.252847938233138E-7 
      118.32716796874999  2.2114626587478477E-6 
      118.43314453125  1.0119749667511693E-5 
      118.53912109375  4.09213916687112E-5 
      118.64509765625  1.468076649710509E-4 
      118.75107421875  4.6905400213493426E-4 
      118.85705078125  0.001339568897671972 
      118.96302734375  0.003431667063644781 
      119.06900390625  0.007912494274407841 
      119.17498046875  0.016474168074008182 
      119.28095703125  0.031069559848294032 
      119.38693359375  0.053237362304553525 
      119.49291015624999  0.08312029291717973 
      119.59888671875  0.11858212932911634 
      119.70486328125  0.15499618083404287 
      119.81083984374999  0.18609626785059508 
      119.91681640624999  0.205756663003065 
      120.02279296875  0.20999817826183048 
      120.12876953125  0.19830385841697576 
      120.23474609375  0.17364965728880846 
      120.34072265625  0.14131326912617437 
      120.44669921875  0.10709391271906468 
      120.55267578125  0.07573469660477421 
      120.65865234374999  0.050074756895170644 
      120.76462890625  0.03101370621768872 
      120.87060546875  0.01802556748006732 
      120.97658203124999  0.00984890621184213 
      121.08255859375  0.0050674438890975235 
      121.18853515625  0.002459267159396346 
      121.29451171875  0.0011275323028458475 
      121.40048828124999  4.891332734869992E-4 
      121.50646484375  2.0107003229928624E-4 
      121.61244140625  7.843607150037437E-5 
      121.71841796875  2.9076311016527732E-5 
      121.82439453125  1.0256646520004106E-5 
      121.93037109375  3.447334444093194E-6 
      122.03634765625  1.105414452268552E-6 
      122.14232421874999  3.385831628287063E-7 
      122.24830078125  9.917965375909987E-8 
      122.35427734375  2.78163182198548E-8 
      122.46025390624999  7.477978860280713E-9 
      122.56623046874999  1.9290766585621443E-9 
      122.67220703125  4.780292898215075E-10 
      122.77818359375  1.1390554598107866E-10 
      122.88416015625  2.612480765528724E-11 
      122.99013671875  5.772967973050545E-12 
      123.09611328125  1.2302427597301173E-12 
      123.20208984375  2.530604092229806E-13 
      123.30806640624999  5.029028572883312E-14 
      123.41404296875  9.663721043647304E-15 
      123.52001953125  1.797084041247012E-15 
      123.62599609374999  3.23675081612833E-16 
      123.73197265625  5.650805257861165E-17 
      123.83794921875  9.569856497607486E-18 
      123.94392578124999  1.5733275530642612E-18 
      124.04990234374999  2.512853990386658E-19 
      124.15587890625  3.901727193048694E-20 
      128.07701171875  1.725353580576519E-56 
 }; 
      \addlegendentry{}
      \label{}
      
        \addplot table[x=x,y=y] { 
      x  y 
      101.20607455922807  0.0459616833815542 
      101.8080156585258  0.04743813376722539 
      102.27107734476282  0.04839304797939822 
      102.70356584256345  0.049135258848334966 
      103.19737523484157  0.04979975445211217 
      103.87764132070127  0.05038775068720478 
      104.46292460121201  0.05058689095119182 
      104.84915361954383  0.0505637034558061 
      105.54968387235802  0.05021453640707879 
      105.95131237626133  0.04984128833934768 
      106.67843690128517  0.04886187392144015 
      107.02478176327975  0.04826557791431572 
      107.63769655273917  0.04702137737726421 
      108.30566934201036  0.045416146147137515 
      108.70846950143121  0.044336900804515966 
      109.21093508335979  0.042888210528317774 
      109.94456947833311  0.04060077992154888 
      110.26303426946761  0.039555454599464863 
      110.80786186110156  0.037709723966024004 
      111.36013692030873  0.03578279404619359 
      112.10819637057736  0.033118979852920886 
      112.5491866648476  0.0315369320577495 
      113.16475819011661  0.029334598306763698 
      113.61506663159136  0.027739490907371076 
      114.13378922943923  0.025930237224972594 
      114.65811718545409  0.024142528275577214 
      115.33836207482031  0.021900356239763268 
      115.82286024135003  0.02036534843587059 
      116.38719408729236  0.01865007771268235 
      116.76404081041699  0.017551217928570093 
      117.3391989416979  0.015949710319387897 
      117.97868100208282  0.014280179858777985 
      118.60628158866469  0.01275795911587771 
      119.01080592535217  0.011838221012779263 
      119.71324621062216  0.010354687925984556 
      120.22743193314574  0.00935853497680795 
      120.59226581667417  0.008696621635874216 
      121.32096696310151  0.007482363778496054 
      121.8175726348924  0.006733885859636989 
      122.2230781579411  0.006167972993910236 
      122.95160469466626  0.005247965184603101 
      123.25258472362239  0.004902249092094642 
      123.79231965933648  0.004329508197655877 
      124.56378394844741  0.003608912383493579 
      125.1379567172101  0.003141083243164034 
      125.66038939669657  0.002761582858579987 
      126.08275973671249  0.0024844003532584107 
      126.72179436659314  0.0021110890454771217 
      127.05929866911912  0.001934535374230523 
      127.74387081849605  0.0016159463169558852 
 }; 
  \legend{}
  %
 
  \end{axis}
  \end{tikzpicture}
  \begin{tikzpicture} 
  \begin{axis}[ 
  	cycle list name=canti,
  xlabel=$x$,
  ylabel=,
  grid,
  xmin=101,
  xmax=128.13,
  width = 0.5\textwidth,
  height = 5cm,
  scaled y ticks=false,
  ] 
        \addplot table[x=x,y=density] { 
      x  density 
      101.42390625  0.05613419421020152 
      102.27171875  0.05708013782437654 
      103.11953125  0.05716641811894432 
      103.96734375  0.05641560191163158 
      104.81515625  0.05488532633780539 
      105.66296875  0.052662509446822624 
      106.51078125  0.04985605473987906 
      107.35859375  0.04658881725025815 
      108.20640625  0.04298956826425256 
      109.05421875  0.03918559168350522 
      109.90203125  0.03529639177977291 
      110.74984375  0.03142881596406118 
      111.59765625  0.027673721271259912 
      112.44546875  0.024104158626746003 
      113.29328124999999  0.020774927222660323 
      114.14109375  0.01772326843931516 
      114.98890625  0.014970424784819159 
      115.83671875  0.012523779915940112 
      116.68453124999999  0.0103793137583678 
      117.53234375  0.008524143660365373 
      118.38015625  0.006938970090675105 
      119.22796875  0.005600296378931479 
      120.07578125  0.0044823407388703 
      120.92359375  0.0035586014408528757 
      121.77140625  0.0028030703428894915 
      122.61921874999999  0.0021911153214425492 
      123.46703124999999  0.0017000688268733835 
      124.31484375  0.001309568905729143 
      125.16265625  0.0010017020508342389 
      126.01046875  7.6099574254739E-4 
      126.85828125  5.743040314906986E-4 
      127.70609375  4.306232816533708E-4 
 }; 
      \addlegendentry{}
      \label{}
      \addplot table[x=x,y=density] { 
      x  density 
      101.42390625  0.03511926123505058 
      102.27171875  0.03917490653655301 
      103.11953125  0.0429822667819591 
      103.96734375  0.046409367042239726 
      104.81515625  0.04933617986319959 
      105.66296875  0.05166187313032996 
      106.51078125  0.05331061731477564 
      107.35859375  0.05423549939726783 
      108.20640625  0.05442029860907759 
      109.05421875  0.05387909914317116 
      109.90203125  0.05265391934455123 
      110.74984375  0.05081070288229046 
      111.59765625  0.04843413006694352 
      112.44546875  0.04562176060429789 
      113.29328124999999  0.04247801478396223 
      114.14109375  0.03910844720890811 
      114.98890625  0.03561467890340422 
      115.83671875  0.032090245115168295 
      116.68453124999999  0.02861750209128829 
      117.53234375  0.025265629139459983 
      118.38015625  0.022089671706223415 
      119.22796875  0.019130502634071185 
      120.07578125  0.016415534267104192 
      120.92359375  0.013959992736702988 
      121.77140625  0.011768564462958948 
      122.61921874999999  0.00983723922951143 
      123.46703124999999  0.00815519920317598 
      124.31484375  0.00670663421028284 
      125.16265625  0.005472396311979634 
      126.01046875  0.0044314380006627175 
      126.85828125  0.0035620059184675082 
      127.70609375  0.0028425845716931186 
 }; 
      \addlegendentry{}
      \label{}
      \addplot table[x=x,y=density] { 
      x  density 
      101.42390625  0.028606888309660627 
      102.27171875  0.03278853856150891 
      103.11953125  0.036933347944038755 
      103.96734375  0.040906333222904545 
      104.81515625  0.0445717300837547 
      105.66296875  0.047801265459659015 
      106.51078125  0.050482003262411194 
      107.35859375  0.052523072286987084 
      108.20640625  0.053860718374521235 
      109.05421875  0.05446131305740746 
      109.90203125  0.05432216989360285 
      110.74984375  0.05347023812454014 
      111.59765625  0.051958934869215385 
      112.44546875  0.04986352206401663 
      113.29328124999999  0.04727552154586563 
      114.14109375  0.044296688678703025 
      114.98890625  0.04103303717689868 
      115.83671875  0.03758933652247359 
      116.68453124999999  0.03406440316039596 
      117.53234375  0.030547392872064264 
      118.38015625  0.027115188477017476 
      119.22796875  0.023830875623843544 
      120.07578125  0.02074321764483784 
      120.92359375  0.017886982264852535 
      121.77140625  0.015283938992877399 
      122.61921874999999  0.012944334230180172 
      123.46703124999999  0.01086865763807571 
      124.31484375  0.009049533321550571 
      125.16265625  0.007473597926204706 
      126.01046875  0.00612326024586702 
      126.85828125  0.004978269612694284 
      127.70609375  0.004017050377048508 
 }; 
      \addlegendentry{}
      \label{}
      \addplot table[x=x,y=density] { 
      x  density 
      101.42390625  0.046622959053159295 
      102.27171875  0.050040748524967225 
      103.11953125  0.052829476493761476 
      103.96734375  0.0548880937822167 
      104.81515625  0.05614901643747536 
      105.66296875  0.05658138103197297 
      106.51078125  0.056191452356607524 
      107.35859375  0.05502033848530546 
      108.20640625  0.05313942221349578 
      109.05421875  0.05064409906150247 
      109.90203125  0.04764650718919425 
      110.74984375  0.04426794400915776 
      111.59765625  0.04063159961000288 
      112.44546875  0.03685611758071748 
      113.29328124999999  0.033050342304363566 
      114.14109375  0.029309450854900037 
      114.98890625  0.025712516629255094 
      115.83671875  0.022321425003157244 
      116.68453124999999  0.01918096699389238 
      117.53234375  0.016319877939357998 
      118.38015625  0.013752562726859537 
      119.22796875  0.01148125200919045 
      120.07578125  0.009498358160684223 
      120.92359375  0.007788837928791375 
      121.77140625  0.006332413768637828 
      122.61921874999999  0.00510555179538215 
      123.46703124999999  0.004083136732032071 
      124.31484375  0.00323982031047312 
      125.16265625  0.0025510478669858385 
      126.01046875  0.0019937880883709527 
      126.85828125  0.0015470035840834022 
      127.70609375  0.0011919062873637972 
 }; 
      \addlegendentry{}
      \label{}
      \addplot table[x=x,y=density] { 
      x  density 
      101.42390625  0.023531372824721276 
      102.27171875  0.027480473069116544 
      103.11953125  0.03153495717691926 
      103.96734375  0.03557796708910175 
      104.81515625  0.03948320378764772 
      105.66296875  0.043122078552752185 
      106.51078125  0.046371290253906884 
      107.35859375  0.04912016412726575 
      108.20640625  0.051277118703984416 
      109.05421875  0.052774729802007904 
      109.90203125  0.053573017113185704 
      110.74984375  0.053660767002672166 
      111.59765625  0.053054899529197945 
      112.44546875  0.05179806488085442 
      113.29328124999999  0.04995479569593939 
      114.14109375  0.04760663520058524 
      114.98890625  0.04484670233286357 
      115.83671875  0.04177414659325627 
      116.68453124999999  0.03848889534744964 
      117.53234375  0.03508701634924327 
      118.38015625  0.03165692144249255 
      119.22796875  0.028276536391788294 
      120.07578125  0.025011467331941244 
      120.92359375  0.021914114361143353 
      121.77140625  0.01902362212424335 
      122.61921874999999  0.016366517609648513 
      123.46703124999999  0.013957866036738582 
      124.31484375  0.011802773987398013 
      125.16265625  0.009898081043866495 
      126.01046875  0.008234102891619036 
      126.85828125  0.0067963159969399055 
      127.70609375  0.005566902954281317 
 }; 
      \addlegendentry{}
      \label{}
      \addplot table[x=x,y=density] { 
      x  density 
      101.42390625  0.03800293512655866 
      102.27171875  0.04131296090330445 
      103.11953125  0.04428929330312458 
      103.96734375  0.04683947260961886 
      104.81515625  0.04888509130197655 
      105.66296875  0.05036582152430735 
      106.51078125  0.05124228358551606 
      107.35859375  0.051497578309416855 
      108.20640625  0.05113742523306632 
      109.05421875  0.05018896654951884 
      109.90203125  0.04869840106406139 
      110.74984375  0.04672769359654434 
      111.59765625  0.04435065706932393 
      112.44546875  0.041648724751326485 
      113.29328124999999  0.03870672031055823 
      114.14109375  0.03560889807648231 
      114.98890625  0.032435471965448145 
      115.83671875  0.0292597866299991 
      116.68453124999999  0.026146216270278817 
      117.53234375  0.02314881199209818 
      118.38015625  0.020310662888653616 
      119.22796875  0.017663892607564992 
      120.07578125  0.01523018362868756 
      120.92359375  0.013021705746468626 
      121.77140625  0.011042321938321404 
      122.61921874999999  0.009288951637232961 
      123.46703124999999  0.007752985687379145 
      124.31484375  0.006421666147086738 
      125.16265625  0.005279365039974182 
      126.01046875  0.004308716993813423 
      126.85828125  0.00349157982873516 
      127.70609375  0.002809813494408022 
 }; 
   \addplot table[x=x,y=y] { 
      x  y 
      101.20607455922807  0.0459616833815542 
      101.8080156585258  0.04743813376722539 
      102.27107734476282  0.04839304797939822 
      102.70356584256345  0.049135258848334966 
      103.19737523484157  0.04979975445211217 
      103.87764132070127  0.05038775068720478 
      104.46292460121201  0.05058689095119182 
      104.84915361954383  0.0505637034558061 
      105.54968387235802  0.05021453640707879 
      105.95131237626133  0.04984128833934768 
      106.67843690128517  0.04886187392144015 
      107.02478176327975  0.04826557791431572 
      107.63769655273917  0.04702137737726421 
      108.30566934201036  0.045416146147137515 
      108.70846950143121  0.044336900804515966 
      109.21093508335979  0.042888210528317774 
      109.94456947833311  0.04060077992154888 
      110.26303426946761  0.039555454599464863 
      110.80786186110156  0.037709723966024004 
      111.36013692030873  0.03578279404619359 
      112.10819637057736  0.033118979852920886 
      112.5491866648476  0.0315369320577495 
      113.16475819011661  0.029334598306763698 
      113.61506663159136  0.027739490907371076 
      114.13378922943923  0.025930237224972594 
      114.65811718545409  0.024142528275577214 
      115.33836207482031  0.021900356239763268 
      115.82286024135003  0.02036534843587059 
      116.38719408729236  0.01865007771268235 
      116.76404081041699  0.017551217928570093 
      117.3391989416979  0.015949710319387897 
      117.97868100208282  0.014280179858777985 
      118.60628158866469  0.01275795911587771 
      119.01080592535217  0.011838221012779263 
      119.71324621062216  0.010354687925984556 
      120.22743193314574  0.00935853497680795 
      120.59226581667417  0.008696621635874216 
      121.32096696310151  0.007482363778496054 
      121.8175726348924  0.006733885859636989 
      122.2230781579411  0.006167972993910236 
      122.95160469466626  0.005247965184603101 
      123.25258472362239  0.004902249092094642 
      123.79231965933648  0.004329508197655877 
      124.56378394844741  0.003608912383493579 
      125.1379567172101  0.003141083243164034 
      125.66038939669657  0.002761582858579987 
      126.08275973671249  0.0024844003532584107 
      126.72179436659314  0.0021110890454771217 
      127.05929866911912  0.001934535374230523 
      127.74387081849605  0.0016159463169558852 
 }; 
 \legend{}

  \end{axis}
  \end{tikzpicture}
\end{center}
\caption{Five realizations of the density estimator (blue), their average (orange, dashed in b/w),
 and  the true density  (thick black)  for the sequential CDE (left) and the bridge CDE (right),
 for the Asian option example.}
\label{fig:asian-cde}
\end{figure}
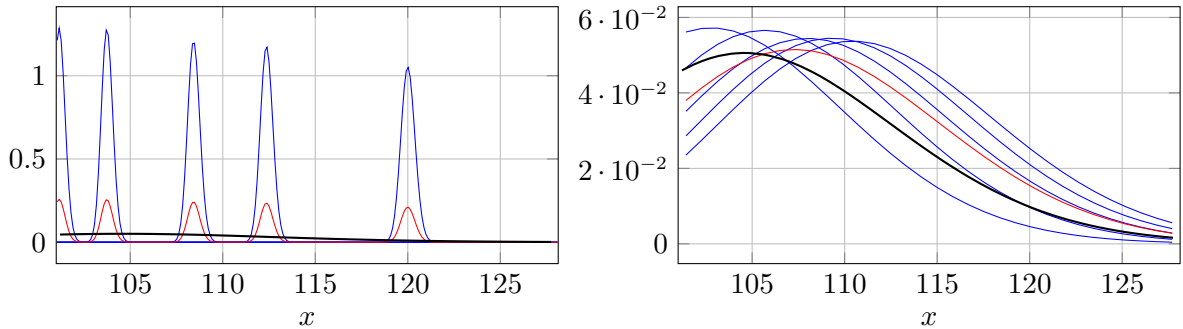

\new{The method introduced here works not only for a Brownian motion, but for more
general L\'evy processes as well.  Some popular models in finance use a L\'evy subordinator process to 
produce a random clock speed to model stochastic volatility, and a geometric Brownian process 
that evolves at that random speed.  This includes the \emph{variance-gamma} (VG) and 
the \emph{normal inverse Gaussian} processes, for example. 
For these models, we can generate the subordinator as usual using a bridge method
to obtain the random times at which the Brownian process is evaluated,
and then apply the method that we just described to the resulting Brownian process.
For the VG process, a more effective alternative could be to use the \emph{double gamma bridge sampling}
method described in \cite{fAVR06a} and hide the first variable as we have done here.
Yet another approach would be to do a monotone mapping between the L\'evy process and a Brownian motion
as explained in \cite{fLEC08a}, estimate the density in the Brownian representation as we did here,
and transform this density to the original L\'evy process via the change of variable that corresponds 
to the mapping.}


\iflong \color{violet}  

\subsection{Estimating a quantile with a confidence interval}
\label{sec:quantile}


For $0 < q < 1$, the $q$-quantile of the distribution of $X$ is defined as 
$\xi_q = F^{-1}(q) = \inf\{x : F(x)\ge q\}$.
Given $n$ i.i.d.\ observations of $X$, a standard (consistent) estimator of $\xi_q$ is 
the $q$-quantile of the empirical distribution, defined 
as $\hat\xi_{q,n} = X_{(\lceil nq\rceil)}$, where $X_{(1)},\dots,X_{(n)}$ are the 
$n$ observations sorted in increasing order (the order statistics).
We assume that the density $f(x)$ is positive and continuously differentiable 
in a neighborhood of $\xi_q$.  Then we have the central limit theorem (CLT): 
\[
  \sqrt{n} (\hat\xi_{q,n} - \xi_q) / \sigma_\xi  \Rightarrow \cN(0,1) \qquad \mbox{ for } n\to\infty, 
\] 
where $\sigma^2_\xi = q(1-q)/f^2(\xi_q)$ \citep{tSER80a}.
This provides a way to compute a confidence interval on $\xi_q$, but requires the estimation
of $f(\xi_q)$, which is generally difficult.  Some approaches for doing this include 
finite differences with the empirical cdf, batching, and sectioning 
\citep{sASM07a,tNAK14a,tNAK14b}.

In our setting, one can do better by taking the $q$-quantile $\hat\xi_{\cmc,q,n}$
of the conditional cdf
\[
  \hat F_{\cmc,n}(x)  = \frac{1}{n} \sum_{i=1}^n F(x\mid \cG^{(i)}).
\]
That is, $\hat\xi_{\cmc,q,n} = \inf\{x : \hat F_{\cmc,n}(x) \ge q\}$.
This idea was already suggested by \cite{tNAK14b}, who pointed out that this estimator obeys
a CLT just like $\hat\xi_{q,n}$, but with the variance constant $\sigma^2_\xi$
replaced by $\sigma^2_{\cmc,\xi} = \Var[F(\xi_q \mid \cG)]/f^2(\xi_q) \le \sigma^2_\xi$.
This is an improvement on the quantile estimator itself.
Our CDE approach also provides an improved estimator of the density $f(\xi_q)$ which 
appears in the variance expression.
We estimate $f(\xi_q)$ by $\hat f_{\cde,n} (\hat\xi_{\cmc,q,n})$.  
This provides a more accurate confidence interval of $\xi_q$.

Further improvements on the variances of both the quantile and density 
estimators can be obtained by using RQMC to generate the realizations $\cG^{(i)}$.
In particular, if $\tilde g(\xi_q, \bu) = F(\xi_q \mid \cG)$ is a sufficiently smooth function 
of $\bu$, $\Var[\hat\xi_{\cmc,q,n}]$ can converge at a faster rate than $\cO(n^{-1})$.
When using RQMC with $n_r$ randomizations to estimate a quantile, the quantile 
estimator will be the empirical quantile of all the $n_r\times n$ observations.
\hpierre{How do we estimate the variance of this quantile estimator?  Perhaps not obvious.}
\hpierre{We could have a short numerical example to illustrate this, based on an
 example which is already in the paper, e.g., the cantilever example. 
 We can estimate the 0.95 quantile and the density at that quantile.}

A related quantity is the \emph{expected shortfall}, defined as 
$c_q = \EE[X\mid X > \xi_q] = \xi_q - \EE[(\xi_q - X)^+] / q$ 
which is often estimated by its empirical version \citep{sHON14a}
\[
   \hat c_{q,n} = \hat\xi_{q,n} - \frac{1}{nq} \sum_{i=1}^n (\hat\xi_{q,n} - X_i)^+.
\]
This estimator obeys the CLT
$
  \sqrt{n} (\hat c_{q,n} - c_q) / \sigma_c  \Rightarrow \cN(0,1)
$ 
for $n\to\infty$, where $\sigma^2_c = \Var[(\xi_q - X)^+]/q^2$, if this variance is finite \citep{sHON14a}.
By improving the quantile estimator, CDE+RQMC can also improve the expected shortfall estimator
as well as the estimator of the variance constant $\sigma^2_c$ and the quality of 
confidence intervals on $c_q$. We leave this as a topic for future work.


\hpierre{We should have a very short discussion here of related papers on VaR and 
quantile sensitivity estimation:   \cite{sHE19a,sHON09a,sHON14a,vNAK14a}.   
Also some papers on estimating the sensitivity of conditional value at risk (CVaR) 
with respect to a parameter $\theta$.  See for example \cite{sHON09b,sHON14a,sHON14b}.
I think it is important to cite some of the papers of Jeff Hong and make appropriate links,
with our future work.}

\fi  \color{black}  

\iflong\else  
\subsection{More examples}

Additional examples are given in the Online Supplement. 
In the first one, $X$ is a sum of independent normal random variables, with known density,
and the purpose is to see how each estimator behaves as a function of the dimension
(the number of summands) and of the relative variance of the one we hide.
The second one is a (real-life) six-dimensional example in which $X$ is the buckling strength 
of a steel plate.
The third one is a multicomponent system in which each component fails at a certain random time, 
and we want to estimate the density of the failure time of the system.
In the fourth one, we explain how accurate density estimation is useful to compute a confidence 
interval on a quantile or on the expected shortfall.  This, alone, has many applications.
\fi

\section{Summary and Guidelines}
\label{sec:guidelines}

Here we provide a summary of the main conditions and some guidelines for applying the method.
The primary task is to select the information $\cG$ on which we condition,
or equivalently, to decide what we hide.
The main constraint is that $\cG$ must be selected in a way that Assumption 1 is satisfied,
at least in the region where we want to estimate the density.
The key condition for this is that the conditional density must be well-defined in that region.
In particular, the conditional CDF should never have jumps. 
A second requirement is that the conditional density can be calculated efficiently for all
realizations of $\cG$ (almost surely).
There are often many possible choices of $\cG$ for which these conditions are satisfied.
The set of admissible choices is highly problem-dependent; it depends on the model and also on 
the variable $X$ of interest.

In many cases, one must hide more than just a single input random variable. 
The hidden information can also be dynamically selected, in the sense that it may depend on the 
sample realization (for example, if $X$ is the maximum of several independent random variables, 
one may generate all the variables and hide the maximum).  
Sometimes, it becomes much easier to apply the method after making an appropriate multivariate 
change of variable.  This may require creativity, as we have shown in our examples.

When there are many choices for $\cG$, finding the optimal one 
(e.g., to minimize the work-normalized MISE) can be difficult in general, 
because the MISE depends on many factors, but just finding a good one is usually much easier and sufficient. 
For comparable computing times,
the best choices are often those for which the variance of the conditional density is the largest.
As a rule of thumb, it is usually better to condition on lower-variance information and hide variables
having a large variance contribution.  
The optimal choice also depends (in general) on the point $x$ at which we want to estimate the density.
In principle, one could optimize by using different conditioning for different intervals, 
but it is usually not worth the additional complications.
When several good choices of $\cG$ are available, selecting a few of them and taking a convex combination 
of the corresponding estimators can provide a more robust CDE than selecting only one $\cG$.

When we want to combine the CDE with RQMC, additional properties come into play:
we want to select $\cG$ and also the formulation of the estimator as a function $\tilde g'$ of the vector $\bU$
of underlying uniform random numbers in a way that the assumption of Proposition 3 is satisfied
(if possible) and the variation of this function $\tilde g'$ is not too large.
Proving these conditions in practice may be difficult, but one can always try RQMC empirically.
Experience shows that it can often reduce the variance significantly, 
in particular when the effective dimension of $\tilde g'$ is small. 
This was illustrated in our examples.

\section{Conclusion}
\label{sec:conclusion}

We have examined a simple and very effective approach for estimating the density of a 
random variable generated by simulation from a stochastic model, 
by using a computable conditional density.
The resulting CDE is unbiased and its MISE
converges faster than for other popular density estimators such as the KDE.
We have also shown how to further reduce the IV, and even improve its convergence rate,
by combining the CDE with RQMC.
Our numerical examples show that this combination can be very efficient.
It sometimes reduces the MISE by factors over a million.
The CDE approach also outperforms the recently proposed GLRDE method,
and CDE+RQMC outperforms both GLRDE+RQMC and KDE+RQMC, in all our examples.
RQMC tends to bring a larger improvement to the CDE than the KDE or GLRDE
because the estimator usually has less variation as a function of the underlying uniforms.

\new{The examples in the paper were selected to provide insight on key issues.
We tried to avoid unnecessary complications in the models. 
But it would not be too difficult to derive CDEs for larger and more complicated versions 
of these models. We outlined some possibilities in the text.}
Suggested future work includes \new{experimenting this methodology in more complicated applications,}
designing and exploring different types of conditioning,
and perhaps adapting the Monte Carlo sampling strategies to make the method more effective 
\new{for specific applications} 
(e.g., by changing the way $X$ is defined in terms of the basic input random variates).
Its application to quantile and expected shortfall estimation also deserves further study.

\ACKNOWLEDGMENT{%
\hpierre{This definitely needs to be shortened.  Max 4-5 lines in total.}
This work has been supported by an IVADO Research Grant, an NSERC-Canada Discovery Grant, 
a Canada Research Chair, and an Inria International Chair, to P. L'Ecuyer.
F. Puchhammer was also supported by Spanish and Basque governments fundings through BCAM (ERDF, ESF, SEV-2017-0718, PID2019-108111RB-I00, PID2019-104927GB-C22, BERC 2018e2021, EXP. 2019/00432, ELKARTEK KK-2020/00049), and the computing infrastructure of i2BASQUE academic network and IZO-SGI SGIker (UPV).
%
%
The reviewers, the editors, and Julien Keutchayan gave several comments that improved the paper.
}


\end{document}



\RUNAUTHOR{L'Ecuyer, Puchhammer, Ben Abdellah}

\RUNTITLE{MC and QMC Density Estimation via Conditioning}

\TITLE{Monte Carlo and Quasi-Monte Carlo Density Estimation via Conditioning}

\ARTICLEAUTHORS{%
 \AUTHOR{Pierre L'Ecuyer}
 \AFF{D\'{e}partement d'Informatique et de Recherche Op\'{e}rationnelle, 
 Pavillon Aisenstadt, Universit\'{e} de Montr\'{e}al, C.P. 6128,
 Succ. Centre-Ville, Montr\'{e}al, Qu\'{e}bec, Canada H3C 3J7, \EMAIL{lecuyer@iro.umontreal.ca}}
 
  \AUTHOR{Florian Puchhammer}
 \AFF{ Basque Center for Applied Mathematics, Alameda de Mazarredo 14, 48009 Bilbao, Basque Country, Spain; and 
  D\'{e}partement d'Informatique et de Recherche Op\'{e}rationnelle, 
  Universit\'{e} de Montr\'{e}al,
 \EMAIL{fpuchhammer@bcamath.org}}
 
  \AUTHOR{Amal Ben Abdellah}
 \AFF{D\'{e}partement d'Informatique et de Recherche Op\'{e}rationnelle, 
 Pavillon Aisenstadt, Universit\'{e} de Montr\'{e}al, C.P. 6128,
 Succ. Centre-Ville, Montr\'{e}al, Qu\'{e}bec, Canada H3C 3J7, \EMAIL{amal.ben.abdellah@umontreal.ca}}
}

\maketitle

%
%
%
%
%

\begin{APPENDICES}

\centerline{\LARGE\bf Online Supplement}
\bigskip

This Supplement contains additional examples and details for which there was not
enough space in the main paper.

In Appendix~\ref{sec:sum-rqmc}, we show with simple examples how one can prove that the 
HK variation of the CDE is bounded uniformly over the interval $[a,b]$ of interest.  
When this can be done, it proves that the MISE for the CDE with good RQMC points 
converges as $\cO(n^{-2+\epsilon})$. 

Appendix~\ref{sec:more-examples} provides additional examples showing how CDEs can be 
constructed, sometimes in non-trivial ways that are adapted to the problem at hand.
In Section~\ref{sec:normals}, we consider the very simple example of a 
sum of independent normal random variables, for which the density is known,
and the purpose is to see how each estimator behaves as a function of the dimension
(the number of summands) and of the relative variance of the one we hide.
In Section~\ref{sec:buckling}, we consider a six-dimensional example taken from \cite{tSCH16a}.
In Section~\ref{sec:network-failure}, we consider a multicomponent system in which each 
component fails at a certain random time, and we want to estimate the density of the 
failure time of the system.
In Section~\ref{sec:quantile}, we explain briefly how accurate density estimation is useful 
for computing a confidence interval on a quantile or on the expected shortfall.

Section~\ref{sec:more-figures} provides additional figures for examples in the paper.

\bigskip
\section{Proving bounded HK variation for the CDE: some simple illustrations}
\label{sec:sum-rqmc}

Here we show how the HK variation of $g \equiv \tilde g'(x,\cdot)$ can be bounded 
uniformly in $x\in[a,b]$ in our CDE setting, 
for Examples~\ref{ex:sum-asm18} to \ref{ex:sum2normal} of the paper.
\hpierre{To do:  Must also be bounded uniformly in $x$ over $[a,b]$.}

\begin{example} \rm
\label{ex:sum-rqmc}
Consider a sum of random variables as in Example~\ref{ex:sum-asm18}, 
with $\cG = \cG_{-k}$ summarized by the single real number $S_{-k}$. 
We have $F(x\mid \cG) = F_k(x-S_{-k})$ and $f(x\mid \cG) = f_k(x-S_{-k})$.
Without loss of generality, 
let $k = d$.
Suppose that each $Y_j$ is generated by inversion from $U_j \sim \cU(0,1)$, 
so $Y_j = F_j^{-1}(U_j)$ and $S_{-d} = F_1^{-1}(U_1) + \cdots + F_s^{-1}(U_s)$ with $s = d-1$.
This gives $\tilde g(x,\bU) = F_d(x-S_{-d}) = F_d(x - F_1^{-1}(U_1) - \cdots - F_s^{-1}(U_s))$ and 
$\tilde g'(x,\bU) = f_d(x-S_{-d}) = f_d(x - F_1^{-1}(U_1) - \cdots - F_s^{-1}(U_s))$.
The partial derivatives of this last function are
\hpierre{Well, here, with the anchor at 1, we have terms of the form $F_j^{-1}(1)$ in $S_{-d}$,
 for $j\not\in \frakv$.  For the normal distribution, these terms are infinite and 
 then the density $f_d$ is zero at these points!  Anything to add on this?}
\hflorian{What we could do is to show this formula without the anchor (this only affects the left-hand side) and  maybe add a sentence what happens for distributions that are supported on $\RR$.}
\[
  \tilde g'_{\frakv}(x, \bU_{\frakv},\bone) 
   = f^{(|\frakv|)}_d (x-S_{-d}) \prod_{j\in\frakv} \frac{\partial (F_j^{-1}(U_j))}{\partial U_j}.
\]
So the functions $F_j^{-1}$ must be differentiable over $(0,1)$ for $j=1,\dots,d-1$, 
the density $f_d$ must be $s$ times differentiable, and the integral of 
$|\tilde g'_{\frakv}(x, \bu_{\frakv},\bone)|$ with respect to $\bu_{\frakv}$ must be 
bounded uniformly in $x\in [a,b]$.
Under these conditions, the HK variation is bounded uniformly in $x$ over $[a,b]$.

For Example~\ref{ex:sum2unif}, with $\cG = \cG_{-2}$ and $Y_1 = U_1$,
we have $\tilde g'(x,\bu) = \tilde g'(x,U_1) = \II[U_1 \le x \le \epsilon + U_1]/\epsilon
 = \II[x-\epsilon \le U_1 \le x]/\epsilon$.   
This function is not continuous, but its HK variation (not given by (\ref{eq:HK}) in this case)
is $2/\epsilon < \infty$, because it is piecewise constant with only two jumps,
each one of size $1/\epsilon$.  Thus, the HK variation is unbounded when $\epsilon\to 0$,
but it is finite for any fixed $\epsilon$, independently of $x$.
The behavior with $\cG = \cG_{-1}$ is similar and the HK variation is 2 in that case,
which is much better.

For Example~\ref{ex:sum2normal}, if $\cG = \cG_{-2}$,  
we have $Y_1 = \sigma_1 \Phi^{-1}(U_1)$ where $U_1\sim \cU(0,1)$.
Then, $F(x\mid \cG_{-2}) = F_2(x-Y_1) = \Phi((x - Y_1)/\sigma_2)$ and 
$f(x\mid \cG_{-2}) = \phi((x - \sigma_1 \Phi^{-1}(U_1))/\sigma_2)/ \sigma_2 = \tilde{g}'(x,U_1)$.
Taking the derivative with respect to $u$ and noting that $\d \Phi^{-1}(u)/\d u = 1/(\phi(\Phi^{-1}(u)))$ yields
\[ 
 \tilde{g}'_{\frakv}(x,u) = \frac{\phi'((x - \sigma_1 \Phi^{-1}(u))/ \sigma_2) \,\sigma_1}
                                 {\sigma_2^2 \phi(\Phi^{-1}(u))}
\]
for $\frakv = \{1\} = \cS$ (the only subset in this case).
Integrating this with respect to $u$ by making the change of variable $z = \Phi^{-1}(u)$ gives
\[ 
  \int_0^1 \tilde{g}'_{\frakv}(x,u)\d u
	 = \frac{\sigma_1} {\sigma_2^2} \int_{-\infty}^\infty |\phi'((x - \sigma_1 z) /\sigma_2)| \, \d z,
\]
which is bounded uniformly in $x$, 
because $|\phi'(\cdot)|$ is bounded by $\phi(\cdot)$ multiplied by the absolute value 
of a polynomial of degree 1.  So the HK variation is bounded uniformly in $x$.
\hflorian{Actually, it's pretty simple in this case: $|\phi'(y)| = |z| \phi(y)$.}
\hpierre{In the Assumption, the bound must be uniform in $x$. Is it?}
\hflorian{Yes: 1) linear change of variables gets rid of $x$. 
  2) What remains is linear in $\EE[|Y|]$ for a standard normal $Y$.}
\end{example}

\bigskip
\section{Additional examples}
\label{sec:more-examples}

\subsection{A sum of normals}
\label{sec:normals}

We start with a very simple example in which the density $f$ is known beforehand,
so there is no real need to estimate it, but this type of example is very convenient 
for testing the performance of various density estimators. 
Let $Z_1,\dots,Z_d$ be independent standard normal random variables, i.e., with mean 0 and variance 1,
and define 
\[
  X = (a_1 Z_1 + \cdots + a_d Z_d)/\sigma, \quad\mbox{ where }\quad 
	     \sigma^2 = a_1^2 + \cdots + a_d^2.
\]
Then $X$ is also standard normal, with density $f(x) = \phi(x) \eqdef \exp(-x^2/2)/\sqrt{2\pi}$ 
and cdf $\PP[X\le x] = \Phi(x)$ for $x\in\RR$.  The term $a_j Z_j$ in the sum has variance $a_j^2$.
We pretend we do not know this and we estimate $f(x)$ over the interval $[-2,2]$,
which contains slightly more than 95\% of the density.
We also tried larger intervals, such as $[-5,5]$, and the IVs for the CDE were almost the same.
\hpierre{What if the interval is larger, e.g., $[-5,5]$?  For the KDE it created a difficulty,
 because the bandwidth $h$ should be much larger in areas where the density is much smaller,
 such as in $[3,5]$.  But for CMC there is no $h$ and no such problem when we take a larger interval.
 I think this is a good point to show and discuss briefly.}

To construct the CDE, we define $\cG_{-k}$ as in Example~\ref{ex:sum-asm18}, for any $k=1,\dots,d$.
That is, we hide $Z_k$ and estimate the cdf by 
\[
  F(x\mid \cG_{-k}) = \PP\left[\left. a_k Z_k \le x \sigma - \sum_{j=1,\, j\not=k}^d a_j Z_j \right| \cG_{-k} \right]
	               = \Phi\left(\frac{x \sigma}{a_k} - \frac{1}{a_k} \sum_{j=1,\, j\not=k}^d a_j Z_j \right).
\]
The CDE becomes
\[
  f(x\mid \cG_{-k}) 
	 = \phi\left(\frac{x \sigma}{a_k} - \frac{1}{a_k} \sum_{j=1,\, j\not=k}^d a_j Z_j\right) \frac{\sigma}{a_k}
	 = \phi\left(\frac{x \sigma}{a_k} - \frac{1}{a_k} \sum_{j=1,\, j\not=k}^d a_j \Phi^{-1}(U_j)\right) \frac{\sigma}{a_k}
	 \eqdef \tilde g'(x,\bU)
\]
for $x\in\RR$, where $\bU = (U_1,\dots,U_{k-1},U_{k+1},\dots, U_d)$, $Z_j = \Phi^{-1}(U_j)$, 
and the $U_j$ are independent $\cU(0,1)$ random variables.
Assumption~\ref{ass:cmc-dct} is easily verified, so this CDE is unbiased.

For CMC+MC (independent sampling), we get an exact formula for the variance of the CDE
from Example~\ref{ex:sum2normal}, by taking in that example 
$Y_2 = a_k Z_k / \sigma$ and $Y_1 = X - Y_2$,
whose variances are $\sigma_2^2 = (a_k/\sigma)^2$ and $\sigma_1^2 = 1 - \sigma_2^2$, and
plugging these values into (\ref{eq:sum2normal-exact-var}).
\hpierre{Please verify this carefully, and verify that the MC empirical results agree with this formula.}
\hflorian{I computed them. The largest deviation in the $e19$ was 0.2 (for $d=5$) in the $a_k=1$ case and 0.1 in the weighted case. }%
\hamal{I tried another parameters of lattice it's gives me  $\hat\nu$=1.14 and e19 = 27.66 . 
 the variance increase from $ n=2^{14} $ (23.34) until  $ n=2^{18} (30.877) $ and after it's starts decreasing.}%
\hpierre{
\[
  \Var[\hat f_n(x)] 
  = \frac{1} {n \sigma_2^2 \sqrt{2\pi(1+2\sigma_1^2/\sigma_2^2)}} 
	     \phi\left( \sqrt{2} x / \sqrt{\sigma_2^2 + 2\sigma_1^2} \right) - \phi^2(x)/n.
\]}%
\hpierre{Thus, for MC, we could compute the exact variance instead of making experiments.}%
\hpierre{For RQMC, we also need to prove that $V_{\rm HK}(\tilde g'(x,\cdot) < \infty$ for any $x$.
Note that $\lim_{U_j\to 1} \tilde g'(x,\bU) = 0$ for any $j\not=k$.
I think we should be able to prove that 
$\tilde g'_{\frakv}(x,\bU_{\frakv},\bone) = 0$ when $|\frakv| < s = d-1$.
Then the only nonzero term in the HK variation is the one with $\frakv = \cS$.
We need to show that this term (the integral of the derivative of $\tilde g'(x,\bU)$ 
with respect to all coordinates of $\bU$) is finite for each $x$.}%
With the same argument as in the second part of Example~\ref{ex:sum-rqmc}\iflong, 
\else\ in the supplement, \fi    
we can show that $V_{\rm HK}(\tilde g'(x,\cdot)) < \infty$, uniformly in $x$ over any
bounded interval $[a,b]$, so Proposition~\ref{prop:hk-cde} applies.
We expect to observe this empirically. 
\hpierre{Since the function $\Phi^{-1}$ is infinitely differentiable, and the derivatives
of all orders are integrable, it follows from Example~\ref{ex:sum-rqmc} 
that the HK variation of $\tilde g'(x,\bu)$ is finite.}

For the GLRDE, with $Y_j = Z_j a_j/\sigma \sim \cN(0,a_j^2/\sigma^2)$,
we obtain $\partial (\log f_j(y_j))/\partial y_j = - y_j \sigma^2/a_j^2$,
$h_j(y_j)=1$, $h_{jj}(y_j)=0$, and then $\Psi_j = - Y_j \sigma^2/a_j^2 = - Z_j \sigma/a_j$.
Note that we could also replace $Y_j$ by $Z_j$ and $f_j$ by $\phi_j$ (the standard normal density),
which would give $\partial (\log \phi_j(z_j))/\partial z_j = - z_j$, 
$h_j(z_j) = a_j/\sigma$, $h_{jj}(y_j)=0$, and again $\Psi_j = - Z_j \sigma/a_j$.

In our first experiment, we take $a_j=1$ for all $j$, and $k=d$.
By symmetry, the true IV is the same for any other $k$.
Table~\ref{tab:normals1} reports the estimated rate $\hat\nu$ and the estimated value of e19 $ =-\log_2(\IV)$
for $n = 2^{19}$, for various values of $d$ and sampling methods.
The rows marked CDE-1 give the results for $k=d$, while those labeled CDE-Avg 
are for a convex combination (\ref{eq:convex-comb1}) with equal weights $\beta_{\ell} = 1/d$ 
for all $\ell=k-1$, after computing the CDE for each $k$ from the same simulations.

\begin{table}[!htbp]
\caption{Values of $\hat\nu$ and e19 for a CDE, a convex combination of CDEs, a GLRDE, and a KDE,
		  for a sum of $d=k$ normals with $a_j=1$, over $[-2,2]$.}
 	\small
  \label{tab:normals1}
  \centering
  \begin{tabular}{|c|l| c c c c c|c c c c c|}
   \cline{3-12}
  \multicolumn{2}{l|}{} &\multicolumn{5}{c|}{$\hat\nu$}	& \multicolumn{5}{c|}{e19}\\
   \cline{3-12}
	\multicolumn{2}{l|}{}	& $d=2$	& $d=3$	& $d=5$	& $d=10$ & $d=20$	& $d=2$	& $d=3$	& $d=5$	& $d=10$ & $d=20$\\
	\hline
	\multirow{4}{*}{CDE-1} 
		& MC				  & 0.99 & 0.98	& 1.02	& 1.00	& 1.02		& 22.1	& 21.4	& 20.8	& 19.8	& 19.2\\
		& Lat+s				& 2.83 & 2.00	& 1.85	& 1.40	& 1.04		& 52.3	& 39.8	& 32.1	& 23.6	& 19.7\\
		& Lat+s+b			& 2.69 & 2.11	& 1.69	& 1.14	& 1.05		& 50.5	& 41.5	& 31.1	& 21.8	& 20.0\\
		& Sob+LMS			& 2.62 & 2.10	& 1.81	& 1.04	& 1.04		& 49.3	& 40.7	& 31.1	& 21.3	& 19.7\\
%
	\hline
	\multirow{4}{*}{CDE-avg}
		& MC				  & 1.06 & 0.92	& 1.03	& 1.01	& 1.01		& 23.4	& 22.1	& 21.6	& 20.6	& 19.8\\
		& Lat+s				& 2.79 & 1.84	& 1.33	& 1.19	& 1.05		& 53.3	& 39.8	& 32.2	& 23.0	& 20.6\\
		& Lat+s+b			& 2.65 & 1.90	& 1.71	& 1.05	& 1.08		& 51.6	& 41.4	& 32.3  & 23.4  & 21.3\\
		& Sob+LMS			& 2.60 & 2.10	& 1.92	& 1.02	& 1.03   	& 49.8	& 42.0	& 33.0	& 22.7	& 20.5\\
		\hline
%
	\multirow{4}{*}{GLRDE} 
		& MC				& 0.98 & 0.95   & 1.03  & 1.05  & 1.00      & 17.0  & 16.1  & 15.9  & 14.9  & 14.1\\
		& Lat+s				& 1.51 & 1.56   & 1.45  & 0.94  & 1.06      & 28.2  & 24.9  & 22.1  & 17.8  & 17.2\\
		& Lat+s+b			& 1.49 & 1.41   & 1.05  & 1.06  & 1.04      & 27.3  & 23.9  & 20.4  & 18.8  & 17.6\\
		& Sob+LMS			& 1.49 & 1.33   & 1.15  & 0.99  & 1.16      & 27.5  & 24.0  & 21.0  & 18.3  & 17.4\\
		\hline
\color{black}
	\multirow{4}{*}{KDE} 
		& MC				  & 0.79 & 0.80	& 0.76	& 0.75	& 0.77		& 17.0	& 17.0	& 16.9	& 16.9	& 17.0\\	
		& Lat+s				& 1.08 & 1.39	& 0.92	& 0.97	& 0.76		& 25.1	& 22.4	& 19.4	& 18.2	& 17.4\\
		& Lat+s+b			& 1.23 & 0.94	& 0.72	& 0.73	& 0.74		& 24.1	& 20.1	& 18.1	& 17.3	& 17.2\\
		& Sob+LMS			& 1.18 & 0.98	& 0.83	& 0.74	& 0.77		& 24.4	& 20.8	& 17.9	& 17.2	& 17.1\\
%
	\hline
  \end{tabular}
 \end{table}

\begin{table}[!htbp]
 \caption{Values of $\hat\nu$ and e19 with a CDE for selected choices of $\cG_{-k}$,
   for a linear combination of $d=11$ normals with $a_j^2 =2^{1-j}$.}
  \label{tab:normals2}
  \centering
  \begin{tabular}{|l| c c c c|c c c c|}
     \cline{2-9}
  \multicolumn{1}{l|}{} &\multicolumn{4}{c|}{$\hat\nu$}	& \multicolumn{4}{c|}{e19}\\
     \cline{2-9}
  \multicolumn{1}{l|}{}	& $k=1$	& $k=2$	& $k=5$	& $k=11$	& $k=1$	& $k=2$	& $k=5$	& $k=11$\\
	\hline
	MC				& 1.00 & 1.02	& 1.01	& 1.00		& 22.2	& 21.0	& 18.8	& 15.5\\
	Lat+s				& 1.43 & 1.48	& 1.34	& 1.04		& 30.3	& 28.5	& 22.8	& 15.6	\\
	Lat+s+b				& 1.57 & 1.65	& 1.28	& 1.02		& 33.5	& 30.8	& 22.1	& 15.6	\\
	Sob+LMS				& 1.78 & 1.56	& 1.21	& 1.02		& 34.1	& 30.4	& 21.7	& 15.7\\
	\hline
  \end{tabular}
 \end{table}
 \hflorian{Is there a reason why we display one digit more than for the $a_j=1$, 
   the Buckling, and the Cantilever example?}
%

For MC, the rates $\hat\nu$ agree with the (known) exact asymptotic rates of
$\nu=1$ for the CDE and GLRDE, and $\nu=0.8$ for the KDE.
By looking at e19, we see that the MISE with MC is much smaller for the CDE than for the GLRDE 
and KDE, for example for $d=2$ by a factor of about 32 for CDE-1 and about 70 for CDE-avg.
For $d=20$, the gains are more modest.
RQMC methods provide huge improvements for small $d$ with the CDE.
We observe rates $\hat\nu$ larger than 2 for $d=2$ and 3.
These rates also hold for larger $d$ asymptotically, but they take longer to kick in, 
so we would need to have much larger values of $n$ to observe them.
By looking at the exponents e19,
we see that for $d=3$, for example, the MISE goes from $2^{-17}$
for the GLRDE and KDE to about $2^{-42}$ for CDE-avg with Sobol' points with LMS.  
This is a MISE reduction by a factor of about $2^{25} \approx 33$ millions!  
The large values of $\hat\nu$ imply of course that this factor is smaller for smaller $n$.
When $d$ is large, such as $d=20$, RQMC brings only a small gain.
The values of $\hat\nu$ are sometimes noisy. 
For the GLRDE with Lat+s and $d=5$, for example, the large $\hat\nu = 1.45$ comes from the fact that 
the IV for $n=2^{14}$ (not shown) is unusually large (an outlier).
Looking at e19 gives a more robust assessment of the performance.
The GLRDE performs better than the KDE under RQMC for small $d$, but is not competitive with the CDE.
Under MC, the GLRDE is slightly worse than the KDE.

\hflorian{With RQMC we observe gains of the GLR over the KDE for small $d$ as well, even though 
they are much more moderate than for the CDE and these improvements completely disappear for $d=20$.
Under MC, the  GLR and the KDE perform equally well, but for $d\geq3$ the MISE of the GLR becomes 
larger than that of the KDE. This happens because the  (asymptotic) MISE for the KDE with MC only 
depends on the kernel and the true density and is therefore independent of $d$ in this example.
For the GLR only the \emph{MISE-rate} does not depend on $d$ under MC.}

In our second experiment, we take $a_j^2 = 2^{1-j}$ for $j=1,\dots,d$.
Now, the choice of $k$ for the CDE makes a difference, and the best choice will obviously be $k=1$,
i.e., hide the term that has the largest variance.
Note that with MC, $\Var[X] = 2-2^{-d}$, and when we apply CMC by hiding $a_k Z_k$ from the sum,
we hide a term of variance $a_k^2 = 2^{1-k}$ and generate a partial sum $S_{-k}$ of variance $2 - 2^{1-k}-2^{-d}$.
Both terms have a normal distribution with mean 0.
The results of Example~\ref{ex:sum2normal} hold with these variances.
Table~\ref{tab:normals2} reports the numerical results for $d=11$ and $k = 1, 2, 5, 11$. 
\hpierre{The lattice parameters were found again with Lattice Builder with $\mathcal{P}_2$,
  with weights $\Gamma_\frakv = 0.6^{|\frakv|}$.}

The MC rates $\hat\nu$ agree again with the theory, but here the IV depends very much on the choice of $k$,
and this effect is more significant when $k$ is smaller.
For example, for Sobol' points, the IV with $k=1$ is about 300,000 times smaller than with $k=11$.
The reason is that with $k=11$, we hide only a variable having a very small variance, 
so the CDE for one sample is a high narrow peak, and the HK variation of $\tilde g'(x,\bu)$ is very large.
For $k=1$ or 2, we have the opposite and the integrand is much more RQMC-friendly.

\subsection{Buckling strength of a steel plate}
\label{sec:buckling}

This is a six-dimensional example, taken from \cite{tSCH16a}.
It models the buckling strength of a steel plate by
\begin{equation}
\label{eq:buckling-def}
  X = \left( \frac{2.1}{\Lambda} - \frac{0.9}{\Lambda^2}\right) 
	           \left( 1- \frac{0.75 Y_5}{\Lambda} \right)\left( 1-\frac{2 Y_6 Y_2}{Y_1} \right),
\end{equation}
where $\Lambda = ({Y_1}/{Y_2}) \sqrt{Y_3/Y_4}$, and $Y_1,\dots,Y_6$ 
are independent random variables whose distributions are given in Table~\ref{tab:buckling-parameters}. 
Each distribution is either normal or lognormal, and the table gives the mean 
and the coefficient of variation (cv), which is the standard deviation divided by the mean. 
We estimate the density of $X$ over $[a,b] = [0.5169,0.6511]$, which contains 
about 99\% of the density (leaving out 0.5\% on each side).
There is a nonzero probability of having $Y_4 \le 0$, in which case $X$ is undefined, but this
probability is extremely small and this has a negligible impact on the density estimator over $[a,b]$,
so we just ignore it (alternatively we could truncate the density of $Y_4$).
There are also negligible probabilities that the density estimates below are negative 
and we ignore this.
%
\hflorian{I don't think $X > 0$ is the crucial condition here, although it contains some critical sub-cases. 
Actually, what we need to derive the estimators is is $V_1,V_2,V_3>0$ and $\Lambda>0$ or $Y_1Y_2>0$, 
depending on which variable we hide. For the $V_j$ we can even go into more detail, but I don't think that 
it would make much sense.}%

\begin{table}[!htbp]
	\centering
	\caption{Distribution of each parameter for the buckling strength model.}
	\label{tab:buckling-parameters}
	\begin{tabular}{c|crl}
		parameter	&  distribution	& mean			& cv \\
		\hline
		$Y_1$   & normal	  &$23.808$	& 0.028 \\
		$Y_2$		& lognormal	&$0.525$	& 0.044 \\
		$Y_3$ 	& lognormal	&$44.2$		&0.1235\\
		$Y_4$		& normal	  & $28623$	&0.076\\
		$Y_5$	  & normal 	  &$0.35$			&0.05\\
		$Y_6$	  & normal	  &$5.25$			&0.07
	\end{tabular}
\end{table}

For this example, computing the density of $X$ conditional on $\cG_{-5}$ or $\cG_{-6}$ 
(i.e., when hiding $Y_5$ or $Y_6$) is relatively easy, so we will try and compare these two choices.
If we hide one of the variables that appear in $\Lambda$, the CDE would be harder to compute
(it would require to solve a polynomial equation of degree 4 for each sample), and we do not do it.
%
%
Let us define 
\[
   V_1 = \frac{2.1}{\Lambda} - \frac{0.9}{\Lambda^2},  \qquad
   V_2 = 1 - \frac{2 Y_6 Y_2}{Y_1}, \quad \mbox{and \ } V_3 = 1 - \frac{3 Y_5}{4\Lambda}.
\]
Then we have 
\begin{equation*}
  X \leq x   \quad\Leftrightarrow\quad  
	Y_5 \ge \left( 1-\frac{x}{V_1 V_2} \right) \frac{4\Lambda}{3}
\end{equation*}
and
\begin{equation*}
 f(x\mid \cG_{-5}) = f_{5}\left(\left(1-\frac{x}{V_1 V_2}\right) \frac{4\Lambda}{3} \right) 
                   \frac{4\Lambda}{3 V_1 V_2}
    = \phi \left( \frac{\left(1-x/(V_1 V_2)\right){4\Lambda}/{3} -0.35}{0.0175}  \right) \frac{4\Lambda}{0.0525\cdot V_1 V_2}.
\end{equation*}
Similarly, 
\begin{equation*}
 f(x\mid \cG_{-6}) = f_{6}\left(\left(1-\frac{x}{V_1 V_3}\right) \frac{Y_1}{2 Y_2} \right) 
                  \frac{Y_1}{2Y_2 V_1 V_3}
        = \phi \left ( \frac{\left(1-{x}/(V_1 V_3)\right) {Y_1}/(2 Y_2)  -5.25}{0.3675}  \right ) \frac{Y_1}{0.735\cdot Y_2 V_1 V_3}.
\end{equation*}

For GLRDE using $Y_6$, let 
$C = \left({2.1}/{\Lambda} - {0.9}/{\Lambda^2}\right) \left( 1- {0.75 Y_5}/{\Lambda} \right)$.
We have $X = h(\bY) = C (1 - 2 Y_6 Y_2/Y_1)$, $h_6(\bY) = 2 C Y_2 / Y_1$, $h_{66}(\bY) = 0$,
$\partial \log f_6(Y_6)/\partial Y_6 = -(Y_6-\mu_6)/\sigma_6^2$, and 
$\Psi_6 = Y_1(Y_6-\mu_6)/(2C Y_2 \sigma_6^2)$.

\begin{table}[!htbp]
  \caption{Values of $\hat\nu$ and e19 with a CDE for $\cG_{-5}$, $\cG_{-6}$, their combination,
	   GLRDE, and the KDE, for the buckling strength model.}
  \label{tab:buckling}
  \centering
	\small
  \begin{tabular}{|l|c c c c c|c c c c c|}
     \cline{2-11}
  \multicolumn{1}{l|}{} &\multicolumn{5}{c|}{$\hat\nu$}	& \multicolumn{5}{c|}{e19}\\
     \cline{2-11}
  \multicolumn{1}{l|}{}	& $\cG_{-5}$	& $\cG_{-6}$	& comb.	&  GLRDE & KDE 
	                      & $\cG_{-5}$	& $\cG_{-6}$	& comb.	&  GLRDE & KDE \\
	\hline
	MC				  & 1.00 	& 1.00		& 1.00	& 0.98  & 0.76   & 13.5		& 15.4		& 15.4 & 10.2  & 11.7  \\
	Lat+s				& 1.89 	& 1.56		& 1.56	& 1.29  & 0.81  & 20.0		& 24.9		& 24.9 & 16.6  & 13.7  \\ 
	Lat+s+b				& 1.46	& 1.65		& 1.60	& 1.19  & 0.85  & 17.5		& 25.1		& 25.1 & 15.9  & 12.7  \\
	Sob+LMS				& 1.40 	& 1.75		& 1.75	& 1.16  & 0.81  & 17.7		& 25.5		& 25.5 & 15.9  & 12.4  \\
	\hline
  \end{tabular}
 \end{table}

\begin{figure}[!htbp]
	\centering
	\begin{tikzpicture} 
	\begin{axis}[ 
	xlabel=$\log_2n$,
	ylabel=$\log_2\MISE$,
	grid,
	width = 7 cm,
	height = 5 cm,
	legend style={at={(0.65,0.75)}, anchor={south west}, font=\scriptsize},	
	] 
	\addplot table[x=logN,y=logIV] { 
      logN  logIV 
      14.0  -8.5888646734209 
      15.0  -9.432264252285806 
      16.0  -10.559705394035218 
      17.0  -11.537675430110456 
      18.0  -12.525597566336398 
      19.0  -13.528333822089996 
 }; 
      \addlegendentry{MC}
      \addplot table[x=logN,y=logIV] { 
      logN  logIV 
      14.0  -10.249822669534305 
      15.0  -13.477838494885265 
      16.0  -15.825114316390513 
       17.0  -16.02432354302361 
      18.0  -19.157757500500075 
      19.0  -20.04066944408722 
 }; 
      \addlegendentry{Lat+s}
      \addplot table[x=logN,y=logIV] { 
      logN  logIV 
      14.0  -10.499627674658207 
      15.0  -11.30669474934358 
      16.0  -13.2633122488705 
      17.0  -14.445726976357793 
      18.0  -16.203261050837373 
      19.0  -17.51378751849192 
 }; 
      \addlegendentry{Lat+s+b}
      \addplot table[x=logN,y=logIV] { 
      logN  logIV 
      14.0  -10.973169341542103 
      15.0  -11.94546933172819 
      16.0  -13.047650507386173 
      17.0  -15.013953026342824 
      18.0  -16.35865350730346 
      19.0  -17.743172055556 
 }; 
      \addlegendentry{Sob+LMS}
%
	%
	\end{axis}
	\end{tikzpicture}
	\begin{tikzpicture} 
	\begin{axis}[ 
	xlabel=$\log_2n$,
	grid,
	width = 7 cm,
	height = 5cm,
	legend style={at={(0.65,0.75)}, anchor={south west}, font=\scriptsize},
	] 
	\addplot table[x=logN,y=logIV] { 
      logN  logIV 
      14.0  -10.468827630852278 
      15.0  -11.455637612646624 
      16.0  -12.474042012883793 
      17.0  -13.62548985151096 
      18.0  -14.434557132285814 
      19.0  -15.42972671866297 
 }; 
      \addlegendentry{MC}
%
%
      \addplot table[x=logN,y=logIV] { 
      logN  logIV 
      14.0  -16.37703161027969 
      15.0  -18.868313141208944 
      16.0  -20.15199009155408 
      17.0  -21.418950129881626 
      18.0  -22.452965207777545 
      19.0  -24.918563623267413 
 }; 
      \addlegendentry{Lattice+Shift}
      \label{Lattice+Shift}
      \addplot table[x=logN,y=logIV] { 
      logN  logIV 
      14.0  -17.36887500221374 
      15.0  -18.505686288096655 
      16.0  -20.376515065554663 
      17.0  -22.565072140303222 
      18.0  -24.121191775232976 
      19.0  -25.082918114354452 
 }; 
      \addlegendentry{Lattice+Baker}
      \addplot table[x=logN,y=logIV] { 
      logN  logIV 
      14.0  -17.058388066453347 
      15.0  -18.350745091350962 
      16.0  -20.201216352192745 
      17.0  -22.176208407458116 
      18.0  -24.070175509342413 
      19.0  -25.46993555515593 
 }; 
      \addlegendentry{Sob+LMS}
%
	%
	\end{axis}
	\end{tikzpicture}
	\caption{MISE vs $n$ in log-log scale for the $\cG = \cG_{-5}$ (left) and $\cG = \cG_{-6}$ (right) 
	         for the buckling strength model.}
	\label{fig:buckling-mise}
\end{figure}
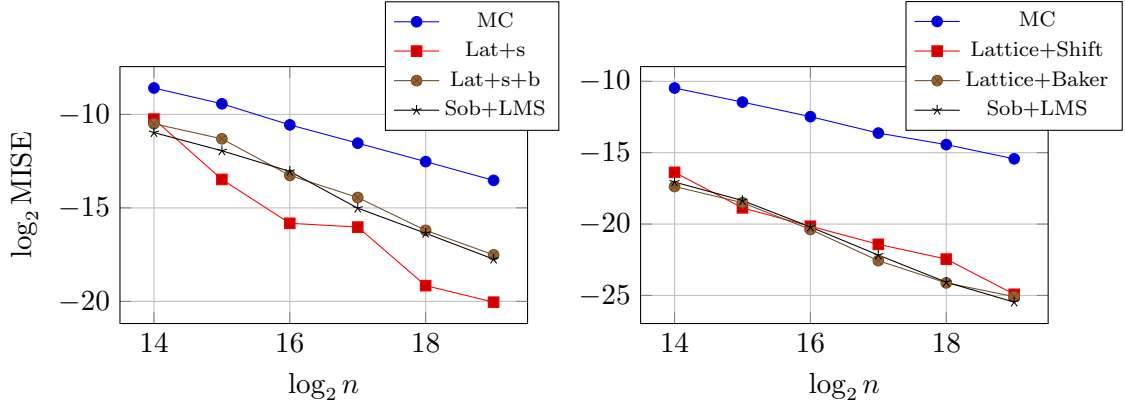

Table~\ref{tab:buckling} summarizes the results.  
\hflorian{The lattices were obtained with Lattice Builder using the $\mathcal{P}_2$ criterion
 with weights $\gamma_\frakv = 0.8^{|\frakv|}$.}%
We see again that with a very simple conditioning, the CDE with RQMC performs extremely well 
and much better than the GLRDE and the KDE.
It is also much better to condition on $\cG_{-6}$ than on $\cG_{-5}$,
and combining the two provides no significant improvement. 
The GLRDE is better than the KDE under RQMC, but not under MC.
\hflorian{Compared to the other methods, the IV of the CDE is again
very small, especially under RQMC. Depending on the RQMC point set, the GLRDE a improves upon
the KDE in terms of MISE by a factor ranging between 8 to 11,
but again performs a little bit worse than the GLRDE with MC.}%
%
Figure~\ref{fig:buckling-mise} displays the IV as a function of $n$ in a log-log-scale  
for the CDE with $\cG_{-5}$ and $\cG_{-6}$.
It unveils a slightly more erratic behavior of the MISE for the shifted lattice rule (Lat+s) 
than for the other methods; the performance depends on the choice of parameters of the lattice rule
and their interaction with the particular integrand.
\hflorian{Do you think that we still need this sentence, or is it ``acceptable'' now. FYI: 
 I have tried fast-CBC with $\mathcal{P}_2$ for 12 different configurations of weights and among those, 
 these (order-dependent, $\Gamma_k=0.8^{k}$) are the least erratic.}
\hflorian{KDE+Lat+s is very erratic for small $n$. So, the rate is flawed, but the 
e19 is ok (and good).}
\hpierre{The lattice for $\log_2 n = 17$ does not seem to perform well. 
  Perhaps a better one can be found, with different weights?}
\hflorian{We considered this problem already in the previous version. 
The ones we had were found with order dependent weights $\Gamma_k=0.8^{k}$. 
Taking $0.6^{k}$ makes $\cG_{-6}$ ``nicer'', but for $\cG_{-5}$ it is much harder. 
Larger and smaller weights made everything much worse. }

\subsection{Density of the failure time of a system}
\label{sec:network-failure}

We consider a $d$-component system in which each component starts in the operating
mode (state 1) and fails (jumps to state 0) at a certain random time, to stay there forever.
Let $Y_j$ be the failure time of component $j$ for $j=1,\dots,d$.
For $t\ge 0$, let $W_j(t) = \II[Y_j > t]$ be the state of component $j$ and 
$\bW(t) = (W_1(t),\dots,W_{d}(t))^\tr$ the system state, at time $t$.
The system is in the failed mode at time $t$ if and only if $\Phi(\bW(t)) = 0$,
where $\Phi : \{0,1\}^d \to \{0,1\}$ is called the \emph{structure function}. 
Let $X = \inf\{t\ge 0 : \Phi(\bW(t)) = 0\}$ be the random time when the system fails.
We want to estimate the density of $X$.
A straightforward way of simulating a realization of $X$ is to generate the component
lifetimes $Y_j = \inf\{t\ge 0 : W_j(t) = 0\}$ for $j=1,\dots,d$, and then compute $X$ from that. 

As in Section~\ref{sec:san}, the GLRDE method 
of Section~\ref{sec:GLRDE-estimator} does not work for this example,
because $h_j(\bY) \not= 0$ only when $X = Y_j$,
and there is no $j$ for which this is certain to happen. 

If the $Y_j$ are independent and exponential, one can construct 
a CMC estimator of the cdf $F(x) = \PP[X\le x]$ as follows \citep{pGER10a,vBOT13a}.
Generate all the $Y_j$'s and sort them in increasing order.
Then, erase their values and retain only their order, which is a permutation $\pi$ of
$\{1,\dots,d\}$.  
Compute the critical number $C = C(\pi)$, defined as the number of component failures required
for the system to fail (that is, the system fails at the $C$th component failure, for the given $\pi$).
Note that $C$ can also be computed by starting with all components failed and resurrecting them one by one 
in reverse order of their failure, until the system becomes operational.
Computing $C$ using this reverse order is often more efficient \citep{vBOT16m}.
Then compute the conditional cdf $\PP[X\le x \mid \pi]$, where $X$ is the time of the 
$C$th component failure.  This is an unbiased estimator of $F(x)$ with smaller variance 
than the indicator $\II[X\le x]$.
It can also be shown that in an asymptotic regime in which the component failure rates converge
to 0 so that $1-F(x) \to 0$, the relative variance of this CMC estimator of $1-F(x)$ remains 
bounded whereas it goes to infinity with the conventional estimator $\II[X > x]$;
i.e., the CMC estimator has bounded relative error \citep{vBOT13a,vBOT16m}.
This $X$ is a sum of $C$ independent exponentials, so it has a hypoexponential
distribution, whose cdf has an explicit formula that can be written in terms of a matrix exponential,
and developed explicitly as a sum of products in terms of the rates of the exponential lifetimes, 
as explained below.   
\hpierre{Note that the rates $\Lambda_j$ in the formula given in 
  \cite{vBOT13a} are for a repair process instead of a failure process.}
By taking the derivative of the conditional cdf formula with respect to $x$, 
one obtains the conditional density.

More specifically, let component $j$ have an exponential lifetime with rate
$\lambda_j > 0$, for $j=1,\dots,d$.
For a given realization, let $\pi(j)$ be the $j$th component that fails and let
$C(\pi)=c$ for the given $\pi$, let $A_1$ be the time until the first failure, 
and let $A_j$ be the time between the $(j-1)$th and $j$th failures, for $j > 1$.
Conditional on $\pi$, we have $X = A_1 + \cdots + A_c$ where the $A_j$'s are independent
and $A_j$ is exponential with rate $\Lambda_j$ for all $j\ge 1$, 
with $\Lambda_1 = \lambda_1 + \cdots + \lambda_d$, 
and $\Lambda_j = \Lambda_{j-1} - \lambda_{\pi(j-1)}$ for all $j\ge 2$. 
The conditional distribution of $X$ is then hypoexponential with cdf
\[
  \PP[X\le x \mid \pi] = \PP[A_1 + \cdots + A_c \le x \mid \pi] 
	  = 1 - \sum_{j=1}^c p_j e^{-\Lambda_j x},
\]
where
\[
  p_j = \prod_{k=1, k\not=j}^c  \frac{\Lambda_k}{\Lambda_k-\Lambda_j}. 
\]
See \cite{pGER10a}, Appendix A, and \cite{vBOT16m}, for example.
Taking the derivative with respect to $x$ gives the CDE
\[
  f(x \mid \pi) = \sum_{j=1}^c \Lambda_j p_j e^{-\Lambda_j x},
\]
in which $c$, the $\Lambda_j$ and the $p_j$ depend on $\pi$.
This conditional density is well defined and computable everywhere in $[0,\infty)$.
There are instability issues for computing $p_j$ when $\Lambda_k-\Lambda_j$ is close to 0 for some $k\not=j$,
but this can be addressed by a stable numerical algorithm of \cite{mHIG09a}.

All of this can be generalized easily to a model in which the lifetimes are dependent,
with the dependence modeled by a Marshall-Olkin copula \citep{vBOT16m}.
In that model, the $Y_j$ represent the occurrence times of shocks that can take down 
one or more components simultaneously.  

It is interesting to note that although $f(x \mid \pi)$ is an unbiased estimator of 
the density $f(x)$ at any $x$, this estimator is a function of the permutation $\pi$ only,
so it takes its values in a finite set, which means that the corresponding $\tilde g(\bu)$
is a piecewise constant function, which is not RQMC-friendly.
Therefore, we do not expect RQMC to bring a very large gain.

\begin{table}[!htbp]	
	\centering
	\caption{Values of $\hat\nu$ and e19 with the CDE, for the network reliability example. }
	\label{tab:network-failure}
	\small
	\begin{tabular}{l| c c }
	  \hline
		 &  $\hat\nu$	& e19 \\
		\hline
		 MC 		  & 1.00		& 19.9\\
		 Lat+s   	& 1.22		& 23.9\\
		 Lat+s+b	& 1.19		& 23.8\\
		 Sob+LMS 	& 1.33 	  & 23.9\\
  \hline
	\end{tabular}
\end{table}

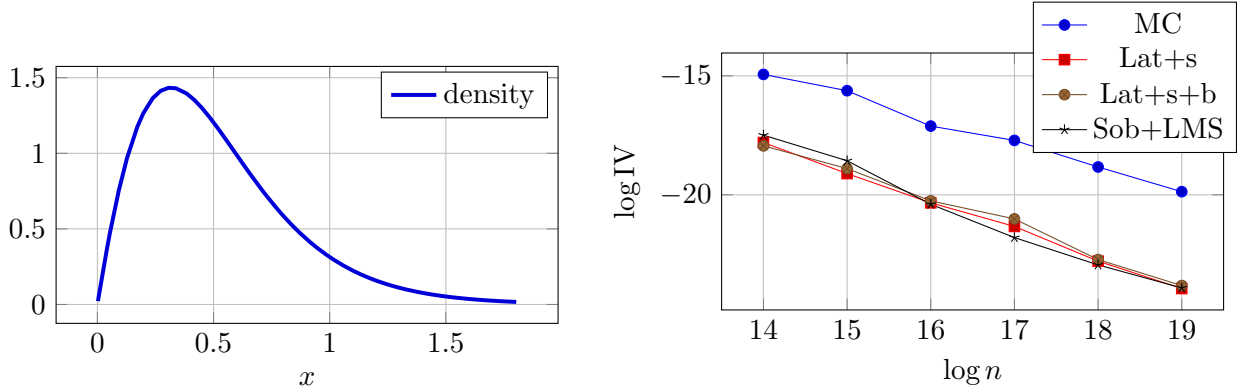
\begin{figure}[hbt]
\centering
  \begin{tikzpicture} 
    \begin{axis}[ 
      title ={},
      xlabel=$x$,
	  width = 0.5\textwidth,
	  height = 5cm,
	   cycle list name=dens,
       grid,
      ] 
      \addplot table[x=x,y=y] { 
      x  y 
      0.0023230334222318957  0.021889337060761006 
      0.04240586917110737  0.3788791843102802 
      0.05325825774998814  0.4678698446463949 
      0.092223420112755  0.7543252156860056 
      0.12761732735707054  0.9686944016652647 
      0.17236355722284408  1.1769177525197994 
      0.19718057990260437  1.2634124285026238 
      0.238283769997187  1.3647586559991387 
      0.2729081653349517  1.413567699022711 
      0.30624185465640985  1.4330994809270416 
      0.3378978739368398  1.4302628844122898 
      0.3844790527767323  1.39513224841477 
      0.41683337357004674  1.3535879769160675 
      0.44945093958428023  1.3009721556756177 
      0.4835221417003091  1.2374613892524162 
      0.5175539218100245  1.1679972851261449 
      0.5596869198270095  1.077122837091747 
      0.600476336728632  0.9871922434424676 
      0.6252757744934879  0.9327163496509246 
      0.656127266958638  0.866009757705516 
      0.7082630377790918  0.7577656990213362 
      0.7291913483433148  0.7163604650037965 
      0.7755979689507088  0.6294592933522631 
      0.8009876366524372  0.5849817907132182 
      0.8347511520535313  0.5293349480338 
      0.8727511204611541  0.47152969008681583 
      0.9165096954356267  0.4112112991101213 
      0.9434200261054388  0.3773382426343356 
      0.9848871140643395  0.32972378358031534 
      1.01478590142484  0.2986661789869366 
      1.057721675408382  0.258537214440538 
      1.0959180753542621  0.22693789625483443 
      1.1335998914517689  0.19921755251308623 
      1.168989383087217  0.1760320104933344 
      1.203917569062904  0.1556095724861285 
      1.2373942653035108  0.13812152111325363 
      1.2689989461812845  0.12331416324086406 
      1.2970120295687444  0.1114508733297215 
      1.3319009112471738  0.09818225473496064 
      1.383316884759155  0.08133528260204631 
      1.41966453129937  0.07113371348219835 
      1.4463528031153752  0.06443801693940115 
      1.4838402509983752  0.0560498877158252 
      1.5209337075962244  0.048793854031599186 
      1.544555813545256  0.04465680772517145 
      1.5968219577052973  0.0366784099703041 
      1.6225979203232024  0.03327320591362017 
      1.6609981364836943  0.028765501234052155 
      1.6996078280381317  0.024837053801642348 
      1.7341578686220565  0.021770716454475716 
      1.760297356975756  0.01970071117948705 
      1.8030023649223177  0.016727440314351857 
 }; 
   \addlegendentry{density}
    \end{axis}
  \end{tikzpicture}
\hspace{10pt}
 \begin{tikzpicture} 
    \begin{axis}[ 
      xlabel=$\log n$,
      ylabel=$\log{\rm IV}$,
      width = 0.5\textwidth,
   	  height = 5cm,
       grid,
       legend style={at={(0.62,1.20)},anchor=north west}
      ] 
      \addplot table[x=logN,y=logIV] { 
      logN  logIV 
      14.0  -14.94022770458407 
      15.0  -15.624091929527514 
      16.0  -17.112147980123485 
      17.0  -17.71402397194226 
      18.0  -18.83015490025875 
      19.0  -19.868773167181406 
 }; 
      \addlegendentry{MC}
      \addplot table[x=logN,y=logIV] { 
      logN  logIV 
      14.0  -17.80305240801714 
      15.0  -19.10783578137406 
      16.0  -20.33668371897971 
      17.0  -21.328756658781206 
      18.0  -22.78785853288337 
      19.0  -23.938698116468675 
 }; 
      \addlegendentry{Lat+s}
      \addplot table[x=logN,y=logIV] { 
      logN  logIV 
      14.0  -17.943213804830975 
      15.0  -18.895658788975958 
      16.0  -20.253947735808012 
      17.0  -21.011499467164786 
      18.0  -22.721973001101595 
      19.0  -23.813651283408436 
 }; 
      \addlegendentry{Lat+s+b}
      \addplot table[x=logN,y=logIV] { 
      logN  logIV 
%
      14.0  -17.492185482753406 
      15.0  -18.56843534396864 
      16.0  -20.398621971141207 
      17.0  -21.7973039986028 
      18.0  -22.94629384089592 
      19.0  -23.912730848298285
 }; 
      \addlegendentry{Sob+LMS}
%
    \end{axis}
  \end{tikzpicture}
 \caption{Density (left) and $\log\IV$ as a function of $\log n$ (right) 
	  for the network failure time.}
  \label{fig:network-failure}
\end{figure}

\hpierre{-- As another illustration, we could take the dodecahedron example from \citep{vBOT13a}.}
For a numerical illustration, we take the same graph as in Section~\ref{sec:san}.
For $j=1,\dots,13$, $Y_j$ is exponential with rate $\lambda_j$ and the $Y_j$ are independent.  
The system fails as soon as there is no path going from the source to the sink.  
For simplicity, here we take $\lambda_j = 1$ for all $j$, although taking different $\lambda_j$'s
brings no significant additional difficulty.
We estimate the density over the interval $(a,b]= (0, 1.829]$, which cuts off roughly 
1\% of the probability on the right side. 
Table~\ref{tab:network-failure} and Figure~\ref{fig:network-failure} give the results.
The density of $X$ estimated with $n=2^{20}$ random samples is shown on the left 
and the IV plots are on the right.
Despite the discontinuity of $\tilde g$, RQMC outperforms MC in terms of the IV 
by a factor of about $2^{4} = 16$ for $n=2^{19}$, and also by improving the empirical rate $\hat\nu$
to about $-1.2$ for lattices and even better with Sobol' points.
The Sobol' points used here were constructed using LatNet Builder \citep{rMAR20a}
with a CBC search based on the $t$-value of all projections up to order 6, 
with order-dependent weights $\gamma_{k} = 0.8^{k}$ for projections of order $k$.

\subsection{Estimating a quantile with a confidence interval}
\label{sec:quantile}


For $0 < q < 1$, the $q$-quantile of the distribution of $X$ is defined as 
$\xi_q = F^{-1}(q) = \inf\{x : F(x)\ge q\}$.
Given $n$ i.i.d.\ observations of $X$, a standard (consistent) estimator of $\xi_q$ is 
the $q$-quantile of the empirical distribution, defined 
as $\hat\xi_{q,n} = X_{(\lceil nq\rceil)}$, where $X_{(1)},\dots,X_{(n)}$ are the 
$n$ observations sorted in increasing order (the order statistics).
We assume that the density $f(x)$ is positive and continuously differentiable 
in a neighborhood of $\xi_q$.  Then we have the central limit theorem (CLT): 
\[
  \sqrt{n} (\hat\xi_{q,n} - \xi_q) / \sigma_\xi  \Rightarrow \cN(0,1) \qquad \mbox{ for } n\to\infty, 
\] 
where $\sigma^2_\xi = q(1-q)/f^2(\xi_q)$ \citep{tSER80a}.
This provides a way to compute a confidence interval on $\xi_q$, but requires the estimation
of $f(\xi_q)$, which is generally difficult.  Some approaches for doing this include 
finite differences with the empirical cdf, batching, and sectioning 
\citep{sASM07a,tNAK14a,tNAK14b}.

In our setting, one can do better by taking the $q$-quantile $\hat\xi_{\cmc,q,n}$
of the conditional cdf
\[
  \hat F_{\cmc,n}(x)  = \frac{1}{n} \sum_{i=1}^n F(x\mid \cG^{(i)}).
\]
That is, $\hat\xi_{\cmc,q,n} = \inf\{x : \hat F_{\cmc,n}(x) \ge q\}$.
This idea was already suggested by \cite{tNAK14b}, who pointed out that this estimator obeys
a CLT just like $\hat\xi_{q,n}$, but with the variance constant $\sigma^2_\xi$
replaced by $\sigma^2_{\cmc,\xi} = \Var[F(\xi_q \mid \cG)]/f^2(\xi_q) \le \sigma^2_\xi$.
This is an improvement on the quantile estimator itself.
Our CDE approach also provides an improved estimator of the density $f(\xi_q)$ which 
appears in the variance expression.
We estimate $f(\xi_q)$ by $\hat f_{\cde,n} (\hat\xi_{\cmc,q,n})$.  
This provides a more accurate confidence interval of $\xi_q$.

Further improvements on the variances of both the quantile and density 
estimators can be obtained by using RQMC to generate the realizations $\cG^{(i)}$.
In particular, if $\tilde g(\xi_q, \bu) = F(\xi_q \mid \cG)$ is a sufficiently smooth function 
of $\bu$, $\Var[\hat\xi_{\cmc,q,n}]$ can converge at a faster rate than $\cO(n^{-1})$.
When using RQMC with $n_r$ randomizations to estimate a quantile, the quantile 
estimator will be the empirical quantile of all the $n_r\times n$ observations.
\hpierre{How do we estimate the variance of this quantile estimator?  Perhaps not obvious.}
\hpierre{We could have a short numerical example to illustrate this, based on an
 example which is already in the paper, e.g., the cantilever example. 
 We can estimate the 0.95 quantile and the density at that quantile.}

A related quantity is the \emph{expected shortfall}, defined as 
$c_q = \EE[X\mid X > \xi_q] = \xi_q - \EE[(\xi_q - X)^+] / q$ 
which is often estimated by its empirical version \citep{sHON14a}
\[
   \hat c_{q,n} = \hat\xi_{q,n} - \frac{1}{nq} \sum_{i=1}^n (\hat\xi_{q,n} - X_i)^+.
\]
This estimator obeys the CLT
$
  \sqrt{n} (\hat c_{q,n} - c_q) / \sigma_c  \Rightarrow \cN(0,1)
$ 
for $n\to\infty$, where $\sigma^2_c = \Var[(\xi_q - X)^+]/q^2$, if this variance is finite \citep{sHON14a}.
By improving the quantile estimator, CDE+RQMC can also improve the expected shortfall estimator
a well as the estimator of the variance constant $\sigma^2_c$ and the quality of 
confidence intervals on $c_q$. We leave this as a topic for future work.


\hpierre{We should have a very short discussion here of related papers on VaR and 
quantile sensitivity estimation:   \cite{sHE19a,sHON09a,sHON14a,vNAK14a}.   
Also some papers on estimating the sensitivity of conditional value at risk (CVaR) 
with respect to a parameter $\theta$.  See for example \cite{sHON09b,sHON14a,sHON14b}.
I think it is important to cite some of the papers of Jeff Hong and make appropriate links,
with our future work.}

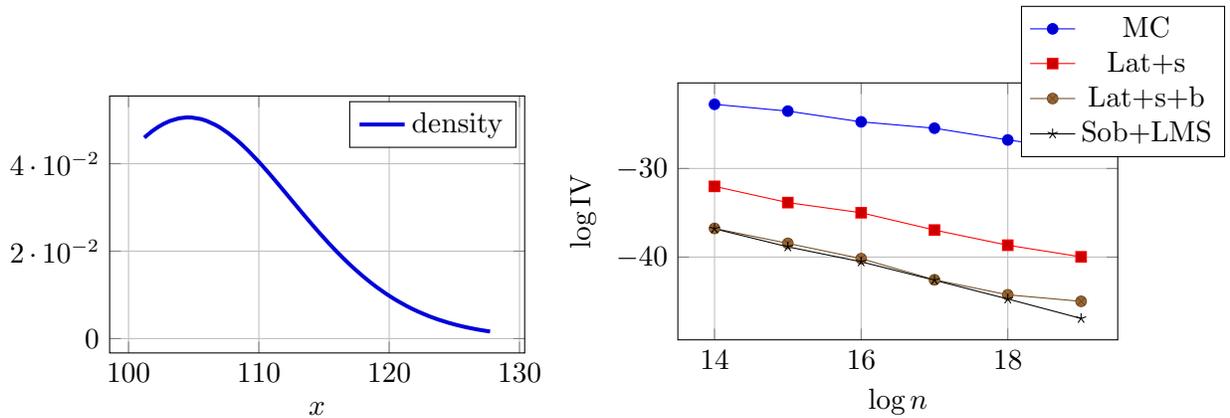
\begin{figure}[thb]
\centering
\hbox{\begin{tikzpicture} 
    \begin{axis}[ 
      title ={},
      xlabel=$x$,
	  width = 0.43\textwidth,
	  height = 5cm,
	   cycle list name=dens,
       grid,
       scaled y ticks=false,
      ] 
      \addplot table[x=x,y=y] { 
      x  y 
      101.20607455922807  0.0459616833815542 
      101.8080156585258  0.04743813376722539 
      102.27107734476282  0.04839304797939822 
      102.70356584256345  0.049135258848334966 
      103.19737523484157  0.04979975445211217 
      103.87764132070127  0.05038775068720478 
      104.46292460121201  0.05058689095119182 
      104.84915361954383  0.0505637034558061 
      105.54968387235802  0.05021453640707879 
      105.95131237626133  0.04984128833934768 
      106.67843690128517  0.04886187392144015 
      107.02478176327975  0.04826557791431572 
      107.63769655273917  0.04702137737726421 
      108.30566934201036  0.045416146147137515 
      108.70846950143121  0.044336900804515966 
      109.21093508335979  0.042888210528317774 
      109.94456947833311  0.04060077992154888 
      110.26303426946761  0.039555454599464863 
      110.80786186110156  0.037709723966024004 
      111.36013692030873  0.03578279404619359 
      112.10819637057736  0.033118979852920886 
      112.5491866648476  0.0315369320577495 
      113.16475819011661  0.029334598306763698 
      113.61506663159136  0.027739490907371076 
      114.13378922943923  0.025930237224972594 
      114.65811718545409  0.024142528275577214 
      115.33836207482031  0.021900356239763268 
      115.82286024135003  0.02036534843587059 
      116.38719408729236  0.01865007771268235 
      116.76404081041699  0.017551217928570093 
      117.3391989416979  0.015949710319387897 
      117.97868100208282  0.014280179858777985 
      118.60628158866469  0.01275795911587771 
      119.01080592535217  0.011838221012779263 
      119.71324621062216  0.010354687925984556 
      120.22743193314574  0.00935853497680795 
      120.59226581667417  0.008696621635874216 
      121.32096696310151  0.007482363778496054 
      121.8175726348924  0.006733885859636989 
      122.2230781579411  0.006167972993910236 
      122.95160469466626  0.005247965184603101 
      123.25258472362239  0.004902249092094642 
      123.79231965933648  0.004329508197655877 
      124.56378394844741  0.003608912383493579 
      125.1379567172101  0.003141083243164034 
      125.66038939669657  0.002761582858579987 
      126.08275973671249  0.0024844003532584107 
      126.72179436659314  0.0021110890454771217 
      127.05929866911912  0.001934535374230523 
      127.74387081849605  0.0016159463169558852 
 }; 
   \addlegendentry{density}
\end{axis}
\end{tikzpicture}
\hfill
\begin{tikzpicture} 
    \begin{axis}[ 
      xlabel=$\log n$,
      ylabel=$\log{\rm IV}$,
      width = 0.45\textwidth,
	  height = 5cm,
       grid,
       legend style={at={(0.78,1.30)},anchor=north west}
      ] 
      \addplot table[x=logN,y=logIV] { 
      logN  logIV 
14.0	 -22.761337
15.0	 -23.525771
16.0	 -24.737236
17.0	 -25.456558
18.0	 -26.780851
 19.0	 -27.917220
 }; 
      \addlegendentry{MC}
      \addplot table[x=logN,y=logIV] { 
      logN  logIV 
14.0	 -32.017227
15.0	 -33.850174
16.0	 -34.995498
17.0	 -36.945304
18.0	 -38.658706
19.0	 -39.970254

 }; 
      \addlegendentry{Lat+s}
      \addplot table[x=logN,y=logIV] { 
      logN  logIV 
14.0     -36.765463
15.0     -38.449402
16.0     -40.173345
17.0     -42.558317
18.0     -44.254211
19.0     -44.991242
 }; 
      \addlegendentry{Lat+s+b}
      \addplot table[x=logN,y=logIV] { 
      logN  logIV  
14.0	 -36.803557
15.0	 -38.839757
16.0	 -40.530596
17.0	 -42.607936
18.0	 -44.725274
19.0	 -46.921732
 }; 
      \addlegendentry{Sob+LMS}
%
%
    \end{axis}
  \end{tikzpicture}
}
  \caption{Estimated density (left) and $\log \IV$ as a function of $\log n$ (left) 
	  for the Asian option.}
  \label{fig:asianBridge}
\end{figure}

\begin{figure}[htbp]
  \begin{tikzpicture} 
    \begin{axis}[ 
      title ={},
		  width=.45\columnwidth,
		  height=.30\columnwidth,
      xlabel=$x$,
      cycle list name=dens,
       grid,
      ] 
      \addplot table[x=x,y=y] { 
      x  y 
      0.002794244685024697  0.41282111409741357 
      0.05100760643872949  0.43273563102311124 
      0.09480209234283196  0.4499388862925675 
      0.1501686309843965  0.46442981722169113 
      0.18087585814720453  0.46751174004078616 
      0.23173469853420214  0.46509792119086385 
      0.2745770324729436  0.4575192082719382 
      0.3158223173476388  0.44687290961814663 
      0.3549917450287107  0.4347479350273885 
      0.41262874921755405  0.4148923269254528 
      0.4526622170158059  0.4005793314919085 
      0.49302141005928946  0.3860256487355689 
      0.5351792529475141  0.3711095680692716 
      0.5772883172002552  0.3564109132046836 
      0.6294213856127225  0.3390334389702375 
      0.6798919802041565  0.32304954519088047 
      0.7105774487545969  0.31377636159236 
      0.7487513989364474  0.30239523594559947 
      0.8132613515826783  0.2844652979188286 
      0.8391568989624817  0.27758144980765465 
      0.89657791433721  0.26321279288179117 
      0.927993701144748  0.25567810473772595 
      0.969770829312581  0.2460244262962898 
      1.016789912831574  0.2356645048824472 
      1.070934381303925  0.22456853329996515 
      1.1042317533363135  0.2179480673514368 
      1.1555408610469893  0.20827369763523002 
      1.1925359860149876  0.2015562900856097 
      1.2456623655221826  0.1923843554353138 
      1.2929245026880636  0.18466436198560907 
      1.339549921931323  0.1774130528931131 
      1.3833389442024058  0.17088008409328664 
      1.4265571720688355  0.16472537702868378 
      1.4679794058873588  0.15905284195422323 
      1.5070853102355601  0.15392596567063785 
      1.541747168195297  0.14955321027419316 
      1.5849167630729935  0.14430279767023022 
      1.6485360778797198  0.136940727625966 
      1.6935106686493893  0.13205168883677923 
      1.7265332771956272  0.12853110011655644 
      1.7729181958360098  0.12385436344078886 
      1.8188156111989284  0.11946268370870343 
      1.8480443137230393  0.11673487404654535 
      1.912715583212404  0.11092988576537006 
      1.9446093502311843  0.10821047703416542 
      1.992123678438767  0.10430887574466767 
      2.0398972000471822  0.10057069129559998 
      2.08264753147836  0.09736248877472524 
      2.1149911053693273  0.09501210513696762 
      2.1678319475979184  0.09134568564384005 
 }; 
       \addlegendentry{density}
  \end{axis}
  \end{tikzpicture}
\hfill
   \begin{tikzpicture} 
    \begin{axis}[ 
      width = 0.5\textwidth,
	    height = 5cm,
      xlabel=$\log n$,
      ylabel=$\log{\rm IV}$,
       grid,
       legend style={at={(1.25,1.1)},anchor=north east},   
       cycle list name=comparison,
      ] 
      \addplot table[x=logN,y=logIV] { 
      logN  logIV 
      14.0  -19.853407349363017 
      15.0  -20.837730535905706 
      16.0  -21.835504415895194 
      17.0  -22.843850664973367 
      18.0  -23.842744090719222 
      19.0  -24.842239203877426 
 }; 
      \addlegendentry{MC-CDE}
%
%
      \addplot table[x=logN,y=logIV] { 
      logN  logIV 
      14.0  -10.855182045494894 
      15.0  -11.838606936573392 
      16.0  -12.83926039468427 
      17.0  -13.83398707870161 
      18.0  -14.843645560130184 
      19.0  -15.839419424762527 
 }; 

      \addlegendentry{MC-GLRDE}
      \addplot table[x=logN,y=logIV] { 
      logN  logIV 
      14.0  -27.29067223962574 
      15.0  -28.07666259760639 
      16.0  -29.294888069803527 
      17.0  -30.021630892383126 
      18.0  -31.12253797408318 
      19.0  -32.25043212022527 
 }; 
      \addlegendentry{Lat+s-CDE}
      \addplot table[x=logN,y=logIV] {
      logN  logIV 
       14.0	 -19.654037
       15.0	 -20.350068
       16.0	 -21.478259
       17.0	 -22.856442
       18.0	 -23.549005
       19.0  -24.636864 
 }; 
      \addlegendentry{Lat+s-GLRDE}
    \end{axis}
  \end{tikzpicture}
  \caption{Estimated density (left) and $\log\IV$ as a function of $\log n$ (right) 
    for the single queue over a finite-horizon.}
  \label{fig:queue-finite-horizon}
\end{figure}
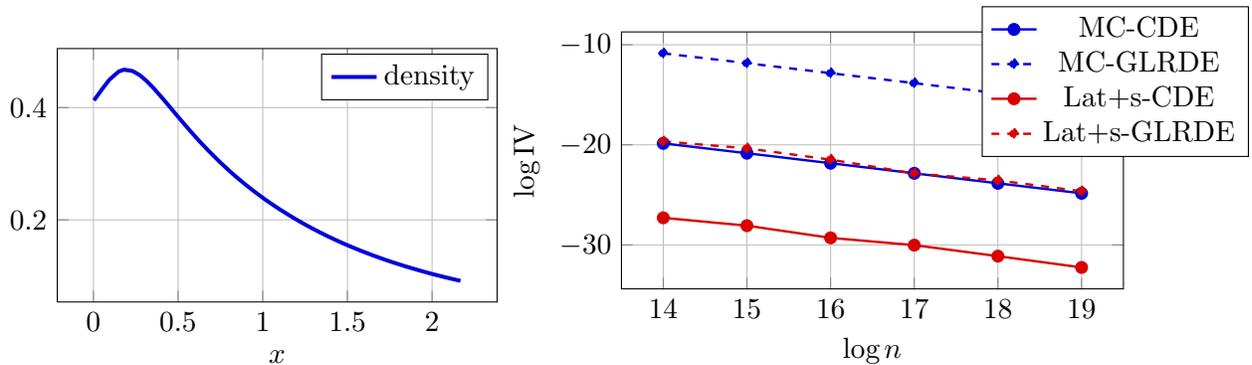

\begin{figure}[htb]
  \begin{tikzpicture} 
    \begin{axis}[ 
      title ={},
		width=.45\columnwidth,
		height=.35\columnwidth,
      xlabel=$x$,
      cycle list name=dens,
       grid,
      ] 
      \addplot table[x=x,y=y] { 
      x  y 
      0.005906017175165837  0.39742752587153185 
      0.10781153179095096  0.43589591609120293 
      0.2003771497246221  0.44892750275994364 
      0.3174018791261108  0.42720191770134747 
      0.3823057910838641  0.4066324644046126 
      0.4898028855382  0.3703009946763053 
      0.5803560004541763  0.33985152408428554 
      0.6675335343938729  0.31233946371649896 
      0.7503234610834112  0.28782182578232635 
      0.872147129028012  0.2557573228592747 
      0.9567633223288625  0.23534128033385596 
      1.042067980352589  0.21661513772725566 
      1.1311743300936095  0.19963333655108015 
      1.220177579536903  0.1836563945861569 
      1.3303679286814363  0.16533184260031605 
      1.43704441270424  0.15001706772845466 
      1.5019023348676708  0.14143204714151175 
      1.5825881841156728  0.13088798567881774 
      1.7189387658452064  0.11504780412967734 
      1.7736725364434272  0.1088141861309927 
      1.8950396825763756  0.09725014657781669 
      1.9614412319650356  0.09211788155390024 
      2.049742889228864  0.08526292123636937 
      2.1491241339394636  0.07823271660814425 
      2.263565851392387  0.07137039105683644 
      2.3339443877335717  0.06755916979791599 
      2.442393183576591  0.06184584019973621 
      2.520587424986224  0.05748247173296238 
      2.6328772725809766  0.051925064249434244 
      2.7327722443179527  0.04785302190670782 
      2.8313214259002955  0.04402821664601987 
      2.923875495700539  0.04054204361942753 
      3.015223113690948  0.03810055665773806 
      3.1027746533528267  0.03532923654129244 
      3.1854303148160703  0.03314924537992229 
      3.2586928782309683  0.030895804771667913 
      3.3499377037679183  0.02926389232797021 
      3.484405800973044  0.0265271200409409 
      3.579465731463482  0.02448879174596624 
      3.6492635177089396  0.023049062342575997 
      3.7473043684715663  0.021253165221502675 
      3.844314814579553  0.019502573188140145 
      3.906093663096424  0.018393221307815982 
      4.042785209971672  0.016385551548199307 
      4.110197035715912  0.015529812949185463 
      4.210625047609212  0.01430736523687535 
      4.311600900099727  0.013176668284291716 
      4.401959555170171  0.012294833956339717 
      4.470322109076078  0.011505996756587677 
      4.5820084346956005  0.01078568327727437 
 }; 
       \addlegendentry{density}
  \end{axis}
  \end{tikzpicture}
\hfill		
		\begin{tikzpicture} 
    \begin{axis}[ 
		width=.45\columnwidth,
		height=.35\columnwidth,
      xlabel=$\log n$,
      ylabel=$\log{\rm IV}$,
      cycle list name=comparison,
	  legend style={at={(1.45,1.)},anchor=north east},     
       grid,
	] 
      \addplot table[x=logN,y=logIV] { 
      logN  logIV 
      14.0  -14.90092233824423 
      15.0  -15.922674757355777 
      16.0  -16.91721009636318 
      17.0  -17.877637604940823 
      18.0  -18.87594158390892 
      19.0  -19.882423945433235 
 };
      \addlegendentry{MC-CDE}
%
      \addplot table[x=logN,y=logIV] { 
      logN  logIV 
      14.0  -6.620403056650534 
      15.0  -7.563311157521403 
      16.0  -8.471188538507484 
      17.0  -9.459099061707633 
      18.0  -10.512135806985476 
      19.0  -11.478168194523814 
 }; 
      \addlegendentry{MC-GLRDE}
      \addplot table[x=logN,y=logIV] { 
      logN  logIV 
      14.0  -22.392693569981542 
      15.0  -23.382865497835372 
      16.0  -24.7233343106602 
      17.0  -25.55620435940835 
      18.0  -26.517461119406576 
      19.0  -27.647943633942436 
 }; 
      \addlegendentry{Lat+s-CDE}
      \addplot table[x=logN,y=logIV] { 
      logN  logIV 
      14.0  -14.136963821198073 
      15.0  -15.036836033546011 
      16.0  -16.562985946524233 
      17.0  -17.554005390299235 
      18.0  -18.722183207944685 
      19.0  -20.13907459572688 
 }; 
      \addlegendentry{Lat+s-GLRDE}
    \end{axis}
  \end{tikzpicture}
  \caption{Estimated density (left) and $\log\IV$ as a function of $\log n$ (right) 
    for the single queue in steady-state.}
  \label{fig:queue-steady-state}
\end{figure}
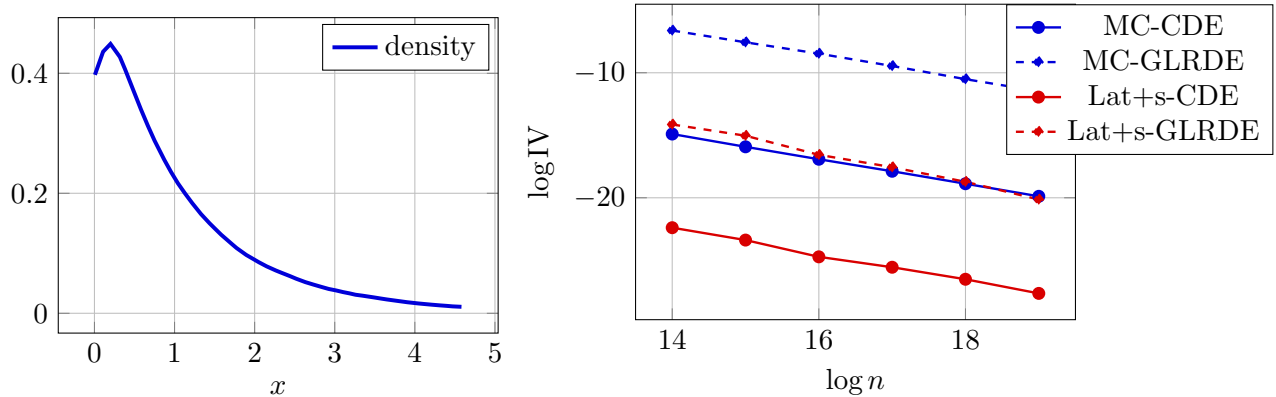


\bigskip
\section{Some additional figures}
\label{sec:more-figures}

Figure~\ref{fig:asianBridge} shows the estimated density in $[a,b]$ (left panel)
and the IV as a function of $n$ in log-log scale for the bridge CDE in 
Example~\ref{sec:function-multinormal}.

Figure~\ref{fig:queue-finite-horizon} shows the estimated density in $[a,b]$ (left panel)
and the IV as a function of $n$ in log-log scale
for the finite-horizon queueing system in Example~\ref{sec:single-queue}.
Figure~\ref{fig:queue-steady-state} does the same for the infinite-horizon case.


\end{APPENDICES}
